\documentclass[a4paper, oneside]{amsart}
\usepackage{mathrsfs}
\usepackage{amsfonts}
\usepackage{amssymb}
\usepackage{amsxtra}
\usepackage{dsfont}

\usepackage[english,polish]{babel}

\newcommand{\pl}[1]{\foreignlanguage{polish}{#1}}

\usepackage{color}

\theoremstyle{plain}
\newtheorem{theorem}{Theorem}[section]
\newtheorem{proposition}{Proposition}[section]

\newtheorem{lemma}{Lemma}[section]
\theoremstyle{definition}

\theoremstyle{remark}
\newtheorem{remark}{Remark}[section]

\numberwithin{equation}{section}

\theoremstyle{plain}
\newcounter{thm}

\newtheorem{main_theorem}[thm]{Theorem}

\newcommand{\RR}{\mathbb{R}}
\newcommand{\BB}{\mathbb{B}}
\newcommand{\ZZ}{\mathbb{Z}}
\newcommand{\TT}{\mathbb{T}}
\newcommand{\CC}{\mathbb{C}}
\newcommand{\NN}{\mathbb{N}}
\newcommand{\QQ}{\mathbb{Q}}

\newcommand{\EE}{\mathbb{E}}
\newcommand{\DD}{\mathbb{D}}

\newcommand{\boldB}{\mathbf{B}}
\newcommand{\boldA}{\mathbf{A}}

\newcommand{\calC}{\mathcal{C}}
\newcommand{\calP}{\mathcal{P}}
\newcommand{\calQ}{\mathcal{Q}}
\newcommand{\calF}{\mathcal{F}}

\newcommand{\calM}{\mathcal{M}}

\newcommand{\calI}{\mathcal{I}}

\newcommand{\calT}{\mathcal{T}}

\newcommand{\seq}[2]{{#1}: {#2}}
\newcommand{\ind}[1]{{\mathds{1}_{{#1}}}}

\newcommand{\supp}{\operatorname{supp}}
\newcommand{\diam}{\operatorname{diam}}

\newcommand{\var}[1]{V_r{#1}}

\newcommand{\bvar}[1]{\mathcal{V}_r{#1}}

\renewcommand{\atop}[2]{\substack{{#1}\\{#2}}}
\newcommand{\norm}[1]{{\left\lvert #1 \right\rvert}}
\newcommand{\sprod}[2] {{#1 \cdot #2}}

\newcommand{\abs}[1]{{\lvert {#1} \rvert}}
\newcommand{\sabs}[1]{{\left\lvert {#1} \right\rvert}}
\newcommand{\vnorm}[1]{{\left\lVert {#1} \right\rVert}}

\newcommand{\vrho}{\varrho}

\title[Discrete operators of Radon type]
{$\ell^p\big(\ZZ^d\big)$-estimates for discrete operators of Radon
  type: Variational estimates  }

\author{Mariusz Mirek}
\address{Mariusz Mirek \\
	Universit\"{a}t Bonn \\
	Mathematical Institute\\
	Endenicher Allee 60\\
	D-53115 Bonn \\
	Germany}
 \email{mirek@math.uni-bonn.de}

\author{Elias M. Stein}
\address{
	Elias M. Stein\\
	Department of Mathematics\\
	Princeton University\\
	Princeton\\
	NJ 08544-100 USA}
\email{stein@math.princeton.edu}

\author{Bartosz Trojan}
\address{
	Bartosz Trojan\\
	Instytut Matematyczny\\
	Uniwersytet \pl{Wroc{\lll}awski}\\
	Plac Grun\-waldzki 2/4\\
	50-384 \pl{Wroc{\lll}aw}\\
	Poland}
\email{trojan@math.uni.wroc.pl}

\begin{document}
\selectlanguage{english}

\begin{abstract}
  We prove $\ell^p\big(\mathbb Z^d\big)$ bounds for $p\in(1, \infty)$, of $r$-variations
  $r\in(2, \infty)$, for discrete averaging operators and truncated singular
  integrals of Radon type. We shall present a new powerful method which
  allows us to deal with these operators in a unified way and obtain
  the range of parameters of $p$ and $r$ which coincide with the
  ranges of their continuous counterparts.
\end{abstract}

\maketitle

\section{Introduction}\label{sec:1}
In this paper we will be concerned with estimates of $r$-variations for discrete operators of averaging and singular Radon type, and their application to ergodic theory. The $r$-variational estimates for the continuous versions of these operators will be treated as well. 

Let
$$
\calP=(\calP_1,\ldots, \calP_{d_0}): \ZZ^{k} \rightarrow \ZZ^{d_0}
$$
be a polynomial mapping where for each $j \in \{1, \ldots, d_0\}$ the
function $\calP_j:\ZZ^{k} \rightarrow \ZZ$ is an integer valued
polynomial of $k$ variables with $\calP_j(0) = 0$.  
Define, for a finitely
supported function $f: \ZZ^{d_0} \rightarrow \CC$, the Radon averages
\begin{align}
	\label{eq:2}
	 M_N^{\calP} f(x)
	=|\BB_N|^{-1} \sum_{y\in\BB_N}
	f\big(x-\calP(y)\big)
\end{align}
where  $\BB_t=\{x\in\ZZ^k: |x|\le t\}$ and $t>0$.  We will be also interested
in discrete truncated singular integrals.

Assume that $K \in
\calC^1\big(\RR^k \setminus \{0\}\big)$ is a Calder\'{o}n--Zygmund
kernel satisfying the differential inequality
\[
	\norm{y}^k \abs{K(y)} + \norm{y}^{k+1} \norm{\nabla K(y)} \leq 1
\]
for all $y \in \RR^k$ with $\norm{y} \geq 1$. We also impose the following cancellation
condition 
\begin{align}
\label{eq:24}
	\int_{B_{\lambda_2} \setminus B_{\lambda_1}} K(y) {\: \rm d} y = 0
\end{align} 
for every $0<\lambda_1\le \lambda_2$ where $B_{\lambda}$ is the
Euclidean ball in $\RR^k$ centered at the origin with
radius $\lambda>0$. Define, for a finitely supported
function $f: \ZZ^{d_0} \rightarrow \CC$, the truncated singular Radon transforms 
\begin{align}
  \label{eq:190}
T_N^{\calP} f(x)
=
\sum_{y\in\BB_N\setminus\{0\}} f\big(x - \calP(y)\big) K(y).
\end{align}

The basic aim of this paper is to strengthen the
$\ell^p\big(\ZZ^{d_0}\big)$ boundedness, $p\in(1, \infty)$, of maximal functions
corresponding to operators \eqref{eq:2} and \eqref{eq:190}, which have been
recently proven in \cite{mst1}, and provide sharp $r$-variational
bounds in the full range of exponents.

Recall that for any
$r\in[1, \infty)$ the $r$-variational seminorm $V_r$ of a sequence $\big(a_n(x):
n\in \NN\big)$ of complex-valued functions is defined by
$$
V_r\big(a_n(x): n\in \NN\big)
=\sup_{\atop{k_0<\ldots <k_J}{k_j\in \NN}}
\Big(\sum_{j=1}^J|a_{k_j}(x)-a_{k_{j-1}}(x)|^r\Big)^{1/r}.
$$ 
The main results of this article are the following theorems.
\begin{main_theorem}
  \label{thm:1}
	For every $p\in(1, \infty)$  and $r\in(2, \infty)$ there is $C_{p, r} > 0$ such that for all
	$f \in \ell^p\big(\ZZ^{d_0}\big)$
        \begin{align}
          \label{eq:3}
          	\big\lVert
	V_r\big( M_N^\calP f: N\in\NN\big)
	\big\rVert_{\ell^p}\le
	C_{p, r}\|f\|_{\ell^p}.
        \end{align}
	Moreover, the constant $C_{p, r}\le C_p\frac{r}{r-2}$ for some
        $C_p>0$ which is independent of the coefficients of the polynomial
        mapping $\calP$.
\end{main_theorem}
We also obtain the corresponing theorem for the truncated singular Radon transforms.
\begin{main_theorem}
\label{thm:4}
	For every $p\in(1, \infty)$  and $r\in(2, \infty)$ there is $C_{p, r} > 0$ such that for all
	$f \in \ell^p\big(\ZZ^{d_0}\big)$
        \begin{align}
       \label{eq:59}
          	\big\lVert
	V_r\big(  T_N^\calP f: N\in\NN\big)
	\big\rVert_{\ell^p}\le
	C_{p, r}\|f\|_{\ell^p}.
        \end{align}
	Moreover, the constant $C_{p, r}\le C_p\frac{r}{r-2}$ for some
        $C_p>0$ which is independent of the coefficients of the polynomial
        mapping $\calP$.
\end{main_theorem}

Theorem \ref{thm:1} and Theorem \ref{thm:4} have ergodic theoretical
interpretations. More precisely, let $(X, \mathcal{B}, \mu)$ be a
$\sigma$-finite measure space with a family of invertible commuting
and measure preserving transformations $S_1, S_2,\ldots,S_{d_0}$. Let
\begin{align}
\label{eq:191}	
\mathcal A^\calP_N f(x)
	= N^{-k}\sum_{y \in \BB_N} 
	f\big(S_1^{\calP_1(y)} S_2^{\calP_2(y)} \cdot \ldots \cdot S_{d_0}^{\calP_{d_0}(y)} x\big)
\end{align}
and
\begin{align}
\label{eq:192}
	\mathcal H^\calP_N f(x)
	= \sum_{y \in \BB_N \setminus \{0\}} 
	f\big(S_1^{\calP_1(y)} S_2^{\calP_2(y)} \cdot \ldots \cdot S_{d_0}^{\calP_{d_0}(y)} x\big)
	K(y).
\end{align}
Specifying a suitable measure space $(X, \mathcal{B}, \mu)$ and a
family of measure preserving transformations
we immediately see that $\mathcal A^\calP_N$ and $\mathcal H^\calP_N$
coincide with $M^\calP_N$ and $T^\calP_N$ respectively. 
Indeed, it suffices to take $X=\ZZ^{d_0}$, $\mathcal B=\mathbf P\big(\ZZ^{d_0}\big)$ 
$\sigma$-algebra of all subsets of $\ZZ^{d_0}$, $\mu=|\; \cdot\; |$ 
to be the counting measure on $\ZZ^{d_0}$ and $S_j^y:\ZZ^{d_0} \rightarrow \ZZ^{d_0}$ the shift operator
acting of $j$-th coordinate, i.e. $S_j^y(x_1,\ldots,
x_{d_0})=(x_1,\ldots,x_j-y,\ldots, x_{d_0})$ for all $j=1, 2,\ldots,
d_0$ and $y\in\ZZ$.
\begin{main_theorem}
\label{thm:24}
Let $\mathcal S_N^{\calP}$ be the operator given either by
\eqref{eq:191} or by \eqref{eq:192} and assume that $p\in(1, \infty)$ and
$r\in(2, \infty)$. Then there is $C_{p, r} > 0$ such that for all
	$f \in L^p(X, \mu)$
        \begin{align}
\label{eq:193}
          	\big\lVert
	V_r\big(  \mathcal S_N^\calP f: N\in\NN\big)
	\big\rVert_{L^p}\le
	C_{p, r}\|f\|_{L^p}.
        \end{align}
        In particular, \eqref{eq:193} implies that for every $f \in
        L^p(X, \mu)$ there exists $f^*\in L^p(X, \mu)$ such that
	\begin{align*}
		\lim_{N\to\infty} \mathcal S_{N}^{\calP}f(x)=f^*(x)
	\end{align*}
	$\mu$-almost everywhere on $X$.
\end{main_theorem}
The estimate \eqref{eq:193} from Theorem \ref{thm:24} can be deduced
from inequality \eqref{eq:3} or \eqref{eq:59} by appealing to the
Calder\'on transference principle. Furthermore,  Theorem \ref{thm:24}
with $\mathcal S_N^{\calP}=\mathcal H_N^{\calP}$ can be thought as an
extension of  Cotlar's ergodic theorem (see \cite{cot}), which states that
for every $\sigma$-finite measure space $(X, \mathcal{B},
\mu)$ with an invertible and a measure preserving transformation $S$ the limit
\[
\lim_{N\to\infty}\sum_{0<|n|\le N} 
	\frac{f\big(S^n x\big)}{n}
\]
exists $\mu$-almost everywhere on $X$ for every $f\in L^p(X, \mu)$
with $p\in(1, \infty)$.

The classical strategy for handling pointwise convergence
problems requires $L^p(X, \mu)$ boundedness for the corresponding maximal function,
reducing the matters to proving pointwise convergence for a dense
class of $L^p(X, \mu)$ functions.  However, establishing pointwise
convergence on a dense class can be a quite challenging problem.  This is the case for
Bourgain's averaging operator along the squares.  Fortunately in \cite{bou}, he was
able to circumvent this issue  for the operators $\mathcal
A_N^{\calP}$, in the one dimensional case $k=d_0=1$, by controlling their
oscillation seminorm.  Given a lacunary sequence $(n_j: j\in\NN)$, the
oscillation seminorm for a sequence $\big(a_n: n\in\NN\big)$ of
complex numbers is defined by
\[
         O_J\big(a_n: n\in\NN\big)
        = 
        \Big(
        \sum_{j=1}^J\sup_{n_j < n \leq n_{j+1}}
        \big|a_n-a_{n_j}\big|^2\Big)^{1/2}.
\]
Bourgain deduced pointwise convergence on $L^2(X, \mu)$ for the
operators $\mathcal A_N^{\calP}$ by proving, roughly, that there are
constants $C>0$ and $c<1/2$ such that for all $J\in\NN$
\[
\big\|O_J\big(\mathcal A_{N}^{\calP}f: N\in\NN\big)\big\|_{L^2} 
\le CJ^{c} \|f\|_{L^2}.
\]

Variational estimates have been the subject of many papers, see
\cite{jkrw, jsw, k, mt3, zk} and the references therein.  Our motivation to
study $r$-variational seminorms is threefold. Firstly,
for any sequence of functions $\big(a_n(x): n\in \NN\big)$, if for some $1 \leq r < \infty$ 
\[
V_r\big(a_n(x): n\in \NN\big)<\infty
\]
then the limit $\lim_{n\to\infty}a_n(x)$ exists. Secondly, $V_r$'s control the supremum norm.
Indeed, for any $n_0\in \NN$ we have the pointwise estimate
\[
\sup_{n\in \NN}|a_n(x)|\le |a_{n_0}(x)| + 2V_r\big(a_n(x): n\in \NN\big).
\]
Furthermore, for any $2\le r< \infty $ by H\"older's inequality we have
\[
 O_{J}\big(a_n(x): n \in \NN\big)\le J^{1/2-1/r} V_r\big(a_n(x): n \in \NN\big).
\]

The variational estimates for Bourgain's averaging
operator \eqref{eq:2} with $k=d_0=1$, have recently been extensively
studied, while only partial results were obtained. Namely,
Krause \cite{k} showed inequality \eqref{eq:3} for $p\in(1, \infty)$ and
$r>\max\{p, p'\}$. Zorin-Kranich \cite{zk} showed \eqref{eq:3} with
$r\in(2, \infty)$ and $p\in(1, \infty)$ in some vicinity of $2$, i.e.  $|1/p-1/2|<1/(2(d'+1))$,
where $d'$ is the degree of the polynomial. Their proofs were based on
variational estimates of the famous Bourgain's logarithmic lemma provided
by Nazarov, Oberlin and Thiele in \cite{not}, see also \cite{K2} for
some improvements. That was the main building block in their
arguments. Although the logarithmic lemma gives very nice $\ell^2(\ZZ)$
results, it is generally very inefficient for $\ell^p(\ZZ)$. The reason, loosely speaking, is that it produces for
$p\not=2$ a polynomial growth in norm unlike the acceptable
logarithmic growth which one has for $p=2$. Therefore in this paper we
introduce a different flexible approach based on a direct analysis of the
multipliers associated with operators \eqref{eq:2} and
\eqref{eq:190}, and instead of Bourgain's logarithmic lemma we will
apply a simple numerical inequality, see Lemma \ref{lem:6}, which
turns out to be a more appropriate tool  in  these  problems with
arithmetic flavor. This lemma is a variant of the crucial Lemma 2.2 that we used in \cite{mst1}.

The proof of Theorem \ref{thm:1} and Theorem \ref{thm:4}, in view of
inequality \eqref{eq:32}, will be based on separate estimates for long
and short variational seminorms of the operators
$M_{N}^{\calP}$ and $ T_{N}^{\calP}$ respectively. We now describe the
key points of our method in the case of averaging operator
\eqref{eq:2}. Assume, for simplicity, that $k=1$ and
$\calP(x)=(x^d,\ldots, x)$ is a moment curve for some $d_0=d\ge2$.
Let $m_{N}$ be the multiplier associated with $M_{N}^{\calP}$,
i.e. $\calF^{-1}\big(m_{N}\hat{f}\big)=M_{N}^{\calP}f$.  

The estimates of long variations will be very much in spirit of the
estimates of maximal functions associated with $M_{N}^{\calP}$ as we gave in
\cite{mst1}. For this purpose as in \cite{mst1} we introduce an appropriate partition of
unity which permits us to identify asymptotic or highly oscillatory
behaviour of $m_{2^n}$, corresponding respectively to the ``major" and ``minor" arcs. This distinction is based on the
ideas of the circle method of Hardy and Littlewood.  More precisely,
let $\eta$ be a smooth cut-off function with a small support, fix $l
\in \NN$ and  for each $n\in\NN$ define projections
\[
	\Xi_{n}(\xi)
	=\sum_{a/q \in\mathscr{U}_{n^l}}
	 \eta\big(\mathcal E_n^{-1}(\xi - a/q)\big)
\]
where  $\mathcal E_n$
is a diagonal $d\times d$  matrix with positive entries
$\big(\varepsilon_{j}: 1\le j\le d\big)$ such that
$\varepsilon_{j}\le e^{-n^{1/5}}$   and
\[
\mathscr{U}_{n^l}=\big\{a/q \in \TT^d\cap\QQ^d : a=(a_1, \ldots, a_d) \in
\NN^d_q \text{ and } \mathrm{gcd}(a_1, \ldots, a_d, q)=1 \text{ and }
q\in  P_{n^l}\big\}
\]
for some family $ P_{n^l}$ such that $\NN_{n^l}\subseteq
P_{n^l}\subseteq \NN_{e^{n^{1/10}}}$, we refer to the last subsection
of Section \ref{sec:3}
for more detailed definitions. The projections $\Xi_n$ will be
critical in the further analysis, since 
\[
V_r^L\big(M_{N}^{\calP}f: N\in\NN\big)\le
V_r\big(\calF^{-1}\big(m_{2^n}\Xi_n\hat{f}\big): n\in\NN_0\big)+
V_r\big(\calF^{-1}\big(m_{2^n}(1-\Xi_n)\hat{f}\big): n\in\NN_0\big)
\]
where $m_{2^n}(\xi)\Xi_n(\xi)$ corresponds to the asymptotic
behaviour of $m_{2^n}(\xi)$, whereas $m_{2^n}(\xi)(1-\Xi_n(\xi))$
localizes the highly oscillatory part. For  the last piece we can
prove that
\[
\big\|V_r\big(\calF^{-1}\big(m_{2^n}(1-\Xi_{n})\hat{f}\big): n\in\NN_0\big)\big\|_{\ell^p}\le C_{
  p, r}\|f\|_{\ell^p}.
\]
These bounds can be deduced from a variant of Weyl's inequality with
logarithmic decay, see  Theorem \ref{thm:3} or
\cite{mst1}, and the following inequality for $p\in(1, \infty)$, that goes back to ideas of Ionescu and Wainger,
\begin{align}
  \label{eq:202}
\big\|\calF^{-1}\big(\Xi_{n}\hat{f}\big)\big\|_{\ell^p}\le C_{l,
  p} \log(n+2) \|f\|_{\ell^p}  
\end{align}
 see Theorem \ref{th:3} or \cite{mst1} and \cite{iw} for more
 detailed expositions. For the proof of the estimate
\[
\big\|V_r\big(\calF^{-1}\big(m_{2^n}\Xi_{n}\hat{f}\big): n\in\NN_0\big)\big\|_{\ell^p}\le C_{
  p, r}\|f\|_{\ell^p},
\]
we  show that 
\[
m_{2^n}(\xi)\Xi_n(\xi)\simeq\sum_{s\ge 0}m_{2^n}^s(\xi)
\]
where
\begin{align}
  \label{eq:203}
m_{2^n}^s(\xi)=\sum_{a/q \in\mathscr{U}_{(s+1)^l}\setminus\mathscr{U}_{s^l}}G(a/q)\Phi_{2^n}(\xi-a/q)
	 \eta\big(\mathcal E_s^{-1}(\xi - a/q)\big),
\end{align}
with $G(a/q)$ being the Gaussian sum and $\Phi_{2^n}$ being the continuous version of $m_{2^n}$, see at the
beginning of Section \ref{sec:4} for relevant definitions.  Then the matters are reduced to proving 
that for each $s\ge0$ we have
\begin{align}
  \label{eq:204}
\big\|V_r\big(\calF^{-1}\big(m_{2^n}^s\Xi_{n}\hat{f}\big): n\in\NN_0\big)\big\|_{\ell^p}\le  C_{
  p}(s+1)^{-2}\|f\|_{\ell^p}.
\end{align}
Firstly, we prove \eqref{eq:204} for $p=2$ with bound $C_r
(s+1)^{-\delta l+1}\|f\|_{\ell^2}$, where $\delta>0$ is an exponent from
the bound for the Gaussian sums $|G(a/q)|\le Cq^{-\delta}$, and
$l\in\NN$ an arbitrary integer.  Secondly, for general $p\not=2$ we
obtain much worse bound $C_{l, p, r}s \log(s+2) \|f\|_{\ell^p}$. Now
interpolating the last bound with much better for $p=2$ we get the
claim from \eqref{eq:204}, since $l\in\NN$ can be arbitrarily
large. To achieve both bounds we partition the $V_r$ in \eqref{eq:204}
into two pieces $n\le 2^{\kappa_s}$ and $n>2^{\kappa_s}$ for some
integer $1<\kappa_s\le Cs$. The case for large scales $n>2^{\kappa_s}$
follows by invoking  the transference principle which allows us to
control discrete $\| \cdot \|_{\ell^p}$ norm of $r$-variations
associated with multipliers from \eqref{eq:203} by the continuous $\|
\cdot \|_{L^p}$ norm of $r$-variations closely related with the
multiplier $\Phi_{2^n}$, which is \emph{a priori} bounded on
$\RR^d$. This is the place, and only place, where we are restricted to
$r\in(2, \infty)$, and then obtain the growth of the constant $C_{p,
  r}\le C_p\frac{r}{r-2}$ in Theorem \ref{thm:1}.  The reason lies in
application L\'epingle's inequality, (see Theorem \ref{thm:22} and
Theorem \ref{thm:23} in the Appendix) to bound $L^p\big(\RR^d\big)$
norm.  The case of small scales $n\le 2^{\kappa_s}$ has different
nature and some new ideas came up. An invaluable tool which surmounted
complications occurring in \cite{k} and \cite{zk} is a simple
numerical inequality from Lemma \ref{lem:6} yielding
\begin{align}
  \label{eq:205}
  		V_r\big(\calF^{-1}\big(m_{2^n}^s\hat{f}\big): 0 \leq n \leq 2^{\kappa_s}\big)
		\leq
		\sqrt{2}
		\sum_{i = 0}^{\kappa_s}
		\Big(
		\sum_{j = 0}^{2^{\kappa_s-i}-1}
		\big|\calF^{-1}\big(m_{2^{(j+1)2^i}}^s\hat{f}\big) - \calF^{-1}\big(m_{2^{j2^i}}^s\hat{f}\big)\big|^2
		\Big)^{1/2}.
\end{align}
Applying now Theorem \ref{th:3} (see also \cite{mst1}) we
shall show that $\ell^p\big(\ZZ^d\big)$ norm of the inner square
function on the right-hand side in \eqref{eq:205} is bounded by
$C_p \log(s+2)\|f\|_{\ell^p}$ for each $0\le i\le \kappa_s$. Consequently,
we get the desired bound since  
there are $\kappa_s+1$ elements. This illustrates roughly the scheme for long $r$-variations.

In order to attack short variations we will again exploit the
partition of unity introduced above and obtain
\begin{multline}
\label{eq:197}
  V_r^S\big(M_{N}^{\calP}f: N\in\NN\big)\le
	\Big(\sum_{n \ge 0}
        V_2\big((M_{N}^{\calP}-M_{2^n}^{\calP})\calF^{-1}\big(\Xi_n\hat{f}\big): N\in
        [2^n, 2^{n+1})\big)^2\Big)^{1/2}\\
	+
	\big(
	\sum_{n \ge 0} V_2\big(\calF^{-1}\big((M_{N}^{\calP}-M_{2^n}^{\calP})(1-\Xi_n)\hat{f}\big): 
	N\in [2^n, 2^{n+1})\big)^2\Big)^{1/2}.
\end{multline}
The last sum corresponds to the  highly oscillatory behaviour of the
multiplier $m_N$. Therefore, invoking  inequality \eqref{eq:39},
Weyl's inequality in Theorem \ref{thm:3} and \eqref{eq:202}, we are
able to prove that the last term in \eqref{eq:197} is bounded on
$\ell^p\big(\ZZ^d\big)$, see Section \ref{sec:5}. 

To bound the first term in \eqref{eq:197} we introduce a tool
reminiscent of the Littlewood-Paley theory. Now as opposed to the
continous theory, for discrete operators there is no known analogue of
the square functions of Littlewood-Paley that give us decisive control
of the operators in question. However as a start in this direction we
consider in Section \ref{sec:5}, (see Theorem \ref{thm:30}), the following family of projections:
\begin{align}
\label{eq:40}
  \Delta_{n, s}^j(\xi)=\sum_{a/q\in\mathscr
    U_{(s+1)^l}\setminus\mathscr
    U_{s^l}}\big(\eta\big(\mathcal E_{n+j}(\xi-a/q)\big)-\eta\big(\mathcal E_{n+j+1}(\xi-a/q)\big)\big)\eta\big(\mathcal E_{s}(\xi-a/q)\big).
\end{align}
and using Theorem \ref{th:3} we will be able to show that for each $p\in(1, \infty)$ there is a constant
$C>0$ such that 
 \begin{align}
   \label{eq:46}
   \Big\|\Big(\sum_{n\in\ZZ}\big|\mathcal
  F^{-1}\big(\Delta_{n,
    s}^j\hat{f}\big)\big|^2\Big)^{1/2}\Big\|_{\ell^p}
	\le 
	C \log(s+2) \|f\|_{\ell^p}.
 \end{align}
uniformly in $j\in\ZZ$. Estimate \eqref{eq:46} can be thought as a
discrete counterpart of Littlewood--Paley inequality and is essential
in our further purposes. Thanks to \eqref{eq:46} we reduce the problem
to showing that
\begin{align}
  \label{eq:176}
  \sum_{s\ge 0}\sum_{j\in\ZZ}
\Big\|\Big(\sum_{n\ge \max\{s, j, -j\}} V_2\big((M_{N}-M_{2^n})
	\calF^{-1}\big(\Delta_{n, s}^j\hat{f}\big): N\in[2^n,
  2^{n+1})\big)^2\Big)^{1/2}\Big\|_{\ell^p}\le C_p\|f\|_{\ell^p}.
\end{align}
In view of inequality \eqref{eq:1} and \eqref{eq:46} we show that the
inner norm in \eqref{eq:176} is dominated for every $p\in(1, \infty)$ by $C_{p}2^{-\varepsilon_p
  |j|}(s+1)^{-2}\|f\|_{\ell^p}$ for some $\varepsilon_p>0$. An
important intermediate step in establishing this bound are the
vector-valued estimates in \cite{mst1}
\[ 
          	\Big\lVert\Big(
	\sum_{t\in\NN}\sup_{N\in\NN}\big|M_N^\calP f_t\big|^2\Big)^{1/2}
	\Big\rVert_{\ell^p}\le
	C_{p}\Big\|\Big(\sum_{t\in\NN}|f_t|^2\Big)^{1/2}\Big\|_{\ell^p}.
\]
        The idea of using vector-valued inequalities allows us to
        overcome many technical difficulties and, as far as we know,
        has not been used in this context before. We also employ
        this idea in the continuous setup and provide a new proof of
        short variations estimates for the operators of Radon type
        in \cite{jsw}. We refer to the Appendix for more details.

        The rest of the paper is organized as follows. In Section
        \ref{sec2} we collected necessary numerical inequalities which
        give relation between $r$-variational seminorms and various
        square functions and even more general objects, see especially Lemma
        \ref{lem:6} and Lemma \ref{lem:8} and inequality
        \eqref{eq:39}.  All these results are important building
        blocks in our approach.  
       Finally, we propose some lifting lemma (see Lemma
        \ref{lem:1}) which allows us to replace any polynomial mapping
        $\calP$ by the canonical polynomial mapping $\calQ$ which has
        all coefficients equal to 1. This guarantees that our further bounds will be
        independent of coefficients of the underlying polynomial
        mapping. 

        In Section \ref{sec:3} we recall further results whose proofs were given in \cite{mst1}. Theorem \ref{thm:3} is a
        variant of multidimensional Weyl's sum estimates with
        logarithmic decay. We also include some basic tools which
        allow us to efficiently compare discrete $\| \cdot
        \|_{\ell^p}$ norms with continuous $\| \cdot \|_{L^p}$
        norms. Finally, Theorem \ref{th:3} is a major step towards
        proving \eqref{eq:202} and \eqref{eq:46}.  This theorem
        originates in Ionescu and Wainger paper \cite{iw} with $(\log
        N)^D$ loss where $D>0$ is a large power.  This is a deep
        result which uses the most sophisticated tools developed to
        date in this area. 

In Section \ref{sec:4} we state Theorem \ref{thm:5} which is the
main result of this Section. However, we omit the proof, since it can
be deduced from the methods of proof of Theorem B from \cite{mst1} by
simply replacing the supremum norm by the long $r$-variational
seminorm; or it can be completed following the scheme of the proof
from Section \ref{sec:6}, which contains long $r$-variational estimates
for the operator $T_N^{\calP}$. We have decided to provide a complete
proof of long $r$-variational estimates for the operator
$T_N^{\calP}$, since there are some subtle differences which did not
occur in \cite{mst1} where the maximal function associated with
$T_N^{\calP}$ was studied, and this would cause some unnecessary confusions.   

In Section \ref{sec:5} and Section \ref{sec:7} we provide detailed
proofs of short variations estimates for operators $M_N^{\calP}$ and
$T_N^{\calP}$ respectively. 

Finally, in the Appendix, which is self-contained, we give a new proof
of strong $r$-variational estimates for the operators of Radon type in
the continuous setting, which are needed above. There are two novel aspects of our proof. The
first concerns a different approach to L\'epingle's inequality for
martingales and is based on Theorem \ref{thm:20} which is a new
ingredient here. The second aspect concerns the estimates
of short variations which now are based to a large extent on
vector-valued estimates for operators of Radon type, which have been
recently obtained in \cite{mst1}. For the reader's convenience we
provide details and describe this method in the
context of dyadic martingales on homogeneous spaces; however, the
methods are general enough to be applicable in a boarder context. 

Finally, let us emphasize the following.

\begin{remark}
  \label{rem:9}
The
  methods of the proof of Theorem \ref{thm:1} and Theorem \ref{thm:4}
allow us to extend inequalities \eqref{eq:3} and \eqref{eq:59} and
establish the following.
\begin{main_theorem}
  \label{thm:100}
	For every $p\in(1, \infty)$  and $r\in(2, \infty)$ there is $C_{p, r} > 0$ such that for all
	$f \in \ell^p\big(\ZZ^{d_0}\big)$
        \begin{align}
\label{eq:37}
          	\big\lVert
	V_r\big( M_t^\calP f: t>0\big)
	\big\rVert_{\ell^p}+\big\lVert
	V_r\big( T_t^\calP f: t>0\big)
	\big\rVert_{\ell^p}\le
	C_{p, r}\|f\|_{\ell^p}.
        \end{align}
	Moreover, the constant $C_{p, r}\le C_p\frac{r}{r-2}$ for some
        $C_p>0$ which is independent of the coefficients of the polynomial
        mapping $\calP$.
\end{main_theorem}
The set of integers in the definition of $r$-variations has been
replaced by the set $(0, \infty)$. The proof of Theorem \ref{thm:100} is presented in Section \ref{sec:9}.
\end{remark}

\begin{remark}
  \label{rem:10}
  The methods of the proof of Theorem \ref{thm:1} and Theorem
  \ref{thm:4} give more general results. Namely, assume that $G$
  is an open bounded convex subset of $\RR^k$ containing the origin,
  and define for any
  $\lambda>0$  $$G_{\lambda}=\{x\in\RR^k: \lambda^{-1}x\in G\}$$  
Define also $\mathbb G_{\lambda}=\{x\in\ZZ^k: \lambda^{-1}x\in G\}$
  for any $\lambda>0$. Then the inequality from Theorem \ref{thm:1} remains
  valid for the averaging operators \eqref{eq:2} defined with $\mathbb G_N$ rather
  than $\mathbb B_N$. Furthermore, if we assume that the cancellation
  condition \eqref{eq:24} holds with $G_{\lambda_2}\setminus
  G_{\lambda_1}$ instead of $B_{\lambda_2}\setminus
  B_{\lambda_1}$ then the conclusion of Theorem \ref{thm:4} remains
  valid for the singular truncated Radon transforms \eqref{eq:190} defined with the summation  taken over
  $\mathbb G_N$ rather than  $\mathbb B_N$. 
\end{remark}

\subsection{Notation}
Throughout the whole article, unless otherwise stated, we will write $A \lesssim B$
($A \gtrsim B$) if there is an absolute constant $C>0$ such that $A\le CB$ ($A\ge CB$).
Moreover, $C > 0$ will stand for a large positive constant whose value may vary from
occurrence to
occurrence. If $A \lesssim B$ and $A\gtrsim B$ hold simultaneously then we will write
$A \simeq B$. We will denote $A \lesssim_{\delta} B$ ($A \gtrsim_{\delta} B$) to
indicate that the constant $C>0$ depends on some $\delta > 0$. Let $\NN_0 = \NN \cup
\{0\}$. For $N \in \NN$
we set
\[
    \NN_N = \big\{1, 2, \ldots, N\big\}, \quad \text{and} \quad
    \ZZ_N = \big\{-N, \ldots, -1, 0, 1, \ldots, N\big\}.
\]
For a vector $x \in \RR^d$ we will use the following norms
\[
    \norm{x}_\infty = \max\{\abs{x_j} : 1 \leq j \leq d\}, \quad \text{and} \quad
    \norm{x} = \Big(\sum_{j = 1}^d \abs{x_j}^2\Big)^{1/2}.
\]
If $\gamma$ is a multi-index from $\NN_0^k$ then $\norm{\gamma} = \gamma_1 + \ldots
+ \gamma_k$. Although we use $|\cdot|$ for the length of a
multi-index $\gamma\in \NN_0^k$
and the Euclidean norm of $x\in\RR^d$, their meaning will be always
clear from the context and it will cause no confusions in
the sequel. 
Finally, let $\mathcal{D}=\{2^n: n\in\NN_0\}$ denote the set of dyadic numbers.
\section{Preliminaries}\label{sec2}

\subsection{Variational norm}
Let $1 \leq r < \infty$. For each sequence $\big(\seq{a_j}{j \in
  A}\big)$ of complex numbers, where $A\subseteq\ZZ$
we define $r$-variational seminorm by
$$
\var{\big(\seq{a_j}{j \in A}\big)} = \sup_{\atop{k_0 < k_1 < \ldots < k_J}{k_j \in A}}
	\Big(\sum_{j = 1}^J \abs{a_{k_j} - a_{k_{j-1}}}^r \Big)^{1/r}
$$
where the supremum is taken over all finite increasing sequences of integers $k_0 < k_1 < \ldots < k_J$.
The function $r \mapsto \var{\big(\seq{a_j}{j \in A}\big)}$ is non-increasing and satisfies
\begin{equation}
	\label{eq:25}
	\sup_{j \in A} \abs{a_j} \leq
	2\var{\big(\seq{a_j}{j \in A}\big)} + \abs{a_{j_0}}
\end{equation}
where $j_0$ is an arbitrary element of $A$. 
For any subset $B \subseteq A$ we have
$$
\var{\big(\seq{a_j}{j \in B}\big)} \leq \var{\big(\seq{a_j}{j \in A}\big)}.
$$
Moreover, if $-\infty\le u < w < v \leq \infty$ then
\begin{equation}
	\label{eq:19}
	\var{\big(\seq{a_j}{u \le j < v}\big)}
	\leq
	2 \sup_{u \le j < v} \abs{a_j}
	+ \var{\big(\seq{a_j}{u \le j < w}\big)}
	+ \var{\big(\seq{a_j}{w \le j < v}\big)}.
\end{equation}
For $r \ge 2$ we also have
\begin{equation}
	\label{eq:11}
	\var{\big(\seq{a_j}{j \in A}\big)}
	\leq 2 \Big(\sum_{j \in A} \abs{a_j}^2\Big)^{1/2}.
\end{equation}
The next lemma will be critical in our further investigations. See also Lemma 2.2 in \cite{mst1}.
\begin{lemma}
	\label{lem:6}
	If $r \in [2, \infty)$ then for any sequence $\big(\seq{a_j}{0 \leq j
          \leq 2^s}\big)$ of complex numbers we have
	\begin{equation}
		\label{eq:62}
		\var{\big(\seq{a_j}{0 \leq j \leq 2^s}\big)}
		\leq
		\sqrt{2}
		\sum_{i = 0}^s
		\Big(
		\sum_{j = 0}^{2^{s-i}-1}
		\sabs{a_{(j+1)2^i} - a_{j 2^i}}^2
		\Big)^{1/2}.
	\end{equation}
\end{lemma}
\begin{proof}
	Let us observe that any interval $[m, n)$ for $m, n \in \NN$ such that $0 \leq m < n \leq 2^s$,
	is a finite disjoint union of dyadic subintervals, i.e. intervals belonging to some $\calI_i$ for
	$0 \leq i \leq s$, where
	$$
	\calI_i = \big\{[j2^i, (j+1)2^i): 0 \leq j \leq 2^{s-i}-1\}
	$$
	and such that each length appears at most twice. For the proof, we set $m_0 = m$. Having
	chosen $m_p$ we select $m_{p+1}$ in such a way that $[m_p, m_{p+1})$ is the longest dyadic
	interval starting at $m_p$ and contained inside $[m_p, n)$. If the lengths of the selected
	dyadic intervals increase then we are done. Otherwise, there is $p$ such that
	$m_{p+1} - m_p \geq m_{p+2} - m_{p+1}$. We show that this implies
	$m_{p+2} - m_{p+1} > m_{p+3} - m_{p+2}$. Suppose for a contradiction that,
	$m_{p+2} - m_{p+1} \leq m_{p+3} - m_{p+2}$. Then
	$$
	[m_{p+1}, 2 m_{p+2} - m_{p+1}) \subseteq [m_{p+1}, m_{p+3}).
	$$
	Therefore, it is enough to show that $2(m_{p+2} - m_{p+1})$ divides $m_{p+1}$. It is clear
	in the case $m_{p+1} - m_p > m_{p+2} - m_{p+1}$. If $m_{p+1} - m_p = m_{p+2} - m_{p+1}$ then,
	by the maximality of $[m_p, m_{p+1})$, $2(m_{p+2} - m_{p+1})$ cannot divide $m_p$, thus divides
	$m_{p+1}$.
	
	Next, let $k_0 < k_1 < \ldots < k_J \leq 2^s$ be any increasing sequence. For
	each $j \in \{0, \ldots, J-1\}$ we may write
	$$
	[k_j, k_{j+1}) = \bigcup_{p=0}^{P_j} [u_p^j, u_{p+1}^j)
	$$
	for some $P_j\ge1$ where each interval $[u_p^j, u^j_{p+1})$ is dyadic. Then
	$$
	\lvert
	a_{k_{j+1}} - a_{k_j}
	\rvert
	\leq
	\sum_{p = 0}^{P_j}
	\big \lvert
	a_{u_{p+1}^j} - a_{u_p^j}
	\big \rvert
	=
	\sum_{i = 0}^s
	\sum_{p:\: [u_p^j, u_{p+1}^j) \in \calI_i}
	\big\lvert
	a_{u_{p+1}^j} - a_{u_p^j}
	\big\rvert.
	$$
	Hence, by Minkowski's inequality
	\begin{multline*}
		\Big(
		\sum_{j = 0}^{J-1}
		\sabs{a_{k_{j+1}} - a_{k_j}}^2
		\Big)^{1/2}
		\leq
		\Big(
		\sum_{j = 0}^{J-1}
		\Big(
		\sum_{i = 0}^s
		\sum_{p:\: [u_p^j, u_{p+1}^j) \in \calI_i}
		\big\lvert a_{u^j_p} - a_{u^j_{p+1}} \big\rvert
		\Big)^2
		\Big)^{1/2}\\
		\leq
		\sum_{i = 0}^s
		\Big(
		\sum_{j = 0}^{J-1}
		\Big(
		\sum_{p:\: [u_p^j, u_{p+1}^j) \in \calI_i}
		\big\lvert a_{u^j_p} - a_{u^j_{p+1}} \big\rvert
		\Big)^2
		\Big)^{1/2}.
	\end{multline*}
	Since for a given $i \in \{0, 1, \ldots, 2^s\}$ and $j \in \{0, 1, \ldots, J-1\}$
	the inner sums contain at most two elements we obtain
	$$
	\Big(
	\sum_{j = 0}^{J-1}
	\sabs{a_{k_{j+1}} - a_{k_j}}^2
	\Big)^{1/2}
	\le
	\sqrt{2}
	\sum_{i = 0}^s
	\Big(
	\sum_{j = 0}^{J-1}
	\sum_{p:\: [u_p^j, u_{p+1}^j) \in \calI_i}
	\big\lvert a_{u^j_p} - a_{u^j_{p+1}} \big\rvert^2
	\Big)^{1/2}
	$$
	which is bounded by the right-hand side of \eqref{eq:62}.
\end{proof}
A \emph{long variation seminorm} $V_r^L$ of a sequence $\big(\seq{a_j}{j \in A}\big)$, is given by
\begin{align*}
	V_r^L\big(a_j: j\in A\big)=V_{r}\big(a_j: j\in A\cap\mathcal{D}\big).
\end{align*}
A \emph{short variation seminorm} $V^S_r$ is given by
\begin{align*}
	V^S_r\big(a_j: j\in A\big)
	=
	\Big(\sum_{n \ge 0} V_r\big(a_j: j \in A_n\big)^r\Big)^{1/r}
\end{align*}
where $A_n = A \cap [2^n, 2^{n+1})$. Then
\begin{align}
	\label{eq:32}
	V_r\big(a_j: j\in \NN\big)
	\lesssim
	V_r^L\big(a_j: j\in \NN\big)
	+V_r^S\big(a_j: j\in \NN\big).
\end{align}
The next lemma will be used in the estimates for short variations. It illustrates the ideas
which have been explored several times (see \cite{jkrw}, or recently
\cite{k, zk}).
\begin{lemma}
\label{lem:8}
	Let $u, v \in \NN$, $u < v$. For any integer $h \in \{1,\ldots, v-u\}$ there is a strictly
	increasing sequence of integers $\big(\seq{t_j}{0 \leq j \leq h}\big)$ with $t_0 = u$ and
	$t_h = v$ such that for every $r\in[1, \infty)$
        \begin{align}
          \label{eq:1}
	V_r\big(a_j: u \leq j \leq v\big)
          \lesssim
		 \Big(\sum_{j=0}^h|a_{t_j}|^r\Big)^{1/r}
		+\Big(\sum_{j=0}^{h-1}\Big(\sum_{k=t_j}^{t_{j+1}-1}|a_{k+1}-a_k|\Big)^r\Big)^{1/r}.
        \end{align}
Moreover, if $p\ge r$ then
\begin{multline}
  \label{eq:5}
\Big(\sum_{j=0}^h|a_{t_j}|^r\Big)^{1/r}
		+\Big(\sum_{j=0}^{h-1}\Big(\sum_{k=t_j}^{t_{j+1}-1}|a_{k+1}-a_k|\Big)^r\Big)^{1/r}\\
\lesssim
          h^{1/r-1/p}\Big(\sum_{j=0}^h|a_{t_j}|^p\Big)^{1/p}
		+h^{1/r-1}(v-u)^{1-1/p}\bigg(\sum_{j=u}^{v-1}|a_{j+1}-a_j|^p\bigg)^{1/p}.
\end{multline}
The implicit constants in \eqref{eq:1} and \eqref{eq:5} are
independent of $h, u$ and $v$.
\end{lemma}
\begin{proof}
        Fix $h \in \{1,\ldots, v-u\}$ and choose a sequence
	$\big(\seq{t_j}{1 \leq j \leq h}\big)$ such that
	$t_0 = u$, $t_h = v$ and $|t_{j+1}-t_j| \simeq (v-u)/h$. Then
	\begin{multline*}
		V_r\big(a_j: u \leq j \leq v\big)
		\lesssim \Big(\sum_{j=0}^h|a_{t_j}|^r\Big)^{1/r}
		+\Big(\sum_{j=0}^{h-1}V_r\big(a_k: t_j \leq k \le t_{j+1}\big)^r\Big)^{1/r}\\
		\lesssim
		 \Big(\sum_{j=0}^h|a_{t_j}|^r\Big)^{1/r}
		+\Big(\sum_{j=0}^{h-1}\Big(\sum_{k=t_j}^{t_{j+1}-1}|a_{k+1}-a_k|\Big)^r\Big)^{1/r}.
	\end{multline*}
        If $p\ge r$ then by H\"{o}lder's inequality the last sum can
        be dominated by
        \begin{multline*}
          \Big(\sum_{j=0}^h|a_{t_j}|^r\Big)^{1/r}
		+\bigg(\sum_{j=0}^{h-1}\Big(\sum_{k=t_j}^{t_{j+1}-1}|a_{k+1}-a_k|\Big)^r\bigg)^{1/r}\\
                \lesssim 
          h^{1/r-1/p}\Big(\sum_{j=0}^h|a_{t_j}|^p\Big)^{1/p}
		+h^{1/r-1/p}\bigg(\sum_{j=0}^{h-1}\Big(\sum_{k=t_j}^{t_{j+1}-1}|a_{k+1}-a_k|\Big)^p\bigg)^{1/p}\\
                \lesssim
          h^{1/r-1/p}\Big(\sum_{j=0}^h|a_{t_j}|^p\Big)^{1/p}
		+h^{1/r-1/p}\bigg(\sum_{j=0}^{h-1}(t_{j+1}-t_j)^{p(1-1/p)}\Big(\sum_{k=t_j}^{t_{j+1}-1}|a_{k+1}-a_k|^p\Big)\bigg)^{1/p}\\
                \lesssim
          h^{1/r-1/p}\Big(\sum_{j=0}^h|a_{t_j}|^p\Big)^{1/p}
		+h^{1/r-1}(v-u)^{1-1/p}\Big(\sum_{j=u}^{v-1}|a_{j+1}-a_j|^p\Big)^{1/p}
        \end{multline*}
and this completes the proof of the lemma.
\end{proof}
We observe that, if $1\le r\le p$ and $(f_j: j\in\NN)$ is a sequence of functions in $\ell^p\big(\ZZ^d\big)$ and $v-u\ge2$ then
\begin{align}
	\label{eq:39}
	\big\|V_r\big(f_j: j\in [u, v]\big)\big\|_{\ell^p}
	\lesssim
	\max\big\{\mathbf U_p,
	(v-u)^{1/r} \mathbf U_p^{1-1/r} \mathbf V_p^{1/r}\big\}
\end{align}
where
$$
\mathbf U_p =	\max_{u \leq j \leq v}\|f_j\|_{\ell^p},
\quad \text{and} \quad
\mathbf V_p = \max_{u \leq j < v}\|f_{j+1}-f_j\|_{\ell^p}.
$$
Indeed, let
$$
h=\big\lceil (v-u) \mathbf V_p / (4\mathbf U_p)\big\rceil.
$$
Then $h \in [1, v- u]$. If $h\ge2$, by Lemma \ref{lem:8}, we have
$$
\big\|V_r\big(f_j: u \leq j \leq v\big)\big\|_{\ell^p}
\lesssim h^{1/r}
\mathbf U_p  + h^{1/r-1}(v-u) \mathbf V_p 
\lesssim
(v-u)^{1/r} \mathbf U_p^{1-1/r} \mathbf V_p^{1/r}.
$$
If $h=1$ then
$\mathbf V_p \lesssim (v - u)^{-1} \mathbf U_p$ and hence
$$
\big\|V_r\big(f_j: u \leq j \leq v\big)\big\|_{\ell^p}
\lesssim
\mathbf U_p.
$$

\subsection{Lifting lemma}
Let $\calP=(\calP_1,\ldots, \calP_{d_0}): \ZZ^k \rightarrow \ZZ^{d_0}$ be a mapping whose
components $\calP_j$ are integer-valued polynomials on $\ZZ^k$ such that $\calP_j(0) = 0$. We set
$$
N_0 = \max\{ \deg \calP_j : 1 \leq j \leq d_0\}.
$$
It is convenient to work with the set
$$
\Gamma =
\big\{
	\gamma \in \ZZ^k \setminus\{0\} : 0 \leq \gamma_j \leq N_0
	\text{ for each } j = 1, \ldots, k
\big\}
$$
with the lexicographic order. Then each $\calP_j$ can be expressed as
$$
\calP_j(x) = \sum_{\gamma \in \Gamma} c_j^\gamma x^\gamma
$$
for some $c_j^\gamma \in \ZZ$. Let us denote by $d$ the cardinality of the set $\Gamma$.
We identify $\RR^d$ with the space of all vectors whose coordinates are labelled by multi-indices
$\gamma \in \Gamma$. Let $A$ be a diagonal $d \times d$ matrix such that
$$
(A v)_\gamma = \abs{\gamma} v_\gamma.
$$
For $t > 0$ we set
$$
t^{A}=\exp(A\log t)
$$
i.e. $t^A x=(t^{|\gamma|}x_{\gamma}: \gamma\in \Gamma)$ for any $x\in \RR^d$. Next, we introduce the \emph{canonical}
polynomial mapping
$$
\calQ = \big(\seq{\calQ_\gamma}{\gamma \in \Gamma}\big) : \ZZ^k \rightarrow \ZZ^d
$$
where $\calQ_\gamma(x) = x^\gamma$ and $x^\gamma=x_1^{\gamma_1}\cdot\ldots\cdot x_k^{\gamma_k}$.
The coefficients $\big(\seq{c_j^\gamma}{\gamma \in \Gamma, j \in \{1, \ldots, d_0\}}\big)$ define
a linear transformation $L: \RR^d \rightarrow \RR^{d_0}$ such that $L\calQ = \calP$. Indeed, it is
enough to set
$$
(L v)_j = \sum_{\gamma \in \Gamma} c_j^\gamma v_\gamma
$$
for each $j \in \{1, \ldots, d_0\}$ and $v \in \RR^d$. The next lemma, inspired by the continuous analogue (see \cite{deL} or \cite[p. 515]{bigs})
reduces proofs of Theorem \ref{thm:1} and Theorem \ref{thm:4} to the canonical polynomial mapping. See also Lemma 2.1 in \cite{mst1}.
\begin{lemma}
	\label{lem:1}
        Let $R_N^{\calP}$ be one of the operators $M_N^{\calP}$ or
        $T_N^{\calP}$. Suppose that for some $p \in (1, \infty)$ and
        $r \in(2, \infty)$
	$$
	\big\lVert
	\bvar{\big(\seq{R_N^\calQ f}{N \in \NN}\big)}
	\big\rVert_{\ell^p(\ZZ^d)}
	\leq C_{p, r}
	\vnorm{f}_{\ell^p(\ZZ^d)}.
	$$
	Then
	\begin{equation}
		\label{eq:6}
		\big\lVert
		\bvar{\big(\seq{R_N^\calP f}{N \in \NN}\big)}
		\big\rVert_{\ell^p(\ZZ^{d_0})}
		\leq
		C_{p, r}
		\vnorm{f}_{\ell^p(\ZZ^{d_0})}.
	\end{equation}
\end{lemma}
\begin{proof}
For the proof we refer to \cite{mst1}.
\end{proof}
From now on $M_N$ and  $T_N$ will denote the operators defined for the canonical
polynomial mapping $\calQ$, i.e. $M_N = M_N^\calQ$ and $T_N = T_N^\calQ$.

\section{Further tools}
\label{sec:3}
\subsection{Gaussian sums}

 For $q\in\NN$ let us define
\[
A_q=\big\{a\in\NN_q^d: \mathrm{gcd}\big(q, \mathrm(a_{\gamma}: \gamma\in\Gamma)\big)\big\}.
\]
Next, for $q \in \NN$ and $a \in A_q$ we define the \emph{Gaussian sum} 
$$
G(a/q) = q^{-k} \sum_{y \in \NN^k_q} e^{2\pi i \sprod{(a/q)}{\calQ(y)}}.
$$
Let us observe that, by the multi-dimensional variant of
Weyl's inequality (see \cite[Proposition 3]{SW0}), there exists $\delta>0$ such that
\begin{equation}
	\label{eq:20}
	\lvert G(a/q) \rvert \lesssim q^{-\delta}.
\end{equation}

\subsection{Weyl's estimates}
 Let $P$ be a polynomial in $\RR^k$ of degree $d \in\NN$ such that
\[
	P(x) = \sum_{0 < \norm{\gamma} \leq d} \xi_\gamma x^\gamma.
\]
Given $N \geq 1$, let $\Omega_N$ be a convex set such that
\[
 \Omega_N \subseteq B_{cN}(x_0)
\]
for some $x_0\in\RR^k$ and $c > 0$, where $B_r(x_0)=\big\{x \in \RR^k : \norm{x-x_0} \leq r \big\}$.
 We define
\[
	S_N = \sum_{n \in \Omega_N \cap \ZZ^k} e^{2\pi i P(n)}\varphi(n)
\]
where $\varphi:\RR^k\rightarrow \CC$ is 
 a $\mathcal
C^1\big(\RR^k\big)$ function which for some $C>0$ satisfies
\[
	|\varphi(x)|\le C, \qquad \text{and} \qquad |\nabla \varphi(x)|\le C(1+|x|)^{-1}.  
\]
In \cite{mst1} we proved the following refinement  of multi-dimensional
Weyl's inequality.
\begin{theorem}
	\label{thm:3}
        Assume that there is a multi-index $\gamma_0$ such that $0 < \norm{\gamma_0} \leq d$ and
	\[
		\Big\lvert
		\xi_{\gamma_0} - \frac{a}{q}
		\Big\rvert
		\leq
		\frac{1}{q^2}
	\]
        for some integers $a, q$ such that $0\le a\le q$ and $(a, q) =
        1$. Then for any $\alpha>0$ there is  $\beta_{\alpha}>0$
          so that, for any
        $\beta\ge \beta_{\alpha}$, if
	\[
		(\log N)^\beta \leq q \leq N^{\norm{\gamma_0}} (\log N)^{-\beta}
	\]
        then there is a constant $C>0$
	\[ 
		|S_N|
		\leq
		C
		N^k (\log N)^{-\alpha}.
	\]
	The implied  constant $C$ is independent of $N$. 
\end{theorem}

\subsection{Transference principle}

Let $\calF$ denote the Fourier transform on $\RR^d$ defined for any function 
$f \in L^1\big(\RR^d\big)$ as
$$
\calF f(\xi) = \int_{\RR^d} f(x) e^{2\pi i \sprod{\xi}{x}} {\: \rm d}x.
$$
If $f \in \ell^1\big(\ZZ^d\big)$ we set
$$
\hat{f}(\xi) = \sum_{x \in \ZZ^d} f(x) e^{2\pi i \sprod{\xi}{x}}.
$$
To simplify the notation we denote by $\mathcal F^{-1}$ the inverse Fourier transform on $\RR^d$
or the inverse Fourier transform on $\TT^d$ (Fourier coefficients), depending on the context.

Let $\eta: \RR^d \rightarrow \RR$ be a smooth function such that $0 \leq \eta(x) \leq 1$ and
$$
\eta(x) =
\begin{cases}
	1 & \text{ for } \norm{x} \leq 1/(16 d),\\
	0 & \text{ for } \norm{x} \geq 1/(8 d).
\end{cases}
$$
\begin{remark}
We may additionally assume that $\eta$ is a convolution of two non-negative smooth functions
$\phi$ and $\psi$ with compact supports contained inside $(-1/(8d), 1/(8d))^d$.   
\end{remark}

Let $\big(\seq{\Theta_N}{N \in \NN}\big)$ be a sequence of multipliers on $\RR^d$
with a property that for each $p \in (1, \infty)$ and $r\in(2,
\infty)$ there is a constant $\boldB_{p, r} > 0$
such that for any $f \in
L^p\big(\RR^d\big) \cap L^2\big(\RR^d\big)$
\begin{equation}
\label{eq:200}
	\big\lVert
	 V_r\big(\seq{\calF^{-1}\big(\Theta_N \calF f \big)}{N \in
          \NN}\big)
	\big\rVert_{L^p}
	\leq
	\boldB_{p, r}
\|f\|_{L^p}.
\end{equation}
Moreover, $\boldB_{p, r}\le\boldB_{p}\frac{r}{r-2}$ for some $\boldB_{p}>0$.
In fact we will be only interested in the multipliers which are
discussed in the Appendix, see Theorem \ref{thm:20} and Theorem \ref{thm:21}.

We assume that
$\mathcal R$
is a diagonal $d\times d$  matrix with  positive entries
$(r_{\gamma}: \gamma\in\Gamma)$ such that
$\inf_{\gamma\in\Gamma}r_{\gamma}\ge h$ for some $h>0$. In \cite{mst1}
we proved, in particular, the following version of the transference principle.

\begin{proposition}
	\label{prop:10}
    Under assumption \eqref{eq:200} for each $p\in (1, \infty)$ and
    $r\in(2, \infty)$ there is a constant $C > 0$ such that for each
	$Q\in\NN$ and $h \geq 2 Q^{d+1}$ and any $m \in \NN_Q^k$ we have
	\begin{align*}
		\big\lVert
		\mathcal V_r
		\big(\seq{\calF^{-1}\big(\Theta_N\eta(\mathcal R \: \cdot \:) 
		\hat{f}\big)(Q x + m)}{N \in \NN}\big)
		\big\rVert_{\ell^p(x)}
		\leq
		C \boldB_{p, r}
		\big\lVert
		\calF^{-1}\big(\eta(\mathcal R \: \cdot \:) \hat{f} \big)(Q x + m)
		\big\rVert_{\ell^p(x)}.
	\end{align*}
\end{proposition}
See also the discussion of sampling in \cite[Proposition 2.1
and Corollary 2.5]{MSW}.

\subsection{Ionescu--Wainger type multipliers}

We now introduce necessary notation to define Ionescu and Wainger type
multipliers.  To
fix notation set $\rho>0$ and  for every $N\in\NN$, let us define $N_0=\lfloor
N^{\rho/2}\rfloor+1$, moreover let $Q_0=(N_0!)^D$ where
$D=D_{\rho}=\lfloor 2/\rho\rfloor+1$. Let $\mathbb P$ denote the set
of all prime numbers and $\mathbb P_N=\mathbb P\cap (N_0, N]$. For any
$k\in\NN_D$ and $V\subseteq\mathbb P_N$ we define
\[
\Pi_k(V)=\big\{p_{1}^{\gamma_{1}}\cdot\ldots\cdot
p_{k}^{\gamma_{k}}: \ \gamma_{l}\in\NN_D\  \text{and}\ p_{l}\in
V\ \text{are distinct for all $1\le l\le k$} \big\}.
\]
Therefore,  $\Pi_{k_1}(V)\cap \Pi_{k_2}(V)=\emptyset$ if $k_1\not=k_2$ and
\[
\Pi(V)=\bigcup_{k\in\NN_D}\Pi_k(V)
\]
is the set of all products of primes factors from $V$ of
length at most $D$, at powers between $1$ and $D$. 

It is easy to see, now, that every integer $q\in\NN_N$ can be uniquely
written as $q=Q\cdot w$ where $Q|Q_0$ and $w\in \Pi(\mathbb
P_N)\cup\{1\}$.  Moreover, for sufficiently large $N\in\NN$
\[q=Q\cdot w\le Q_0\cdot w\le (N_0!)^DN^{D^2}\le
e^{N^{\rho}}
\]
thus for the set 
\[
 P_N=\big\{q=Q\cdot w: Q|Q_0\ \text{and}\ w\in
\Pi(\mathbb P_{N})\cup\{1\}\big\}
\]
we have 
$\NN_N\subseteq P_N\subseteq\NN_{e^{N^{\rho}}}$. Moreover, if $N_1\le
N_2$ then $P_{N_1}\subseteq P_{N_2}$. For any $S\subseteq \NN$ define
\[
\mathcal R(S)=\{a/q\in\mathbb Q^d\cap\TT^d: a\in A_q \text{ and } q\in
S\}.
\]
We will assume that
for every $p\in(1, \infty)$ there is a constant $\boldA_p >0 $ 
\begin{align}
  \label{eq:155}
  \big\lVert\calF^{-1}\big(\Theta\calF f\big)\big\rVert_{L^p}
  \leq
  \boldA_p \vnorm{f}_{L^p}.
\end{align}
For each $N \in \NN$
we define new periodic  multipliers
\[
	\Delta_N(\xi)
	=\sum_{a/q \in\mathscr{U}_N}
	 \Theta(\xi - a/q) \eta_N(\xi - a/q)
\]
where $\eta_N(\xi)=\eta\big(\mathcal E_N^{-1}\xi\big)$ and $\mathcal E_N$
is a diagonal $d\times d$  matrix with positive entries
$\big(\varepsilon_{\gamma}: \gamma\in\Gamma\big)$ such that
$\varepsilon_{\gamma}\le e^{-N^{2\rho}}$  and
\begin{align}
  \label{eq:156}
\mathscr{U}_N=\mathcal R( P_N)
\end{align}
Furthermore, if $N_1\le N_2$ then $\mathscr{U}_{N_1}\subseteq \mathscr{U}_{N_2}$. 
The main result of this subsection is the following.
\begin{theorem}
	\label{th:3}
	Let $\Theta$ be a  multiplier on $\RR^d$ obeying \eqref{eq:155}. 
	Then for every $\rho>0$ and $p\in(1, \infty)$ there is a
        constant $C_{\rho, p} > 0$ such that
	for any $N\in\NN$ and $f \in \ell^p\big(\ZZ^d\big)$
	\[
		\big\lVert
		\calF^{-1}\big(\Delta_{N}
		\hat{f}\big)\big\rVert_{\ell^p} 
		\leq C_{\rho, p} \boldA_p(\log N)
		\vnorm{f}_{\ell^p}.
	\]
\end{theorem}
Theorem \ref{th:3} inspired by the ideas of  Ionescu and Wainger from
\cite{iw} was proven in \cite{mst1}. 

\section{Long variation estimates for averaging operators}
\label{sec:4}

For any function $f: \ZZ^d \rightarrow \CC$ with a finite support we have
$$
M_N f(x) = K_N * f(x)
$$
where $K_N$ is a kernel defined by
\[
	K_N(x) = |\BB_N|^{-1} \sum_{y \in \BB_N} \delta_{\calQ(y)}
\]
and $\delta_y$ denotes Dirac's delta at $y \in \ZZ^k$. Let $m_N$ denote the discrete  Fourier transform of $K_N$, i.e.
$$
m_N(\xi) = |\BB_N|^{-1} \sum_{y \in \BB_N} e^{2\pi i \sprod{\xi}{\calQ(y)}}.
$$
Finally, we define
$$
\Phi_N(\xi) = |B_1|^{-1}\int_{B_1} e^{2\pi i \sprod{\xi}{\calQ(N y)}} {\: \rm d}y.
$$
Using a multi-dimensional version of van der Corput lemma (see \cite{bigs, ccw}) we may
estimate
\begin{equation}
	\label{eq:9}
	\abs{\Phi_N(\xi)}
	\lesssim
	\min\big\{1, \norm{N^A \xi}_\infty^{-1/d} \big\}.
\end{equation}
Additionally, we have
\begin{equation}
	\label{eq:10}
	\abs{\Phi_N(\xi) - 1}
	\lesssim
	\min\big\{1, \norm{N^A \xi}_\infty\big\}.
\end{equation}

\begin{proposition}
  \label{prop:0}

 There is a constant $C>0$ such
  that for every $N\in\NN$ and for every $\xi\in [-1/2, 1/2)^d$ satisfying 
        $$
	\Big\lvert \xi_\gamma - \frac{a_\gamma}{q} \Big\rvert \leq
        L_1^{-|\gamma|}L_2
	$$
	for all $\gamma \in \Gamma$, where  $1\le q\le L_3\le N^{1/2}$, $a\in
        A_q$, $L_1\ge N$  and $L_2\ge1$ we have 
        \begin{align}
           \label{eq:41}
          \big|m_N(\xi)-G(a/q)\Phi_{N}(\xi-a/q)\big|\le
          C\Big(L_3N^{-1}+L_2L_3N^{-1}\sum_{\gamma \in
            \Gamma}\big(N/L_1\big)^{|\gamma|}\Big)\lesssim L_2L_3N^{-1}.
        \end{align}

\end{proposition}
\begin{proof}
	Let $\theta = \xi - a/q$. For any $r \in \NN_q^k$, if $y \equiv r \pmod q$ then
	for each $\gamma \in \Gamma$
	$$
	\xi_\gamma y^\gamma \equiv \theta_\gamma y^\gamma
	+ (a_\gamma/q) r^\gamma \pmod 1,
	$$
	thus
	$$
	\sprod{\xi}{\calQ(y)} \equiv \sprod{\theta}{\calQ(y)} + \sprod{(a/q)}{\calQ(r)} \pmod 1.
	$$
	Therefore,
	\[
	|\BB_N|^{-1}\sum_{y \in \BB_N} e^{2\pi i \sprod{\xi}{\calQ(y)}} 
	=
	q^{-k}\sum_{r \in \NN_q^k}
	e^{2\pi i \sprod{(a/q)}{\calQ(r)}}
	\cdot \Big(q^k|\BB_N|^{-1}\sum_{\atop{y \in
            \ZZ^k}{|qy+r|\le N}}
	e^{2\pi i \sprod{\theta}{\calQ(qy+r)}}\Big).
	\]
	If $ \norm{q y + r}, \norm{qy} \leq N$ then
	\[
	\big\lvert
	\sprod{\theta}{\calQ(q y + r)} - \sprod{\theta}{\calQ(q y)}
	\big\rvert
	\lesssim
	\norm{r}
	\sum_{\gamma \in \Gamma}
	\abs{\theta_\gamma}
	\cdot
	N^{(\abs{\gamma} - 1)}
	\lesssim
	q \sum_{\gamma \in \Gamma}
	L_1^{-\abs{\gamma}}L_2 N^{(\abs{\gamma}-1)}
	\lesssim
	L_2L_3/N\sum_{\gamma \in \Gamma}\big(N/L_1\big)^{\abs{\gamma}}.
	\]
	Thus
        \[
        |\BB_N|^{-1}\sum_{y \in \BB_N} e^{2\pi i \sprod{\xi}{\calQ(y)}} 
	=G(a/q)\cdot q^k|\BB_N|^{-1}\sum_{\atop{y \in \ZZ^k}{|qy|\le N}}
	e^{2\pi i \sprod{\theta}{\calQ(qy)}}+\mathcal O\Big(q/N+L_2L_3/N\sum_{\gamma \in \Gamma}\big(N/L_1\big)^{|\gamma|}\Big).
        \] 
We have used the formula for the number of lattice points in the
Euclidean ball, i.e. $|\BB_N|=|B_1|N^k+\mathcal
O\big(N^{k-1}\big)$ as $N\to\infty$.
Now we are going to replace the  exponential sum on the right-hand
side of the last display by the integral. By the mean value theorem,  we  obtain
  \begin{multline*}
    \Big|\sum_{\atop{y \in \ZZ^k}{
        |qy|\le  N}}
    e^{2\pi i \sprod{\theta}{\calQ(qy)}}-\int_{ |qt|\le
      N}e^{2\pi i \theta\cdot\calQ(qt)}{\: \rm d}t\Big|\\
    = \Big|\sum_{y \in \ZZ^k} e^{2\pi i
      \sprod{\theta}{\calQ(qy)}}\ind{B_N}(qy)-\sum_{y \in \ZZ^k}
\int_{y+(0, 1]^k}e^{2\pi i \theta\cdot\calQ(qt)}\ind{B_N}(qt){\: \rm d}t\Big|\\
=\Big|\sum_{y \in \ZZ^k} \int_{(0, 1]^k}e^{2\pi i
      \theta\cdot\calQ(qy)}\ind{B_N}(qy)-e^{2\pi i
      \theta\cdot\calQ(q(t+y))}\ind{B_N}(q(t+y)){\:\rm d}t\Big|\\
=\mathcal O\Big((N/q)^{k-1}+(N/q)^{k}L_2L_3/N\sum_{\gamma \in \Gamma}\big(N/L_1\big)^{|\gamma|}\Big).
  \end{multline*}
This completes the proof of Proposition \ref{prop:0}.
\end{proof}

In Remark \ref{rem:10} we mentioned that Theorem \ref{thm:1} holds
with the operators $M_N$ defined with the sets $\mathbb G_N$ instead
of $\mathbb B_N$. Then we obtain analogous definitions of $K_N$, $m_N$
and $\Phi_N$ with the sets $\mathbb G_N$ and $G=G_1$ in place of the sets
$\mathbb B_N$ and $B_1$ respectively. All of the arguments remain
unchanged apart from the proof of Proposition \ref{prop:0}. Here we
must proceed more delicately.  However, \cite[Proposition 3.1]{mst1}
used with the sets $\mathbb G_N$ in place of the asymptotic formula
for the number of lattice points in $\mathbb B_N$ does the job and we
obtain conclusion of the same type. Proposition 3.1 from \cite{mst1}
states that for a given convex set $\Omega \subseteq\RR^k$ such that
$B_{cr}(x_0')\subseteq \Omega\subseteq B_r(x_0)$ for some $x_0, x_0'\in\RR^k$
and $c>0$, we have that for any $1\le s\le r$ the number of lattice points
$N_{\Omega}$ in $\Omega$ of distance $< s$  from the boundary of
$\Omega$ is $\mathcal O\big(sr^{k-1}\big)$.

The main result of this section is the following.
\begin{theorem}
\label{thm:5}
	For every $1 < p < \infty$  and $r\in(2, \infty)$ there is $C_{p} > 0$ such that for all
	$f \in \ell^p\big(\ZZ^{d_0}\big)$
        \begin{align}
          \label{eq:4}
          	\big\lVert
	V_r\big( M_{2^n} f: n\in\NN\big)
	\big\rVert_{\ell^p}\le
	C_{p}\frac{r}{r-2}\|f\|_{\ell^p}.
        \end{align}
\end{theorem}

As we have indicated, a complete proof of this theorem will not be
given here. It will suffice to say that it uses Proposition
\ref{prop:0} and follows the ideas \cite[Section 6]{mst1}. If one
replaces the supremum norm occuring there by the variational norm
$V_r$, we can then obtain \eqref{eq:4}. One can also follow closely
the corresponding argument for the singular Radon transforms in
Section 6.

\section{Short variation estimates for averaging operators}
\label{sec:5}
According to \eqref{eq:32} and  the estimates for long variations from
the previous section it remains to  prove that for all $p\in(1, \infty)$ there is $C_p>0$
such that for all $f\in\ell^p\big(\ZZ^d\big)$ with finite support we have
\[ 
  \Big\|\Big(\sum_{n\ge0} V_2\big(M_{N}f: N\in[2^n,
  2^{n+1})\big)^2\Big)^{1/2}\Big\|_{\ell^p}\le C_p\|f\|_{\ell^p}.
\]
For this purpose, fix the numbers $\chi>0$ and $l\in\NN$ whose precise
values will be specified later, and let us introduce for every $n\in\NN_0$
the multipliers
\[
  \Xi_{n}(\xi)=\sum_{a/q\in\mathscr U_{n^l}}\eta\big(2^{n(A-\chi I)}(\xi-a/q)\big)
\]
with $\mathscr U_{n^l}$ defined as in \eqref{eq:156}.
Theorem \ref{th:3} guarantees  that for every $p\in(1, \infty)$ 
\begin{align}
\label{eq:28}
  \big\|\mathcal
  F^{-1}\big(\Xi_{n}\hat{f}\big)\big\|_{\ell^p}\lesssim \log(n+2)\|f\|_{\ell^p}.
\end{align}
The implicit constant in \eqref{eq:28} depends on the parameter
$\rho>0$, which was fixed, see Section \ref{sec:3}. However, from
now on we will assume that $\rho>0$ and the integer $l\ge10$ are related by the
equation
\[
  10\rho l=1.
\]
Observe that
\begin{multline}
  \label{eq:30}
   \Big\|\Big(\sum_{n\ge0} V_2\big(M_{N}f: N\in[2^n,
  2^{n+1})\big)^2\Big)^{1/2}\Big\|_{\ell^p}\\
\le 
  \Big\|\Big(\sum_{n\ge0} V_2\big((M_{N}-M_{2^n})\mathcal
  F^{-1}\big(\Xi_{n}\hat{f}\big): N\in[2^n,
  2^{n+1})\big)^2\Big)^{1/2}\Big\|_{\ell^p}\\
+ 
  \Big\|\Big(\sum_{n\ge0} V_2\big((M_{N}-M_{2^n})\mathcal
  F^{-1}\big((1-\Xi_{n})\hat{f}\big): N\in[2^n,
  2^{n+1})\big)^2\Big)^{1/2}\Big\|_{\ell^p}.
\end{multline}
\subsection{The estimate of the second norm in \eqref{eq:30}}
We may assume without of
loss of generality, that $1<
r\le \min\{2, p\}$, since $r$-variations are decreasing,  and it suffices to show that
\begin{align}
  \label{eq:38}
  \big\|V_r\big((M_{N}-M_{2^n})\mathcal
  F^{-1}\big((1-\Xi_{n})\hat{f}\big): N\in[2^n,
  2^{n+1})\big)\big\|_{\ell^p}\lesssim (n+1)^{-2}\|f\|_{\ell^p}.
\end{align}
Appealing to  \eqref{eq:39}
we immediately see that
\begin{align}
  \label{eq:50}
  \big\| V_r\big((M_{N}-M_{2^n})\mathcal
  F^{-1}\big((1-\Xi_{n})\hat{f}\big): N\in[2^n,
  2^{n+1})\big)\big\|_{\ell^p}
\lesssim\max\big\{\mathbf U_p, 2^{n/r}\mathbf U_p^{1-1/r}\mathbf V_p^{1/r}\big\}
\end{align}
where 
\begin{align*}
  \mathbf U_p=\sup_{2^n\le N\le 2^{n+1}}\big\|(M_{N}-M_{2^n})\mathcal
  F^{-1}\big((1-\Xi_{n})\hat{f}\big)\big\|_{\ell^p}
\end{align*}
and 
\begin{align*}
\mathbf V_p=  \sup_{2^n\le N< 2^{n+1}}\big\|(M_{N+1}-M_{N})\mathcal
  F^{-1}\big((1-\Xi_{n
})\hat{f}\big)\big\|_{\ell^p}.
\end{align*}
In view of \eqref{eq:28} we see that
\begin{align}
  \label{eq:58}
 \mathbf U_p\lesssim \log(n+2) \|f\|_{\ell^p} \quad \text{and} \quad
  \mathbf V_p\lesssim 2^{-n} \log(n+2) \|f\|_{\ell^p}.
\end{align}
In fact we show that there is possible a refinement of these estimates
for $p=2$, which in turn improve the estimates from \eqref{eq:58}
for all $p\in(1, \infty)$ and finally one could conclude \eqref{eq:38}.
Indeed, we claim that for big enough $\alpha>0$, which will be specified
later, and for all $n\in\NN_0$ and $N\simeq 2^n$ we have
\begin{align}
  \label{eq:60}
 \big|\big( m_{N}(\xi)-
 m_{2^n}(\xi)\big)(1-\Xi_{n}(\xi))\big|\lesssim (n+1)^{-\alpha}.
\end{align}
This estimate will be a consequence of Theorem \ref{thm:3}. To do so,
by Dirichlet's principle we have for every $\gamma\in\Gamma$ 
\[
	\bigg|
	\xi_{\gamma}-\frac{a_{\gamma}}{q_{\gamma}}
	\bigg|
	\le
	\frac{n^\beta}{q_{\gamma} 2^{n|\gamma|}}
\] 
where $1\le q_{\gamma}\le n^{-\beta}2^{n|\gamma|}$. In order to apply
Theorem \ref{thm:3} we must show that there exists some
$\gamma\in\Gamma$ such that $n^{\beta}\le q_{\gamma }\le
n^{-\beta}2^{n|\gamma|}$. Suppose for a contradiction that for every
$\gamma \in \Gamma$  we have $1\le q_{\gamma }<n^{\beta} $. Then for
some 
$q\le \mathrm{lcm}(q_{\gamma}: \gamma\in\Gamma)\le n^{\beta d}$ we have 
\[
	\bigg|
	\xi_{\gamma}- \frac{a_{\gamma}'}{q}
	\bigg|
	\le \frac{n^{\beta}}{2^{n|\gamma|}}
\] 
where $\mathrm{gcd}\big(q, \mathrm{gcd}({a_{\gamma}'}:
\gamma\in\Gamma)\big)=1$.
Hence, taking $a'=(a_{\gamma}': \gamma\in\Gamma)$ we have
$a'/q\in\mathscr U_{n^l}$ provided that $\beta d<l$.
On the other hand, if $1-\Xi_{n}(\xi)\not=0$ then for every
$a'/q\in\mathscr U_{n^l}$ there exists $\gamma\in\Gamma$ such that
\[
	\bigg|
	\xi_{\gamma}- \frac{a_{\gamma}'}{q}
	\bigg|
	> \frac{1}{16d\cdot2^{n(|\gamma|-\chi)}}.
\] 
Therefore, one obtains
\[
2^{n\chi}<16dn^{\beta}
\]
which is not possible if $n$ is sufficiently large. We have already shown that there exists some
$\gamma\in\Gamma$ such that $n^{\beta}\le q_{\gamma }\le
n^{-\beta}2^{n|\gamma|}$ and consequently Theorem \ref{thm:3} yields 
\[
|m_N(\xi)|\lesssim (n+1)^{-\alpha}
\]
provided that $1-\Xi_{n}(\xi)\not=0$ proving \eqref{eq:60}. We obtain
\begin{align}
\label{eq:56}  
\mathbf U_2\lesssim(1+n)^{-\alpha} \log(n+2) \|f\|_{\ell^2}.
\end{align}
Interpolating \eqref{eq:56} with \eqref{eq:58} we obtain 
\begin{align}
\label{eq:57}  
\mathbf U_p\lesssim(1+n)^{-c_p\alpha} \log(n+2) \|f\|_{\ell^p}.
\end{align}
for some $c_p>0$. Now, by choosing $\alpha>0$ and $l\in\NN
$ appropriately
large we see that \eqref{eq:57} combined with \eqref{eq:50} easily
imply \eqref{eq:38}. 

\subsection{The estimate of the first  norm in \eqref{eq:30}}
Note that for any $\xi\in\TT^d$ such that 
\[
	\bigg|
	\xi_{\gamma}-
	\frac{a_{\gamma}}{q}
	\bigg|
	\le 
	\frac{1}{8d\cdot 2^{n(|\gamma|-\chi)}}
\] 
for any $\gamma \in \Gamma$
with $1\le q\le e^{n^{1/10}}$ we have 
\begin{align}
\label{eq:61}
  m_N(\xi)=G(a/q)\Phi_N(\xi-a/q)+q^{-\delta}E_{2^n}(\xi)
\end{align}
for every $N \in [2^n, 2^{n+1})$,
where 
\begin{align}
\label{eq:64}
  |E_{2^n}(\xi)|\lesssim 2^{-n/2}.
\end{align}
These two properties \eqref{eq:61} and \eqref{eq:64}  follow from Proposition \ref{prop:0} with
$L_1=2^{n}$, $L_2=(8d)^{-1}2^{n\chi}$ and $L_3=e^{n^{1/10}}$, since
\[
|E_{2^n}(\xi)|\lesssim q^{\delta}L_2L_32^{-n}\lesssim
e^{-n((1-\chi)\log 2-2n^{-9/10})}
\lesssim 2^{-n/2}
\]
which holds for sufficiently large $n\in\NN$ when $\chi>0$ is
sufficiently small. Let us introduce for
every $j, n\in\NN_0$ the
new multipliers
\[
  \Xi_{n}^j(\xi)=\sum_{a/q\in\mathscr U_{n^l}}\eta\big(2^{nA+jI}(\xi-a/q)\big)
\]
and note that
\begin{multline*}
    \bigg\|
	\Big(\sum_{n\ge0}
	V_2 \big((M_{N}-M_{2^n})\mathcal F^{-1} \big(\Xi_{n}\hat{f}\big): N\in[2^n, 2^{n+1})\big)^2\Big)^{1/2}
	\bigg\|_{\ell^p} \\
	\le 
	\bigg\|\Big(\sum_{n\ge0} V_2\big((M_{N}-M_{2^n})\mathcal
  F^{-1}\Big(\sum_{-\chi n\le j< n}\big(\Xi_{n}^j-\Xi_{n}^{j+1}\big)\hat{f}\Big): N\in[2^n,
  2^{n+1})\big)^2\Big)^{1/2}\bigg\|_{\ell^p}\\
+\bigg\|\Big(\sum_{n\ge0} V_2\big((M_{N}-M_{2^n})\mathcal
  F^{-1}\big(\Xi_{n}^n\hat{f}\big): N\in[2^n,
  2^{n+1})\big)^2\Big)^{1/2}\bigg\|_{\ell^p}=I_p^1+I_p^2.
\end{multline*}
We will estimate $I_p^1$ and $I_p^2$ separately. First, observe that by \eqref{eq:61} and \eqref{eq:64},
for any $N\simeq 2^n$ and any $a/q\in\mathscr U_{n^l}$ we have
\begin{align}
\label{eq:207}
	\big|
	m_{N}(\xi)-m_{2^n}(\xi)
	\big|
	&\lesssim
	q^{-\delta} \big|\Phi_N(\xi - a/q) - \Phi_{2^n}(\xi - a/q)\big|
	+ q^{-\delta} 2^{-n/2}\\
\nonumber&\lesssim q^{-\delta}\big(\min\big\{1, |2^{nA}(\xi-a/q)|_{\infty},
|2^{nA}(\xi-a/q)|_{\infty}^{-1/d}\big\}+2^{-n/2}\big)  
\end{align}
where the last bound follows from \eqref{eq:9} and \eqref{eq:10}. 
Thus
 using  \eqref{eq:207}  we get
\begin{align}
  \label{eq:66}
 \big|\big(m_{N}(\xi)-m_{2^n}(\xi)\big)
\big(\eta\big(2^{nA+jI}(\xi-a/q)\big)-\eta\big(2^{nA+(j+1)I}(\xi-a/q)\big)\big)\big|
\lesssim
\big(2^{-|j|/d}+2^{-n/2}\big).
\end{align}
We begin with bounding $I_p^2$. Since $r$-variations are decreasing
we can assume that $1<r\le \min\{2, p\}$ and it will suffice to show, for some
$\varepsilon=\varepsilon_{p, r}>0$, that
\begin{align}
  \label{eq:102}
  \big\|V_r\big((M_{N}-M_{2^n})\mathcal
  F^{-1}\big(\Xi_{n}^{n}\hat{f}\big): N\in[2^n,
  2^{n+1})\big)\big\|_{\ell^p}\lesssim 2^{-\varepsilon
    n}\|f\|_{\ell^p}.
\end{align} 
Likewise above we shall exploit  \eqref{eq:39}
which immediately gives
\[
  \big\|V_r\big((M_{N}-M_{2^n})\mathcal
  F^{-1}\big(\Xi_{n}^{n}\hat{f}\big): N\in[2^n,
  2^{n+1})\big)\big\|_{\ell^p}
\lesssim\max\big\{\mathbf U_p, 2^{n/r}\mathbf U_p^{1-1/r}\mathbf V_p^{1/r}\big\}\
\]
where 
\begin{align*}
  \mathbf U_p=\sup_{2^n\le N\le 2^{n+1}}\big\|(M_{N}-M_{2^n})\mathcal
  F^{-1}\big(\Xi_{n}^{n}\hat{f}\big)\big\|_{\ell^p}
\end{align*}
and 
\begin{align*}
\mathbf V_p=  \sup_{2^n\le N< 2^{n+1}}\big\|(M_{N+1}-M_{N})\mathcal
  F^{-1}\big(\Xi_{n}^{n}\hat{f}\big)\big\|_{\ell^p}.
\end{align*}
In view of Theorem \ref{th:3} we see that
\begin{align}
  \label{eq:104}
 \mathbf U_p\lesssim \log(n+2)\|f\|_{\ell^p} \quad \text{and} \quad
  \mathbf V_p\lesssim 2^{-n}\log(n+2)\|f\|_{\ell^p}.
\end{align}
 For $p=2$ by Plancherel's theorem and  \eqref{eq:207} we obtain
 \begin{multline}
   \label{eq:105}
   \big\|(M_{N}-M_{2^n})\mathcal
  F^{-1}\big(\Xi_{n}^{n}\hat{f}\big)\big\|_{\ell^2}\\
=
\Big(\int_{\TT^d}\sum_{a/q\in\mathscr U_{n^l}}|m_{N}(\xi)-m_{2^n}(\xi)|^2\eta\big(2^{nA+nI}(\xi-a/q)\big)^2
  |\hat{f}(\xi)|^2 {\: \rm d} \xi\Big)^{1/2}
\lesssim 
2^{-n/(2d)}\|f\|_{\ell^2}.
 \end{multline}
Therefore, interpolating \eqref{eq:104} with \eqref{eq:105} we obtain
for every $p\in(1, \infty)$ that 
\begin{align*}
  \mathbf U_p\lesssim 2^{-\varepsilon n}\|f\|_{\ell^p}
\end{align*}
which in turn implies \eqref{eq:102} and $I_p^2\lesssim \|f\|_{\ell^p}$.

We shall now estimate $I_p^1$, for this purpose we need to define new multipliers for any $0\le s< n$
\[
  \Delta_{n, s}^j(\xi)=\sum_{a/q\in\mathscr
    U_{(s+1)^l}\setminus\mathscr
    U_{s^l}}\big(\eta\big(2^{nA+jI}(\xi-a/q)\big)-\eta\big(2^{nA+(j+1)I}(\xi-a/q)\big)\big)\eta\big(2^{s(A-\chi
  I)}(\xi-a/q)\big),
\]
It makes sense since $\mathscr U_{s^l}\subseteq \mathscr
U_{(s+1)^l}$, thus
\begin{align*}
  \Xi_{n}^j(\xi)-\Xi_{n}^{j+1}(\xi)=\sum_{0\le s<n}\Delta_{n, s}^j(\xi).
\end{align*}
Now we have  
\begin{multline*}
  I_p^1=\bigg\|\Big(\sum_{n\ge0} V_2\big((M_{N}-M_{2^n})\mathcal
  F^{-1}\Big(\sum_{-\chi n\le j< n}\sum_{0\le s<n}\Delta_{n, s}^j\hat{f}\Big): N\in[2^n,
  2^{n+1})\big)^2\Big)^{1/2}\bigg\|_{\ell^p}\\
\le\sum_{s\ge 0}\sum_{j\in\ZZ}\bigg\|\Big(\sum_{n\ge \max\{s, j, -j/\chi\}} V_2\big((M_{N}-M_{2^n})\mathcal
  F^{-1}\big(\Delta_{n, s}^j\hat{f}\big): N\in[2^n,
  2^{n+1})\big)^2\Big)^{1/2}\bigg\|_{\ell^p}.
\end{multline*}
The task now is to show that for some $\varepsilon_p>0$ 
\begin{align}
  \label{eq:73}
  J_p=\bigg\|\Big(\sum_{n\ge \max\{s, j, -j/\chi\}} V_2\big((M_{N}-M_{2^n})\mathcal
  F^{-1}\big(\Delta_{n, s}^j\hat{f}\big): N\in[2^n,
  2^{n+1})\big)^2\Big)^{1/2}\bigg\|_{\ell^p}\lesssim
  s^{-2}2^{-\varepsilon_p|j|}\|f\|_{\ell^p}.
\end{align}
 Before we establish \eqref{eq:73} we
need some prerequisites. 
\subsubsection{Some preparatory estimates} 
The next result, more precisely inequality \eqref{eq:83}, can be thought as a discrete counterpart of
Littlewood--Paley theory.
\begin{theorem}
  \label{thm:30}
For every
$p\in(1, \infty)$ there is a
constant $C>0$ such that for all $f\in\ell^p\big(\ZZ^d\big)$ we have
 \begin{align}
   \label{eq:83}
   \bigg\|\Big(\sum_{n\ge \max\{s, j, -j/\chi\}}\big|\mathcal
  F^{-1}\big(\Delta_{n,
    s}^j\hat{f}\big)\big|^2\Big)^{1/2}\bigg\|_{\ell^p}\le C \log(s+2)\|f\|_{\ell^p}.
 \end{align}
\end{theorem}
\begin{proof}
  
By Khinchine's inequality \eqref{eq:83} is equivalent to the following
\begin{align}
  \label{eq:95}
  \Big(\int_0^1\bigg\|\sum_{n\ge \max\{s, j, -j/\chi\}}\varepsilon_n(t)\mathcal
  F^{-1}\big(\Delta_{n,
    s}^j\hat{f}\big)\bigg\|_{\ell^p}^p {\: \rm d}t\Big)^{1/p}\lesssim \log(s+2)\|f\|_{\ell^p}.
\end{align}
Indeed,  the multiplier from \eqref{eq:95} can be
rewritten as follows
\begin{align*}
 \sum_{n\ge \max\{s, j, -j/\chi\}}\varepsilon_j(t) \Delta_{n, s}^j(\xi)=\Big(\sum_{a/q\in\mathscr
    U_{(s+1)^k}}-\sum_{a/q\in\mathscr
    U_{s^k}}\Big)\sum_{n\ge \max\{s, j, -j/\chi\}}\mathfrak
  m_n(\xi-a/q)\eta\big(2^{s(A-\chi I)}(\xi-a/q)\big)
\end{align*}
with the functions 
\[
\mathfrak m_n(\xi)=\varepsilon_j(t)\big(\eta\big(2^{nA+jI}\xi\big)-
\eta\big(2^{nA+(j+1)A}\xi\big)\big).
\]
We observe that
\[
|\mathfrak m_n(\xi)|\lesssim\min\big\{|2^{nA+jI}\xi|_{\infty}, |2^{nA+jI}\xi|_{\infty}^{-1}\big\}. 
\] 
The first bound follows from the mean-value theorem, since
\[
\big|\eta\big(2^{nA+jI}\xi\big)-
\eta\big(2^{nA(j+1)I}\xi\big)\big|\lesssim\big|2^{nA+jI}\xi-2^{nA+(j+1)I}\xi\big|\sup_{\xi\in[-1,
  1]^d}\big|\nabla
\eta(\xi)\big| \lesssim |2^{nA+jI}\xi|_{\infty}.
\]
The second bound follows since $\eta$ is a Schwartz
function. Moreover, for every $p\in(1, \infty)$ we have
\[
\big\|\sup_{n\in\NN_0}\big|\calF^{-1}\big(\mathfrak m_n\calF f\big)\big|\big\|_{L^p}\lesssim \|f\|_{L^p}
\]
for every $f\in L^p\big(\RR^d\big)$. Therefore, by \cite{bigs} the
 multiplier  
\[
\sum_{n\ge \max\{s, j, -j/\chi\}}\mathfrak m_n(\xi)
\]
defines a bounded operator 
on $L^p\big(\RR^d\big)$ for all
$p\in(1, \infty)$. Hence, Theorem \ref{th:3} applies and one can see that 
the multiplier 
\[
\sum_{n\ge
  \max\{s, j, -j/\chi\}}\varepsilon_n(t)\Delta_{n, s}^j(\xi)
\]
 defines a
bounded  operator on $\ell^p\big(\ZZ^d\big)$ with the
logarithmic loss with respect to $s$, and \eqref{eq:95} is
established.
  
\end{proof}

    To estimate \eqref{eq:73} we will use
\eqref{eq:1} from Lemma \ref{lem:8} with $r=2$. Namely, 
let $h_{j, s}=2^{\varepsilon |j|}(s+1)^{\tau}$ with some $\varepsilon>0$ and $\tau>2$
which we choose later. Then for some
$2^n\le t_0<t_1<\ldots<t_{h}<2^{n+1}$ such that $t_{v+1}-t_v\simeq
2^n/h$ where $h=\min\{h_{j, s}, 2^n\}$ we have
\begin{multline*}
  J_p\lesssim \bigg\|\Big(\sum_{n\ge \max\{s, j, -j/\chi\}}\sum_{v=0}^{h}\big|(M_{2^{t_v}}-M_{2^n})\mathcal
  F^{-1}\big(\Delta_{n, s}^j\hat{f}\big)\big|^2\Big)^{1/2}\bigg\|_{\ell^p}\\
+\bigg\|\bigg(\sum_{n\ge \max\{s, j, -j/\chi\}}\sum_{v=0}^{h-1}\Big(\sum_{u=t_v}^{t_{v+1}-1}\big|(M_{u+1}-M_{u})\mathcal
  F^{-1}\big(\Delta_{n, s}^j\hat{f}\big)\big|\Big)^2\bigg)^{1/2}\bigg\|_{\ell^p}=J_p^1+J_p^2.
\end{multline*}
\subsubsection{The estimates for $J_p^1$}  We begin with $p=2$ and
show that
\begin{align}
  \label{eq:14}
J_2^1\lesssim 2^{-|j|(1/(2d)-\varepsilon/2)}(s+1)^{-\delta l+\tau/2}\|f\|_{\ell^2}.  
\end{align}
For the simplicity of notation define
\[
\varrho_{n, j}(\xi)=\big(\eta\big(2^{nA+jI}\xi\big)-
\eta\big(2^{nA+(j+1)I}\xi\big)\big)
\eta\big(2^{s(A-\chi I)}\xi\big).
\] 
By Plancherel's theorem and \eqref{eq:66}  we have
\begin{multline*}
  J_2^1=\Big(\sum_{v=0}^{h}\int_{\TT^d}\sum_{a/q\in\mathscr
    U_{(s+1)^l}\setminus\mathscr
    U_{s^l}}\sum_{n\ge \max\{s, j, -j/\chi\}}|m_{2^{t_v}}(\xi)-m_{2^n}(\xi)|^2
\varrho_{n,j}(\xi-a/q)^2|\hat{f}(\xi)|^2 {\: \rm d}\xi\Big)^{1/2}\\
\lesssim 
\Big(\sum_{v=0}^{h}2^{-|j|/d}\int_{\TT^d}\sum_{a/q\in\mathscr
    U_{(s+1)^l}\setminus\mathscr
    U_{s^l}}q^{-2\delta}\eta\big(2^{s(A-\chi I)}(\xi-a/q)\big))|\hat{f}(\xi)|^2 {\: \rm d}\xi\big)^{1/2}\\
\lesssim h^{1/2}2^{-|j|/(2d)}\Big(\int_{\TT^d}(s+1)^{-2\delta l}\sum_{a/q\in\mathscr
    U_{(s+1)^l}\setminus\mathscr
    U_{s^l}}\eta\big(2^{s(A-\chi I)}(\xi-a/q)\big)|\hat{f}(\xi)|^2 {\: \rm d}\xi\Big)^{1/2}\\
\lesssim h^{1/2}2^{-|j|/(2d)}(s+1)^{-\delta l}\|f\|_{\ell^2}
\end{multline*}
as desired. We have used the fact that $q\gtrsim (s+1)^l$ whenever
 $a/q\in\mathscr U_{(s+1)^l}\setminus\mathscr U_{s^l}$ 
 and
\[ 
  \sum_{n\ge \max\{s, j, -j/\chi\}}\varrho_{n, j}(\xi)\lesssim
\eta\big(2^{s(A-\chi I)}(\xi-a/q)\big)
\]
and the disjointness of $\eta\big(2^{s(A-\chi I)}( \cdot-a/q)\big)$
while $a/q$ varies over $\mathscr U_{(s+1)^k}\setminus\mathscr
U_{s^k}$.

Moreover, for any  $p\in(1, \infty)$ we have
\begin{align}
  \label{eq:16}
  J_p^1\lesssim 2^{\varepsilon |j|/2}(s+1)^{\tau/2}\log(s+2)\|f\|_{\ell^p}.
\end{align}
Indeed, 
 appealing to the vector-valued
inequality for the maximal function corresponding to the  averaging
operators from \cite{mst1} we see that  
\begin{multline*}
  J_p^1=\bigg\|\Big(\sum_{n\ge \max\{s, j, -j/\chi\}}\sum_{v=0}^{h}\big|(M_{2^{t_v}}-M_{2^n})\mathcal
  F^{-1}\big(\Delta_{n,
    s}^j\hat{f}\big)\big|^2\Big)^{1/2}\bigg\|_{\ell^p}\\
\lesssim h^{1/2}
\bigg\|\Big(\sum_{n\ge \max\{s, j, -j/\chi\}}\sup_{N\in\NN}M_{N}\big(\big|\mathcal
  F^{-1}\big(\Delta_{n,
    s}^j\hat{f}\big)\big|\big)^2\Big)^{1/2}\bigg\|_{\ell^p}\\
\lesssim h^{1/2}
\bigg\|\Big(\sum_{n\ge \max\{s, j, -j/\chi\}}\big|\mathcal
  F^{-1}\big(\Delta_{n,
    s}^j\hat{f}\big)\big|^2\Big)^{1/2}\bigg\|_{\ell^p}\lesssim h^{1/2}\log(s+2)\|f\|_{\ell^p}.
\end{multline*}
In the last step we have used \eqref{eq:83}. Interpolating now
\eqref{eq:16} with better \eqref{eq:14} estimate, we obtain for some $\varepsilon_p>0$
that
\[
  J_p^1\lesssim (s+1)^{-2}2^{-\varepsilon_p |j|}\|f\|_{\ell^p}.
\]
The proof of \eqref{eq:73} will be completed if we obtain the same
kind of bound for $J_p^2$.
\subsubsection{The estimates for $J_p^2$} 
We begin with $p=2$ and our aim will be to show 
\begin{align}
  \label{eq:17}
J_2^2\lesssim 2^{-\varepsilon
  |j|/2}(s+1)^{-\tau/2}\|f\|_{\ell^2}.  
\end{align}
Since $t_{v+1}-t_v\simeq 2^n/h$ then by the Cauchy--Schwarz inequality
we obtain
\[
  J_2^2\le \bigg\|\bigg(\sum_{n\ge \max\{s, j, -j/\chi\}}2^n/h\sum_{u=2^n}^{2^{n+1}-1}\big|(M_{u+1}-M_{u})\mathcal
  F^{-1}\big(\Delta_{n,
    s}^j\hat{f}\big)\big|^2\bigg)^{1/2}\bigg\|_{\ell^2}.
\]
By \eqref{eq:66} we have for $u\simeq2^n$ that
\begin{align}
  \label{eq:99}
  |m_{u+1}(\xi)-m_{u}(\xi)|\lesssim \min\big\{2^{-n}, q^{-\delta}(2^{-|j|/d}+2^{-n/2})\big\}.
\end{align}
Two cases must be distinguished. Assume now that $h=2^{\varepsilon
  |j|}(s+1)^{\tau}$, therefore, again by Plancherel's theorem,  we obtain 
\begin{multline*}
  J_2^2\le\Big(\sum_{n\ge \max\{s, j, -j/\chi\}}\int_{\TT^d}2^n/h\sum_{u=2^n}^{2^{n+1}-1}\sum_{a/q\in\mathscr
    U_{(s+1)^k}\setminus\mathscr
    U_{s^k}}|m_{u+1}(\xi)-m_{u}(\xi)|^2
\varrho_{n,j}(\xi-a/q)^2|\hat{f}(\xi)|^2 {\: \rm d}\xi\Big)^{1/2}\\
\lesssim
h^{-1/2}\Big(\int_{\TT^d}\sum_{a/q\in\mathscr
    U_{(s+1)^k}\setminus\mathscr
    U_{s^k}}\sum_{n\ge \max\{s, j, -j/\chi\}}
\varrho_{n,j}(\xi-a/q)|\hat{f}(\xi)|^2 {\: \rm d}\xi\Big)^{1/2}\lesssim h^{-1/2}\|f\|_{\ell^2}
\end{multline*}
since by the telescoping nature and the disjointness of supports  when $a/q$ varies over $\mathscr U_{(s+1)^k}\setminus\mathscr
  U_{s^k}$ we have
\[
\sum_{a/q\in\mathscr U_{(s+1)^k}\setminus\mathscr
  U_{s^k}}\sum_{n\ge \max\{s, j, -j/\chi\}} \varrho_{n,j}(\xi-a/q)\lesssim \sum_{a/q\in\mathscr U_{(s+1)^k}\setminus\mathscr
  U_{s^k}}\eta\big(2^{s(A-\chi I)}(\xi-a/q)\big)\lesssim 1.
\]
If $h=2^n$ then by \eqref{eq:99} we get
\begin{multline*}
  J_2^2\le\Big(\sum_{n\ge \max\{s, j, -j/\chi\}}\int_{\TT^d}\sum_{u=2^n}^{2^{n+1}-1}\sum_{a/q\in\mathscr
    U_{(s+1)^k}\setminus\mathscr
    U_{s^k}}|m_{u+1}(\xi)-m_{u}(\xi)|^2
\varrho_{n,j}(\xi-a/q)^2|\hat{f}(\xi)|^2d\xi\Big)^{1/2}\\
\lesssim
2^{-|j|/4}2^{-s/4}\Big(\int_{\TT^d}\sum_{a/q\in\mathscr
    U_{(s+1)^k}\setminus\mathscr
    U_{s^k}}\sum_{n\ge \max\{s, j, -j/\chi\}}
\varrho_{n,j}(\xi-a/q)|\hat{f}(\xi)|^2 {\: \rm d}\xi\Big)^{1/2} \\
\lesssim 2^{-|j|/4}2^{-s/4}\|f\|_{\ell^2}
\end{multline*}
and \eqref{eq:17} is proven.

For $p\in(1, \infty)$ we shall prove that 
\begin{align}
  \label{eq:21}
  J_p^2\lesssim \log(s+2)\|f\|_{\ell^p}. 
\end{align}
Observe that
\begin{multline*}
\sum_{u=2^n}^{2^{n+1}-1}| K_{u+1}- K_u|\lesssim \sum_{u=2^n}^{2^{n+1}-1} 
\Big(\frac{1}{(u+1)^k}\sum_{y\in\BB_{u+1}\setminus\BB_u}\delta_{\mathcal
  Q(y)}
+\frac{1}{u(u+1)^k}\sum_{y\in\BB_u}\delta_{\mathcal Q(y)}\Big)\\
\lesssim\frac{1}{2^{nk}}\sum_{y\in\BB_{2^{n+1}}\setminus\BB_{2^n}}\delta_{\mathcal
  Q(y)}+\frac{1}{2^{(n+1)k}}\sum_{y\in\BB_{2^{n+1}}}\delta_{\mathcal
  Q(y)}\lesssim  K_{2^{n+1}}.
\end{multline*}
This in turn implies 
\begin{multline*}
  J_p^2=\bigg\|\bigg(\sum_{n\ge \max\{s, j, -j/\chi\}}\sum_{v=0}^{h-1}\Big(\sum_{u=t_v}^{t_{v+1}-1}\big|(M_{u+1}-M_{u})\mathcal
  F^{-1}\big(\Delta_{n,
    s}^j\hat{f}\big)\big|\Big)^2\bigg)^{1/2}\bigg\|_{\ell^p}\\
\le\bigg\|\bigg(\sum_{n\ge \max\{s, j, -j/\chi\}}\Big(\sum_{u=2^n}^{2^{n+1}-1}\big|(K_{u+1}-K_{u})\big|*\big(\big|\mathcal
  F^{-1}\big(\Delta_{n,
    s}^j\hat{f}\big)\big|\big)\Big)^2\bigg)^{1/2}\bigg\|_{\ell^p}\\
  \lesssim\bigg\|\bigg(\sum_{n\ge \max\{s, j, -j/\chi\}}\sup_{N\in\NN}M_{N}\big(\big|\mathcal
  F^{-1}\big(\Delta_{n,
    s}^j\hat{f}\big)\big|\big)^2\bigg)^{1/2}\bigg\|_{\ell^p}\\
\lesssim \bigg\|\bigg(\sum_{n\ge \max\{s, j, -j/\chi\}}\big|\mathcal
  F^{-1}\big(\Delta_{n,
    s}^j\hat{f}\big)\big| ^2\bigg)^{1/2}\bigg\|_{\ell^p}\lesssim \log(s+2)\|f\|_{\ell^p}.
\end{multline*}
In the penultimate line we have used vector-valued maximal
estimates corresponding to the averaging operators from \cite{mst1} and in the last line we invoked \eqref{eq:83}.
Interpolating now the estimate \eqref{eq:21} with the estimate from
\eqref{eq:17} we obtain for some $\varepsilon_p>0$
that
\[
  J_p^2\lesssim (s+1)^{-2}2^{-\varepsilon_p |j|}\|f\|_{\ell^p}
\]
and the proof of \eqref{eq:73} is completed.

\section{Long variation estimates for truncated singular integral operators}
\label{sec:6}

For any function $f: \ZZ^d \rightarrow \CC$ with a finite support we have
$$
T_N f(x) = H_N * f(x)
$$
with a kernel $H_N$ defined by
\[
	H_N(x) = \sum_{y \in \BB_N\setminus\{0\}} \delta_{\calQ(y)}K(y)
\]
where $K$ is the kernel as in \eqref{eq:40}
and $\delta_y$ denotes Dirac's delta at $y \in \ZZ^k$ and $\calQ$ is
the canonical polynomial, see Section \ref{sec2}. Let $m_N$ denote the discrete  Fourier transform of $H_N$, i.e.
$$
m_N(\xi) = \sum_{y \in \BB_N\setminus\{0\}} e^{2\pi i \sprod{\xi}{\calQ(y)}}K(y).
$$
Finally, we define
$$
\Psi_t(\xi) = {\rm p.v.}\int_{B_t} e^{2\pi i \sprod{\xi}{\calQ(y)}} K(y) {\: \rm d}y
$$
where $B_t$ is the Euclidean ball in $\RR^k$ centered at the origin
with radius $t>0$.  Using the method of the proof of the
multi-dimensional version of van der Corput lemma in \cite{sw} we
may estimate
\begin{equation}
\label{eq:48}
	\abs{\Psi_{N}(\xi)-\Psi_{cN}(\xi)}=\Big|\int_{B_{1}\setminus
        B_{c}} e^{2\pi i
          \sprod{\xi}{\calQ(Ny)}}N^k K(Ny){\: \rm d}y\Big|
	\lesssim
	\min\big\{1, \norm{N^A \xi}_{\infty}^{-1/d} \big\}
\end{equation}
with the implicit  constant depending on $c\in(0, 1)$.
Additionally, we have
\begin{equation}
	\label{eq:49}
	\abs{\Psi_{N}(\xi)-\Psi_{cN}(\xi)}
	\lesssim
	\norm{N^A \xi}_{\infty}
\end{equation}
due to cancellation condition \eqref{eq:24}. We shall prove that for every $p\in(1, \infty)$
and $r\in(2, \infty)$ there is
$C_{p, r} > 0$ such that for all $f \in \ell^p\big(\ZZ^{d_0}\big)$
and $f\ge0$ we have
\[
          	\big\lVert
	V_r\big(  T_{2^n} f: n\in\NN_0\big)
	\big\rVert_{\ell^p}\le
	C_{p, r}\|f\|_{\ell^p}
\]
and $C_{p, r}\le C_p\frac{r}{r-2}$ for some $C_p>0$.
We begin with proving the following, which is a variant of  Proposition \ref{prop:0}.
\begin{proposition}
  \label{prop:2}
 There is a constant $C>0$ such
  that for every $N\in\NN$, $M\in\NN$ such that $cN\le M\le N$ for
  some $c>0$ and for every $\xi\in [-1/2, 1/2)^d$ satisfying 
        $$
	\Big\lvert \xi_\gamma - \frac{a_\gamma}{q} \Big\rvert \leq
        L_1^{-|\gamma|}L_2
	$$
	for all $\gamma \in \Gamma$, where  $1\le q\le L_3\le N^{1/2}$, $a\in
        A_q$, $L_1\ge N$  and $L_2\ge1$ we have
	\[
         \big|m_N(\xi)-m_M(\xi)-G(a/q)\big(\Psi_{N}(\xi-a/q)-\Psi_{M}(\xi-a/q)\big)\big|
		\le C\Big(L_3/N+L_2L_3/N\sum_{\gamma \in \Gamma}\big(N/L_1\big)^{|\gamma|}\Big).
	\]
\end{proposition}
\begin{proof}
	Let $\theta = \xi - a/q$. For any $r \in \NN_q^k$, if $y \equiv r \pmod q$ then
	for each $\gamma \in \Gamma$
	$$
	\xi_\gamma y^\gamma \equiv \theta_\gamma y^\gamma
	+ (a_\gamma/q) r^\gamma \pmod 1,
	$$
	thus
	$$
	\sprod{\xi}{\calQ(y)} \equiv \sprod{\theta}{\calQ(y)} + \sprod{(a/q)}{\calQ(r)} \pmod 1.
	$$
	Therefore,
	\[
	\sum_{y \in \BB_N\setminus\BB_M} e^{2\pi i \sprod{\xi}{\calQ(y)}} K(y)
	=
	q^{-k}\sum_{r \in \NN_q^k}
	e^{2\pi i \sprod{(a/q)}{\calQ(r)}}
	\cdot \Big(q^k\sum_{\atop{y \in
            \ZZ^k}{qy+r\in \BB_N\setminus\BB_M}}
	e^{2\pi i \sprod{\theta}{\calQ(qy+r)}}K(qy+r)\Big).
	\]
	If $ q y + r\in \BB_N\setminus\BB_M$ then
	\[
	\big\lvert
	\sprod{\theta}{\calQ(q y + r)} - \sprod{\theta}{\calQ(q y)}
	\big\rvert
	\lesssim
	\norm{r}
	\sum_{\gamma \in \Gamma}
	\abs{\theta_\gamma}
	\cdot
	N^{(\abs{\gamma} - 1)}
	\lesssim
	q \sum_{\gamma \in \Gamma}
	L_1^{-\abs{\gamma}}L_2 N^{(\abs{\gamma}-1)}
	\lesssim
	L_2L_3/N\sum_{\gamma \in \Gamma}\big(N/L_1\big)^{\abs{\gamma}}
	\]
and
	\[
	\big\lvert
	K(q y + r) - K(q y)
	\big\rvert
	\lesssim
	N^{-(k+1)}q.
        \]
	Thus
        \begin{align*}
          q^k\sum_{\atop{y \in
            \ZZ^k}{qy+r\in \BB_N\setminus\BB_M}}
	e^{2\pi i \sprod{\theta}{\calQ(qy+r)}}K(qy+r)&=
q^k\sum_{\atop{y \in
            \ZZ^k}{qy+r\in \BB_N\setminus\BB_M}}
	e^{2\pi i \sprod{\theta}{\calQ(qy)}}K(qy)\\
\nonumber&+\mathcal O\Big(q/N+L_2L_3/N\sum_{\gamma \in \Gamma}\big(N/L_1\big)^{|\gamma|}\Big).
        \end{align*}
Now  we see that
\[
q^k\sum_{\atop{y \in
            \ZZ^k}{qy+r\in \BB_N\setminus\BB_M}}
	e^{2\pi i \sprod{\theta}{\calQ(qy)}}K(qy)=q^k\sum_{\atop{y \in
            \ZZ^k}{qy\in \BB_N\setminus\BB_M}}
	e^{2\pi i \sprod{\theta}{\calQ(qy)}}K(qy)+\mathcal O\big(q/N\big).
\]
Thus
        \[
        \sum_{y \in \BB_N\setminus\BB_M} e^{2\pi i \sprod{\xi}{\calQ(y)}} K(y)
	=G(a/q)\cdot q^k\sum_{\atop{y \in \ZZ^k}{qy\in\BB_N\setminus\BB_M}}
	e^{2\pi i \sprod{\theta}{\calQ(qy)}}K(qy)+\mathcal O\Big(q/N+L_2L_3/N\sum_{\gamma \in \Gamma}\big(N/L_1\big)^{|\gamma|}\Big).
        \]
Now we are going to replace the  exponential sum on the right-hand
side of the last display by the integral. By the mean value theorem,  we  obtain
  \begin{multline*}
    q^k\Big|\sum_{\atop{y \in \ZZ^k}{
        qy\in\BB_N\setminus\BB_M}}
    e^{2\pi i \sprod{\theta}{\calQ(qy)}}K(qy)-\int_{B_{N/q}\setminus B_{M/q}}e^{2\pi i \theta\cdot\calQ(qt)}K(qt){\: \rm d}t\Big|\\
    = q^k\Big|\sum_{y \in \ZZ^k } e^{2\pi i
      \sprod{\theta}{\calQ(qy)}}K(qy)\ind{\BB_N\setminus\BB_M}(qy)-\sum_{y \in \ZZ^k
        }\int_{y+(0, 1]^k}e^{2\pi i \theta\cdot\calQ(qt)}K(qt)\ind{B_{N/q}\setminus B_{M/q}}(t){\: \rm d}t\Big|\\
=q^k\Big|\sum_{y \in \ZZ^k} \int_{(0, 1]^k}e^{2\pi i
      \theta\cdot\calQ(qy)}K(qy)\ind{B_{N}\setminus B_{M}}(qy)-e^{2\pi i
      \theta\cdot\calQ(q(t+y))}K(q(t+y))\ind{B_{N}\setminus B_{M}}(q(y+t)){\:\rm d}t\Big|\\
=\mathcal O\Big(q/N+L_2L_3/N\sum_{\gamma \in \Gamma}\big(N/L_1\big)^{|\gamma|}\Big).
  \end{multline*}
This completes the proof of Proposition \ref{prop:2}.
\end{proof}

In Remark \ref{rem:10} we mentioned that 
Theorem \ref{thm:4} holds with the operators $T_N$ defined with the sets
$\mathbb G_N$ instead of $\mathbb B_N$. Then we obtain analogous
definitions of
$H_N$, $m_N$ and $\Psi_N$ with the sets $\mathbb G_N$, $G_1=G$ and $G_t$ in place
of the sets $\mathbb B_N$, $B_1$ and $B_t$ respectively. All of the arguments
remain unchanged apart from the proof of Proposition
\ref{prop:2}.  However,  \cite[Proposition
3.1]{mst1} used  with the sets $\mathbb G_N$ allows us to follow the
same scheme and we obtain conclusion of the same type.

\medskip

As in the previous sections  fix the numbers $\chi>0$ and $l\in\NN$ whose precise
values will be chosen later, and let us consider for every $n\in\NN_0$
the multipliers
\begin{align}
\label{eq:53}
  \Xi_{n}(\xi)=\sum_{a/q\in\mathscr U_{n^l}}\eta\big(2^{n(A-\chi I)}(\xi-a/q)\big)
\end{align}
with $\mathscr U_{n^l}$ defined as in \eqref{eq:156}.
Theorem \ref{th:3} guarantees  that for every $p\in(1, \infty)$ 
\begin{align}
\label{eq:55}
  \big\|\mathcal
  F^{-1}\big(\Xi_{n}\hat{f}\big)\big\|_{\ell^p}\lesssim \log(n+2)\|f\|_{\ell^p}.
\end{align}
The implicit constant in \eqref{eq:55} depends on the parameter
$\rho>0$, see Section \ref{sec:3}. However, from
now on we will assume that $\rho>0$ and the integer $l\ge10$ are related by the
equation 
\begin{align}
  \label{eq:7}
  10\rho l=1.
\end{align}
Observe that
\begin{multline}
\label{eq:82}
\big\|V_r\big(T_{2^n}f: n\in\NN_0\big)\big\|_{\ell^p}
\le
\Big\|V_r\Big(\mathcal
  F^{-1}\Big(\sum_{j=1}^n(m_{2^j}-m_{2^{j-1}})\Xi_{j}\hat{f}\Big): n\in\NN_0\Big)\Big\|_{\ell^p} \\
+ \Big\|V_r\Big(\mathcal
  F^{-1}\Big(\sum_{j=1}^n(m_{2^j}-m_{2^{j-1}})(1-\Xi_{j})\hat{f}\Big): n\in\NN_0\Big)\Big\|_{\ell^p}.
\end{multline}
\subsection{The estimate for the second norm in \eqref{eq:82}}
Since the variational norm is increasing when $r$ decreases we get
\[
\Big\|V_1\Big(\mathcal
  F^{-1}\Big(\sum_{j=1}^n(m_{2^j}-m_{2^{j-1}})(1-\Xi_{j})\hat{f}\Big):
  n\in\NN_0\Big)\Big\|_{\ell^p}
 \le\sum_{n\in\NN_0}
\big\|\mathcal
  F^{-1}\big((m_{2^n}-m_{2^{n-1}})(1-\Xi_{n})\hat{f}\big)\big\|_{\ell^p}.
\]
Therefore, it suffices to show that
\begin{align}
  \label{eq:86}
  \big\|\mathcal
  F^{-1}\big((m_{2^n}-m_{2^{n-1}})(1-\Xi_{n})\hat{f}\big)\big\|_{\ell^p}\le
  (n+1)^{-2}
  \|f\|_{\ell^p}.
\end{align}
For every $1<p<\infty$ we have 
\begin{align}
\label{eq:98}
  \big\|\mathcal
  F^{-1}\big((m_{2^n}-m_{2^{n-1}})(1-\Xi_{n})\hat{f}\big)\big\|_{\ell^p}\lesssim
\big\|
  M_{2^n}f\big\|_{\ell^p}+\big\| M_{2^n}\big(
  \mathcal
  F^{-1}\big(\Xi_{n}\hat{f}\big)\big)\big\|_{\ell^p}\lesssim \log(n+2)\|f\|_{\ell^p}
\end{align}
since for $f\ge0$ we have a pointwise bound
\[
\big|\mathcal
  F^{-1}\big((m_{2^n}-m_{2^{n-1}})\hat{f}\big)(x)\big|\lesssim M_{2^n}f(x)
\]
where $M_N$ is the averaging operator from Section \ref{sec:4}.
In fact we  improve estimate
\eqref{eq:98} 
for $p=2$. 
Indeed, we will show that for big enough $\alpha>0$, which will be specified
later, and for all $n\in\NN_0$  we have
\begin{align}
  \label{eq:63}
 \big|
 (m_{2^n}(\xi)-m_{2^{n-1}}(\xi))(1-\Xi_{n}(\xi))\big|\lesssim (n+1)^{-\alpha}.
\end{align}
This estimate will be a consequence of Theorem \ref{thm:3}. For do so,
by Dirichlet's principle we have for every $\gamma\in\Gamma$ 
\[
\bigg|\xi_{\gamma}-\frac{a_{\gamma}}{q_{\gamma}}\bigg|\le \frac{n^{\beta}}{q_{\gamma}2^{n|\gamma|}}
\] 
where $1\le q_{\gamma}\le n^{-\beta}2^{n|\gamma|}$. In order to apply
Theorem \ref{thm:3} we must show that there exists some
$\gamma\in\Gamma$ such that $n^{\beta}\le q_{\gamma }\le
n^{-\beta}2^{n|\gamma|}$. Suppose for a contradiction that for every
$\gamma \in \Gamma$  we have $1\le q_{\gamma }<n^{\beta} $ then for
some 
$q\le \mathrm{lcm}(q_{\gamma}: \gamma\in\Gamma)\le n^{\beta d}$ we have 
\[
\bigg|\xi_{\gamma}-\frac{a_{\gamma}'}{q}\bigg|\le \frac{n^{\beta}}{2^{n|\gamma|}}
\] 
where $\mathrm{gcd}\big(q, \mathrm{gcd}({a_{\gamma}'}:
\gamma\in\Gamma)\big)=1$.
Hence, taking $a'=(a_{\gamma}': \gamma\in\Gamma)$ we have
$a'/q\in\mathscr U_{n^l}$ provided that $\beta d<l$. On the
other hand, if $1-\Xi_{n}(\xi)\not=0$ then for every
$a'/q\in\mathscr U_{n^l}$ there exists $\gamma\in\Gamma$ such that
\[
\bigg|\xi_{\gamma}-\frac{a_{\gamma}'}{q}\bigg|> \frac{1}{16d\cdot2^{n(|\gamma|-\chi)}}.
\] 
Therefore, one obtains
\[
2^{\chi n}<16dn^{\beta}
\] 
 but this gives a contradiction, for sufficiently large $n\in\NN$.
We have already shown that there exists some
$\gamma\in\Gamma$ such that $n^{\beta}\le q_{\gamma }\le n^{-\beta}2^{n|\gamma|}$ and consequently 
Theorem \ref{thm:3} yields 
\[
	|m_{2^n}(\xi)-m_{2^{n-1}}(\xi)|\lesssim (n+1)^{-\alpha}
\]
provided that $1-\Xi_{n}(\xi)\not=0$ and this proves \eqref{eq:63} and we
obtain

\begin{align}
	\label{eq:110}
	\big\|
	\calF^{-1}\big((m_{2^n}-m_{2^{n-1}})(1-\Xi_{n})\hat{f}\big)
	\big\|_{\ell^2}
	\lesssim
	(1+n)^{-\alpha} \log(n+2) 
	\|f\|_{\ell^2}.
\end{align}
Interpolating \eqref{eq:110} with \eqref{eq:98} we obtain 
\[
	\big\|\calF^{-1}\big((m_{2^n}-m_{2^{n-1}}) (1-\Xi_{n})\hat{f}\big)\big\|_{\ell^p}
	\lesssim
	(1+n)^{-c_p\alpha}\log(n+2)\|f\|_{\ell^p}.
\]
for some $c_p>0$. Choosing $\alpha>0$ and $l\in\NN$ appropriately
large one obtains \eqref{eq:86}.

\subsection{The estimate for the first norm in \eqref{eq:82}}

 Note that for
any $\xi\in\TT^d$ so that
\[
\bigg|\xi_{\gamma}-\frac{a_{\gamma}}{q}\bigg|\le \frac{1}{8d\cdot2^{j(|\gamma|-\chi)}}
\] 
for every $\gamma\in \Gamma$
with $1\le q\le e^{j^{1/10}}$ we have 
\begin{align}
\label{eq:116}
  m_{2^j}(\xi)-m_{2^{j-1}}(\xi)=G(a/q)\big(\Psi_{2^j}(\xi-a/q)-\Psi_{2^{j-1}}(\xi-a/q)\big)+q^{-\delta}E_{2^j}(\xi)
\end{align}
where 
\begin{align}
\label{eq:117}
  |E_{2^j}(\xi)|\lesssim 2^{-j/2}.
\end{align}
These two properties \eqref{eq:116} and \eqref{eq:117} follow from
Proposition \ref{prop:2} with $L_1=2^j$, $L_2=8d\cdot2^{\chi j}$ and
$L_3=e^{j^{1/10}}$, since
\[
|E_{2^j}(\xi)|\lesssim q^{\delta}L_2L_32^{-j}\lesssim
\big(e^{-j((1-\chi)\log 2-2j^{-9/10})}\big)\lesssim 2^{-j/2}
\]
which holds for sufficiently large $j\in\NN$, when $\chi>0$ is
sufficiently small. Let us introduce for
every $j\in\NN$ 
new multipliers
\[
  \nu_{2^j}(\xi)=\sum_{a/q\in\mathscr U_{j^l}}G(a/q)\big(\Psi_{2^j}(\xi-a/q)-\Psi_{2^{j-1}}(\xi-a/q)\big)\eta\big(2^{j(A-\chi I)}(\xi-a/q)\big)
\]
and note that by \eqref{eq:116} 
\begin{align*}
 \big|(m_{2^j}(\xi)-m_{2^{j-1}}(\xi))\Xi_j(\xi)-\nu_{2^j}(\xi)\big|\lesssim 2^{-j/2}
\end{align*}
and consequently by Plancherel's theorem
\begin{align}
\label{eq:119}
  \big\|\mathcal
  F^{-1}\big(\big((m_{2^j}-m_{2^{j-1}})\Xi_j-\nu_{2^j}\big)\hat{f}\big)\big\|_{\ell^2}\lesssim
2^{-j/2}\|f\|_{\ell^2}.
\end{align}
Moreover, by Theorem \ref{th:3} we have
\begin{align*}
  \big\|\mathcal
  F^{-1}\big((m_{2^j}-m_{2^{j-1}})\Xi_j\hat{f}\big)\big\|_{\ell^p}\lesssim  \log(j+2)\|f\|_{\ell^p}
\end{align*}
and
\begin{align*}
\big\|\mathcal F^{-1}(\nu_{2^j}\hat{f})\big\|_{\ell^p}
\lesssim
  |\mathscr U_{j^l}|\|f\|_{\ell^p}\lesssim e^{(d+1)j^{1/10}}\|f\|_{\ell^p}
\end{align*}
thus
\begin{align}
\label{eq:120}
  \big\|\mathcal
  F^{-1}\big(\big((m_{2^j}-m_{2^{j-1}})\Xi_j-\nu_{2^j}\big)\hat{f}\big)\big\|_{\ell^p}\lesssim
e^{(d+1)j^{1/10}}\|f\|_{\ell^p}.
\end{align}
Interpolating now \eqref{eq:119} with \eqref{eq:120} we can conclude
that for some $c_p>0$
\begin{align}
\label{eq:121}\big\|\mathcal
  F^{-1}\big(\big((m_{2^j}-m_{2^{j-1}})\Xi_j-\nu_{2^j}
\big)\hat{f}\big)\big\|_{\ell^p}\lesssim 2^{-c_pj}\|f\|_{\ell^p}.
\end{align}
For every $j\in\NN$, $s\in\NN_0$ define multipliers
\[
  \nu_{2^j}^s(\xi)=\sum_{a/q\in\mathscr U_{(s+1)^l}\setminus\mathscr
    U_{s^l}}
G(a/q)\big(\Psi_{2^j}(\xi-a/q)-\Psi_{2^{j-1}}(\xi-a/q)\big)\eta\big(2^{s(A-\chi I)}(\xi-a/q)\big)
\]
and note that by \eqref{eq:48} we see
\begin{multline}
\label{eq:123}
  \Big|
\nu_{2^j}(\xi)-\sum_{0\le s<j}\nu_{2^j}^s(\xi)\Big|\\
\le\sum_{0\le s<j}
\sum_{a/q\in\mathscr U_{(s+1)^l}\setminus\mathscr
    U_{s^l}}
|G(a/q)|\big|\Psi_{2^j}(\xi-a/q)-\Psi_{2^{j-1}}(\xi-a/q)\big|\big|\eta\big(2^{s(A-\chi
  I)}(\xi-a/q)\big)-\eta\big(2^{j(A-\chi I)}(\xi-a/q)\big)\big|\\
\lesssim 2^{-\chi j/d}
\end{multline}
since $|\Psi_{2^j}(\xi-a/q)-\Psi_{2^{j-1}}(\xi-a/q)|\lesssim 2^{-\chi j/d}$, provided that
$\eta\big(2^{s(A-\chi I)}(\xi-a/q)\big)-\eta\big(2^{j(A-\chi
  I)}(\xi-a/q)\big)\not=0$. The estimate \eqref{eq:123} combined with
Plancherel's theorem implies that
\begin{align}
\label{eq:124}
\Big\|\mathcal F^{-1}\Big(\big(\nu_{2^j}-\sum_{0\le
  s<j}\nu_{2^j}^s\big)\hat{f}\Big)\Big\|_{\ell^2}\lesssim
2^{-\chi j/d}\|f\|_{\ell^2}
\end{align}
Moreover, since $|\mathscr U_{s^l}|\le |\mathscr U_{j^l}|\lesssim e^{(d+1)j^{1/10}}$ we
have
\begin{align}
\label{eq:127}
  \Big\|\mathcal F^{-1}\Big(\big(\nu_{2^j}-\sum_{0\le
  s<j}\nu_{2^j}^s\big)\hat{f}\Big)\Big\|_{\ell^p}\lesssim e^{(d+1)j^{1/10}}\|f\|_{\ell^p}.
\end{align}
Interpolating \eqref{eq:124} with \eqref{eq:127} one immediately
concludes that  for some $c_p>0$ 
\begin{align}
  \label{eq:128}
\Big\|\mathcal F^{-1}\Big(\big(\nu_{2^j}-\sum_{0\le
  s<j}\nu_{2^j}^s\big)\hat{f}\Big)\Big\|_{\ell^p}\lesssim 2^{-c_pj}\|f\|_{\ell^p}.
\end{align}
In view of \eqref{eq:121} and \eqref{eq:128}  it suffices to prove that
for every $s\in\NN_0$ we have
\begin{align}
  \label{eq:129}
\Big\|V_r\Big(\mathcal
  F^{-1}\Big(\sum_{j=s+1}^n\nu_{2^j}^s\hat{f}\Big): n\in\NN_0\Big)\Big\|_{\ell^p}
  \lesssim
  (s+1)^{-2}
  \|f\|_{\ell^p}.
\end{align}

\subsection{$\ell^2(\ZZ^d)$ estimates for \eqref{eq:129}}
Our aim will be to prove  the
following.
\begin{theorem}
	\label{thm:2}
	For every $r\in(2, \infty)$ there is $C_r > 0$ such that for any $s \in \NN_0$ and
        $f\in\ell^2\big(\ZZ^d\big)$
        \begin{align}
          \label{eq:130}
\Big\|V_r\Big(\mathcal
  F^{-1}\Big(\sum_{j=s+1}^n\nu_{2^j}^s\hat{f}\Big): n\in\NN_0\Big)\Big\|_{\ell^2}
  \le C_r
  (s+1)^{-\delta l+1}
 \|f\|_{\ell^2}
        \end{align}
with $l\in\NN_0$ defined as in \eqref{eq:7} and $\delta>0$ as in \eqref{eq:20}.
\end{theorem}
\begin{proof}
  For $s \in \NN_0$ we set $\kappa_s = 20 d \big(\lfloor
  (s+1)^{1/10}\rfloor+1\big)$ and $Q_s = \big(\big\lfloor
  e^{(s+1)^{1/10}}\big\rfloor\big)!$. We shall estimate separately the
  pieces of $r$-variations where $0\le n \le 2^{\kappa_s}$ and $n \ge
  2^{\kappa_s}$. By \eqref{eq:19} and \eqref{eq:25} we see that
	\begin{multline}
		\label{eq:131}
                \Big\|V_r\Big(\mathcal
                F^{-1}\Big(\sum_{j=s+1}^n\nu_{2^j}^s\hat{f}\Big):
                n\in\NN_0\Big)\Big\|_{\ell^2} \lesssim \big\lVert
                \calF^{-1}\big(\nu_{2^{s+1}}^s \hat{f} \big)
                \big\rVert_{\ell^2} \\
+\Big\|V_r\Big(\mathcal
                F^{-1}\Big(\sum_{j=s+1}^n\nu_{2^j}^s\hat{f}\Big):
                0\le n\le 2^{\kappa_s}\Big)\Big\|_{\ell^2}
		+\Big\|V_r\Big(\mathcal
                F^{-1}\Big(\sum_{j=s+1}^n\nu_{2^j}^s\hat{f}\Big):
                n\ge2^{\kappa_s}\Big)\Big\|_{\ell^2}.
	\end{multline}
	
By Plancherel's theorem, \eqref{eq:20} and the disjointness of supports of $\eta_s(\xi-a/q)$'s
	while $a/q$ varies over $\mathscr{U}_{(s+1)^l}\setminus\mathscr{U}_{s^l}$, the first term in \eqref{eq:131} is bounded by
	$(s+1)^{-\delta l}\|f\|_{\ell^2}$.
        Now we estimate the supremum over $0 \le n \le
        2^{\kappa_s}$. By Lemma \ref{lem:6} we have
	$$
	\Big\|V_r\Big(\mathcal
                F^{-1}\Big(\sum_{j=s+1}^n\nu_{2^j}^s\hat{f}\Big):
                0\le n\le 2^{\kappa_s}\Big)\Big\|_{\ell^2}
\lesssim 
	\sum_{i = 0}^{\kappa_s}
	\Big(
	\sum_{j = 0}^{2^{\kappa_s-i}-1}
	\Big\lVert
	\sum_{m \in I_j^i}
	\calF^{-1}\big(\nu_{2^{m}}^s \hat{f}\big)
	\Big\rVert_{\ell^2}^2
	\Big)^{1/2}
	$$
	where $I_j^i = (j 2^i, (j+1) 2^i]$. For any $i \in \{0, \ldots, \kappa_s\}$, by Plancherel's
	theorem  we get
	\begin{multline*}
		\sum_{j = 0}^{2^{\kappa_s-i}-1}
	\Big\lVert
	\sum_{m \in I_j^i}
	\calF^{-1}\big(\nu_{2^{m}}^s \hat{f}\big)
	\Big\rVert_{\ell^2}^2
=
		\sum_{j = 0}^{2^{\kappa_s-i} - 1}
		\sum_{m,m' \in I_j^i}
		\int_{\TT^d}
		\abs{\nu_{2^{m}}^s(\xi)}
		\cdot
		\abs{\nu_{2^{m'}}^s(\xi)}
		\abs{\hat{f}(\xi)}^2
		{\:\rm d}\xi\\
		\leq
		\sum_{a/q \in \mathscr{U}_{(s+1)^l}\setminus\mathscr{U}_{s^l}}
		\abs{G(a/q)}^2
		\sum_{j = 0}^{2^{\kappa_s-i} - 1}
		\sum_{m, m' \in I_j^i}
		\int_{\TT^d}
		\abs{\Delta_m(\xi - a/q)}
		\cdot
		\abs{\Delta_{m'}(\xi - a/q)}
		\cdot
		\eta_s(\xi - a/q)^2
		\abs{\hat{f}(\xi)}^2
		{\: \rm d}\xi,
	\end{multline*}
where $\Delta_m(\xi)=\Psi_{2^{m}}(\xi)-\Psi_{2^{m-1}}(\xi)$ and
$\eta_s(\xi)=\eta(2^{s(A-\chi I)}\xi)$, since the supports are effectively disjoint.
	Using \eqref{eq:48} and \eqref{eq:49} we conclude
	\begin{equation*}
\sum_{m \in \ZZ} \big\lvert
                \Delta_m(\xi)  \big\rvert \lesssim  \sum_{m \in \ZZ} 
                \min\big\{|2^{mA}\xi|_{\infty}, |2^{mA}\xi|_{\infty}^{-1/d}\big\}   \lesssim 1.
	\end{equation*}
	Therefore, by \eqref{eq:20} we may estimate
        \begin{multline*}
        \sum_{j = 0}^{2^{\kappa_s-i} - 1}
	\Big\lVert
	\sum_{m \in I_j^i}
	\calF^{-1}\big(\nu_{2^{m}}^s \hat{f}\big)
	\Big\rVert_{\ell^2}^2
	\lesssim
	(s+1)^{-2\delta l}
	\sum_{a/q \in \mathscr{U}_{(s+1)^l}\setminus\mathscr{U}_{s^l}}
	\int_{\TT^d} \eta_s(\xi - a/q)^2 \abs{\hat{f}(\xi)}^2 {\: \rm d}\xi
	\lesssim
	(s+1)^{-2\delta l}
	\lVert f \rVert_{\ell^2}^2  
        \end{multline*}
since if $a/q \in \mathscr{U}_{(s+1)^l}\setminus\mathscr{U}_{s^l}$
then $q\gtrsim (s+1)^l$. In the last step we have used disjointness of supports of $\eta_s(\cdot - a/q)$
	while $a/q$ varies over $\mathscr{U}_{(s+1)^l}\setminus\mathscr{U}_{s^l}$. We have just proven
	\begin{equation}
\label{eq:132}
		\Big\|V_r\Big(\mathcal
                F^{-1}\Big(\sum_{j=s+1}^n\nu_{2^j}^s\hat{f}\Big):
                0\le n\le 2^{\kappa_s}\Big)\Big\|_{\ell^2}
		\lesssim
		\kappa_s (s+1)^{-\delta l } 
		\lVert f \rVert_{\ell^2}
		\lesssim (s+1)^{-\delta l +1}
		\vnorm{f}_{\ell^2}.
	\end{equation}
	Next, we consider the case when the supremum is taken over $n \geq 2^{\kappa_s}$. For any
	$x, y \in \ZZ^d$ we define
	\[
	I(x, y)
	=
	V_r\Big(
	\sum_{a/q \in \mathscr{U}_{(s+1)^l}\setminus\mathscr{U}_{s^l}}
	G(a/q)
	e^{-2\pi i \sprod{(a/q)}{x}}
	\calF^{-1} \Big(
	\sum_{j=s+1}^n\big(\Psi_{2^j}-\Psi_{2^{j-1}}\big)
	\eta_s
	\hat{f}(\cdot + a/q)
	\Big)(y):
                n\ge 2^{\kappa_s}\Big)
	\]
    and
	\begin{align*}
		J(x, y) =
		\sum_{a/q \in \mathscr{U}_{(s+1)^l}\setminus\mathscr{U}_{s^l}} G(a/q) e^{-2\pi i \sprod{(a/q)}{x}}
		\calF^{-1} \big( \eta_s \hat{f}(\cdot + a/q)\big)(y).
	\end{align*}
	By Plancherel's theorem, for any $u \in \NN^d_{Q_s}$ and $a/q \in \mathscr{U}_{(s+1)^l}\setminus\mathscr{U}_{s^l}$
	we have
	\begin{multline*}
		\big\lVert
		\calF^{-1}\big((\Psi_{2^n}-\Psi_{2^{n-1}})\eta_s
                \hat{f}(\cdot + a/q)\big)(x+u)
		-
		\calF^{-1}\big( (\Psi_{2^n}-\Psi_{2^{n-1}})\eta_s \hat{f}(\cdot + a/q)\big)(x)
		\big\rVert_{\ell^2(x)}\\
		=
		\big\lVert
		(1 - e^{-2\pi i \sprod{\xi}{u}}) (\Psi_{2^n}(\xi)-\Psi_{2^{n-1}}(\xi)) \eta_s(\xi) \hat{f}(\xi + a/q)
		\big\rVert_{L^2({\rm d}\xi)}
		\lesssim
		2^{-n/d}
		\cdot
		\norm{u}
		\cdot
		\big\lVert
		\eta_s(\cdot - a/q) \hat{f}
		\big\rVert_{L^2}
	\end{multline*}
	since, by \eqref{eq:48},
	$$
	\sup_{\xi \in \TT^d}
	\norm{\xi} \cdot \abs{(\Psi_{2^n}(\xi)-\Psi_{2^{n-1}}(\xi))}
	\lesssim
	\sup_{\xi \in \TT^d}
	\norm{\xi} \cdot \norm{2^{nA}\xi}^{-1/d}
	\leq 2^{-n/d}.
	$$
	Therefore,
	$$
	\big\lvert
	\lVert I(x, x+u) \rVert_{\ell^2(x)}
	- \lVert I(x,x) \rVert_{\ell^2(x)}
	\big\rvert
	\leq
	\norm{u}
	\sum_{n = 2^{\kappa_s}}^\infty
	2^{-n/d}
	\sum_{a/q \in \mathscr{U}_{(s+1)^l}\setminus\mathscr{U}_{s^l}}
	\lVert \eta_s(\cdot - a/q) \hat{f} \rVert_{\ell^2}
	$$
	because the set
        $\mathscr{U}_{(s+1)^l}\setminus\mathscr{U}_{s^l}\subseteq
        \mathscr{U}_{(s+1)^l}$ contains at most $e^{(d+1)(s+1)^{1/10}}$ elements and
	$$ 2^{\kappa_s}(\log 2)/d - (s+1)^{1/10}e^{(s+1)^{1/10}}-
        (d+1) (s+1)^{1/10} \geq  s$$ 
for sufficiently large $s\ge 0$. Thus we obtain
	$$
	\lVert I(x, x) \rVert_{\ell^2(x)} \lesssim
	\lVert I(x, x+u) \rVert_{\ell^2(x)} + 2^{-s } \lVert f \rVert_{\ell^2}.
	$$
	In particular,
	\[
		\Big\|V_r\Big(\mathcal
                F^{-1}\Big(\sum_{j=s+1}^n\nu_{2^j}^s\hat{f}\Big):
                n\ge 2^{\kappa_s}\Big)\Big\|_{\ell^2}^2
	\lesssim
	\frac{1}{Q_s^d}
	\sum_{u \in \NN_{Q_s}^d}
	\big\lVert I(x, x+u) \big\rVert_{\ell^2(x)}^2
	+
	2^{-2s }
	\lVert f \rVert_{\ell^2}^2.
	\]
	Let us observe that the functions $x \mapsto I(x, y)$ and $x \mapsto J(x, y)$ are
	$Q_s\ZZ^d$-periodic. Next, by double change of variables and periodicity we get
	$$
	\sum_{u \in \NN_{Q_s}^d}
	\lVert I(x, x+u) \rVert_{\ell^2(x)}^2
	=
	\sum_{x \in \ZZ^d}
	\sum_{u \in \NN_{Q_s}^d}
	I(x-u, x)^2
	=
	\sum_{x \in \ZZ^d}
	\sum_{u \in \NN_{Q_s}^d} I(u, x)^2
	=
	\sum_{u \in \NN_{Q_s}^d}
	\lVert I(u, x) \rVert_{\ell^2(x)}^2
	$$
	Using Proposition \ref{prop:10} and \eqref{eq:20}, we obtain
	\begin{multline*}
\sum_{u \in \NN_{Q_s}^d}
	\lVert I(u, x) \rVert_{\ell^2(x)}^2\lesssim 
		\sum_{u \in \NN_{Q_s}^d}
		\lVert
		J(u, x)
		\rVert_{\ell^2(x)}^2
		=
        \sum_{u \in \NN_{Q_s}^d}
		\lVert J(x, x+u) \rVert_{\ell^2(x)}^2
        \\
		=\sum_{u \in \NN_{Q_s}^d}
		\int_{\TT^d}\Big|\sum_{a/q\in\mathscr{U}_{(s+1)^l}\setminus\mathscr{U}_{s^l}}G(a/q) e^{2\pi i \sprod{(a/q)}{u}}
		\eta_s(\xi-a/q)
		\hat{f}(\xi)\Big|^2
		{\: \rm d}\xi
		\lesssim
		(s+1)^{-2\delta l}
		Q_s^d
		\cdot
		\|f\|_{\ell^2}^2.
	\end{multline*}
	In the last step we have also used the disjointness of supports of
	$\eta_s(\cdot - a/q)$ while $a/q$ varies over $\mathscr{U}_{(s+1)^l}\setminus\mathscr{U}_{s^l}$. Therefore,
	\[
		\Big\|V_r\Big(\mathcal
                F^{-1}\Big(\sum_{j=s+1}^n\nu_{2^j}^s\hat{f}\Big):
                n\ge 2^{\kappa_s}\Big)\Big\|_{\ell^2}
		\lesssim
		(s+1)^{-\delta l} 
		\vnorm{f}_{\ell^2}
	\]
	which together with \eqref{eq:132} concludes the proof.
\end{proof}
\subsection{$\ell^p(\ZZ^d)$ estimates for \eqref{eq:129}} Recall that for
$s\in\NN_0$ we have 
 $$\kappa_s = 20 d \big(\lfloor (s+1)^{1/10}\rfloor+1\big)$$ 
and 
$$Q_s = \big(\big\lfloor e^{(s+1)^{1/10}}\big\rfloor\big)!$$
as in the proof of Theorem \ref{thm:2}. We show that  for every
$p\in(1, \infty)$ and $r\in(2, \infty)$ there is a constant $C_{p, r}>0$ such that for every
$s\in\NN_0$ 
\begin{align}
\label{eq:133}
  \Big\|V_r\Big(\mathcal
  F^{-1}\Big(\sum_{j=s+1}^n\nu_{2^j}^s\hat{f}\Big): n\in\NN_0\Big)\Big\|_{\ell^p}
  \le C_{p, r}
  s\log(s+2)
  \|f\|_{\ell^p}.
\end{align}
Then interpolation \eqref{eq:133} with \eqref{eq:130} will immediately
imply \eqref{eq:129}. The proof of \eqref{eq:133} will consist of two
parts. We shall bound separately the variations when $0\le n\le 2^{\kappa_s}$
and when $n\ge 2^{\kappa_s}$, see Theorem \ref{thm:7} and Theorem
\ref{thm:8} respectively. 

\begin{theorem}
\label{thm:7}
Let $p\in(1, \infty)$ and $r\in(2, \infty)$ then there is a constant $C_{p, r}>0$
such that for every $s\in\NN_0$ and every $f\in\ell^p\big(\ZZ^d\big)$ we have
\[
  \Big\|V_r\Big(\mathcal
  F^{-1}\Big(\sum_{j=s+1}^n\nu_{2^j}^s\hat{f}\Big): 0\le n\le 2^{\kappa_s}\Big)\Big\|_{\ell^p}
  \le C_{p, r}
  s \log(s+2)
  \|f\|_{\ell^p}.
\]
\end{theorem}
\begin{proof}
	Let $J=\big\lfloor e^{(s+1)^{1/2}}\big\rfloor$ and
        define the multiplier 
	$$
	\mu_J(\xi) = J^{-k} \sum_{y \in \NN_J^k} e^{2\pi i \sprod{\xi}{\calQ(y)}}
	$$
	where $\NN^k_J = \{1, 2, \ldots, J\}^k$. We see that $\mu_J$
        corresponds to the averaging operator,
        i.e. $M_Jf=\calF^{-1}\big(\mu_J\hat f\big)$. Thus
        for each $r \in [1, \infty]$
	we have
	\[
		\big\lVert
		\calF^{-1} \big(\mu_J \hat{f} \big)
		\big\rVert_{\ell^r}
		\leq
		\vnorm{f}_{\ell^r}.
	\]
	Moreover, if $\xi \in \TT^d$ is such that  $\norm{\xi_{\gamma} - a_{\gamma}/q}
        \leq 2^{-s(|\gamma|-\chi)}$  for every $\gamma\in\Gamma$ with some $1 \leq q \leq e^{(s+1)^{1/10}}$ and
	$a \in A_q$, then 
	\begin{equation*}
		\mu_J(\xi) = G(a/q)\Phi_J(\xi-a/q)+\mathcal O\big(e^{-\frac{1}{2}(s+1)^{1/2}}\big).
	\end{equation*}
	Indeed, by Proposition \ref{prop:0} with $L_1=2^s$,
        $L_2=2^{s\chi}$, $L_3=e^{(s+1)^{1/10}}$ and $N=J$ we see that the
        error term is dominated by
        \begin{align*}
			L_3/J+L_2L_3J^{-1}\sum_{\gamma \in \Gamma}\big(J/L_1\big)^{|\gamma|}	
			& \lesssim e^{(s+1)^{1/10}-(s+1)^{1/2}}+2^{s\chi}e^{(s+1)^{1/10}-(s+1)^{1/2}}(e^{(s+1)^{1/2}}\cdot2^{-s})\\
			& \lesssim e^{-\frac{1}{2}(s+1)^{1/2}}.
        \end{align*}
Therefore,
	\begin{equation}
	\label{eq:136}
		\big\lvert \mu_J(\xi) - G(a/q) \big\rvert 
		\lesssim\big|G(a/q)(\Phi_J(\xi-a/q)-1)\big|+e^{-\frac{1}{2}(s+1)^{1/2}}
		\lesssim e^{-\frac{1}{2}(s+1)^{1/2}}
	\end{equation}
since
\[
|\Phi_J(\xi-a/q)-1|\lesssim 
		 \big\lvert J^{A}(\xi-a/q) \big\rvert\lesssim e^{(s+1)^{1/2}}2^{-s(1-\chi)}.
\]
Let us define the multipliers
\[ 
  \Pi_{2^j}^s(\xi)=\sum_{a/q\in\mathscr U_{(s+1)^l}\setminus\mathscr
    U_{s^l}}
\big(\Psi_{2^j}(\xi-a/q)-\Psi_{2^{j-1}}(\xi-a/q)\big)\eta\big(2^{s(A-\chi I)}(\xi-a/q)\big)
\]
and observe that by \eqref{eq:136} we have
\begin{align}
  \label{eq:138}
  \nu_{2^j}^s(\xi)-\mu_J(\xi)\Pi_{2^j}^s(\xi)=\mathcal O\big(e^{-\frac{1}{2}(s+1)^{1/2}}\big).
\end{align}
By \eqref{eq:138} and Plancherel's theorem we have 
\begin{align}
\label{eq:140}
  \big\|\mathcal
  F^{-1}\big((\nu_{2^j}^s-\mu_J\Pi_{2^j}^s)\hat{f}\big)\big\|_{\ell^2}\lesssim e^{-\frac{1}{2}(s+1)^{1/2}}\|f\|_{\ell^2}
\end{align}
furthermore, for every $p\in(1, \infty)$ we obtain
\begin{align}
\label{eq:139}
\big\|\mathcal
  F^{-1}\big((\nu_{2^j}^s-\mu_J\Pi_{2^j}^s)\hat{f}\big)\big\|_{\ell^p}\lesssim
  |U_{(s+1)^l}|\|f\|_{\ell^p}\lesssim e^{(s+1)^{1/10}}\|f\|_{\ell^p}.  
\end{align}
Interpolating now \eqref{eq:140} with \eqref{eq:139} one has for some
$c_p>0$ that
\begin{align}
\label{eq:141}
\big\|\mathcal
  F^{-1}\big((\nu_{2^j}^s-\mu_J\Pi_{2^j}^s)\hat{f}\big)\big\|_{\ell^p}\lesssim
   e^{-c_p(s+1)^{1/2}}\|f\|_{\ell^p}.  
\end{align}
Thus by \eqref{eq:141} we obtain
\begin{multline*}
  \Big\|V_r\Big(\mathcal
  F^{-1}\Big(\sum_{j=s+1}^n(\nu_{2^j}^s-\mu_J\Pi_{2^j}^s)\hat{f}\Big): 0\le n\le 2^{\kappa_s}\Big)\Big\|_{\ell^p}\\
\lesssim \sum_{n=0}^{2^{\kappa_s}}
\big\|\mathcal
  F^{-1}\big((\nu_{2^n}^s-\mu_J\Pi_{2^n}^s)\hat{f}\big)\big\|_{\ell^p}
	\lesssim 2^{\kappa_s}e^{-c_p(s+1)^{1/2}}\|f\|_{\ell^p}\lesssim \|f\|_{\ell^p}
\end{multline*}
since $2^{\kappa_s}e^{-c_p(s+1)^{1/2}}\lesssim 1$. The proof of
Theorem \ref{thm:7} will be completed if we show
\[ 
\Big\|V_r\Big(\mathcal
  F^{-1}\Big(\sum_{j=s+1}^n\Pi_{2^j}^s\hat{f}\Big): 0\le n\le 2^{\kappa_s}\Big)\Big\|_{\ell^p}
\lesssim \kappa_{s}\log(s+2)\|f\|_{\ell^p}.
\]
  Appealing to inequality \eqref{eq:62} we see that
\[
\Big\|V_r\Big(\mathcal
  F^{-1}\Big(\sum_{j=s+1}^n\Pi_{2^j}^s\hat{f}\Big): 0\le n\le 2^{\kappa_s}\Big)\Big\|_{\ell^p}
\lesssim\sum_{i=0}^{\kappa_s}\bigg\|\Big(\sum_{j = 0}^{2^{\kappa_s-i}-1}
	\Big|
	\sum_{m \in I_j^i}
	\calF^{-1}\big(\Pi_{2^{m}}^s \hat{f}\big)
	\Big|^2
	\Big)^{1/2}\bigg\|_{\ell^p}
\]
where $I_j^i = (j 2^i, (j+1)2^i]$.
For each $i\in\{0, 1, \ldots, \kappa_s\}$  we have by
Khinchine's inequality that
\[
	\bigg\|\Big(\sum_{j = 0}^{2^{\kappa_s-i}-1}
	\Big|
	\sum_{m \in I_j^i}
	\calF^{-1}\big(\Pi_{2^{m}}^s \hat{f}\big)
	\Big|^2
	\Big)^{1/2}\bigg\|_{\ell^p}
\lesssim 
	\Big(\int_0^1\Big\|\sum_{j = 0}^{2^{\kappa_s-i}-1}
	\sum_{m \in I_j^i}\varepsilon_j(\omega)
	\calF^{-1}\big(\Pi_{2^{m}}^s \hat{f}\big)
	\Big\|_{\ell^p}^p{\: \rm d}\omega
	\Big)^{1/p}.
\]
It suffices to show that for every 
$i\in\{0, 1, \ldots, \kappa_s\}$ and $\omega\in[0, 1]$ we have 
\begin{align}
\label{eq:146}
\Big\|\sum_{j = 0}^{2^{\kappa_s-i}-1}
	\sum_{m \in I_j^i}\varepsilon_j(\omega)
	\calF^{-1}\big(\Pi_{2^{m}}^s \hat{f}\big)
	\Big\|_{\ell^p}\lesssim \log(s+2)\|f\|_{\ell^p}.  
\end{align}
For any sequence  
	$\varepsilon=\big(\varepsilon_j(\omega) : 0 \leq j < 2^{\kappa_s - i} \big)$ with $\varepsilon_j(\omega) \in \{-1, 1\}$,
	we consider the operator
	$$
	\calT_\varepsilon f 
	= \sum_{a/q \in \mathscr{U}_{(s+1)^l}\setminus \mathscr{U}_{s^l}} \calF^{-1}
	\big(\Theta(\cdot- a/q) \eta_s(\cdot- a/q) \hat{f}\big)
	$$
	with
	$$
	\Theta
	= \sum_{j = 0}^{2^{\kappa_s - i}-1} \varepsilon_j(\omega)
	\sum_{m \in I_j^i} (\Psi_{2^{m}}-\Psi_{2^{m-1}}).
	$$ 
We notice that the multiplier
        $\Theta$ corresponds to a continuous singular Radon transform.
        Thus $\Theta$ defines a bounded
	operator on $L^r\big(\RR^d\big)$ for any $r \in (1, \infty)$ with the bound independent of
	the sequence $\big(\varepsilon_j(\omega) : 0 \leq j \leq 2^{\kappa_s-i}\big)$
	(see \cite[Section 11]{bigs}). Hence, by Theorem \ref{th:3}
\[
\|\calT_\varepsilon f\|_{\ell^p}\lesssim \log(s+2)\|f\|_{\ell^p}
\]
and consequently we obtain \eqref{eq:146} and the proof of Theorem
\ref{thm:7} is completed. 
\end{proof}
For each $N \in \NN$ and $s\in\NN_0$ we define  multipliers
\[
	\Omega_N^s(\xi)
	=\sum_{a/q \in \mathscr{U}_{(s+1)^l}\setminus \mathscr{U}_{s^l}}
	G(a/q) \Theta_N(\xi - a/q) \vrho_s(\xi - a/q),
\]
where $\vrho_s(\xi)=\eta\big(Q_{s+1}^{3dA}\xi\big)$ and
$\big(\Theta_N: N\in\NN)$ is a sequence of multipliers on $\RR^d$
such that for  $p\in(1, \infty)$ and $r\in(2, \infty)$ there is a constant
$\mathbf B_{p, r}>0$ such that for every $f\in L^2\big(\RR^d\big)\cap
L^p\big(\RR^d\big)$ we have
\begin{align}
  \label{eq:148}
  \big\|V_r\big(\mathcal
  F^{-1}\big(\Theta_N\calF f\big): N\in\NN\big)\big\|_{L^p}\le \mathbf
  B_{p, r}\|f\|_{L^p}.
\end{align}
In fact the multipliers obeying \eqref{eq:148} have been discussed  in
the Appendix, see Theorem \ref{thm:21} in the context of truncated
singular integrals. Moreover, in this case  $\mathbf
  B_{p, r}\le\mathbf
  B_{p}\frac{r}{r-2}$ for some $\mathbf
  B_{p}>0$.

\begin{theorem}
	\label{th:1}
	Let $p \in (1, \infty)$ and $r\in(2, \infty)$ then there exists $C_{p}>0$ such that for any
	$s \in\NN_0$ and $f\in\ell^p\big(\ZZ^d\big)$ we have
    \[      
	\big\|V_r\big(\calF^{-1}\big(\Omega_N^s \hat{f}\big):
          N\in\NN\big)\big\|_{\ell^p}\le C_p\mathbf
          B_{p, r } \log(s+2)
	\|f\|_{\ell^p}.
	\]
\end{theorem}
\begin{proof}
	Let us observe that
	$$
	\calF^{-1}\big(\Theta_N (\cdot - a/q) \vrho_s(\cdot - a/q) \hat{f} \big)(Q_s x + m)
	=
	\calF^{-1}\big(\Theta_N \vrho_s \hat{f}(\cdot  + a/q)\big)(Q_sx + m) e^{-2\pi i \sprod{(a/q)}{m}}.
	$$
	Therefore,
        \begin{align*}
          \big\|V_r\big(\calF^{-1}\big(\Omega_N^s \hat{f}\big):
          N\in\NN\big)\big\|_{\ell^p}^p
	=
	\sum_{m \in \NN_{Q_s}^d}
\big\|V_r\big(\calF^{-1}\big(\Theta_N \vrho_s F(\cdot\ ;
        m) \big)(Q_s x + m): N \in \NN\big)\big\|_{\ell^p(x)}^p
        \end{align*}
	where
	\begin{equation}
		\label{eq:151}
		F(\xi; m)
		= \sum_{a/q \in \mathscr{U}_{(s+1)^l}\setminus \mathscr{U}_{s^l}} G(a/q) \hat{f}(\xi + a/q) e^{-2\pi i \sprod{(a/q)}{m}}.
	\end{equation}
	Now, by Proposition \ref{prop:10} and  \eqref{eq:151} we get
	\begin{multline*}
		\sum_{m \in \NN_{Q_s}^d}
\big\|V_r\big(\calF^{-1}\big(\Theta_N \vrho_s F(\cdot\ ;
        m) \big)(Q_s x + m): N \in \NN\big)
	\big\|_{\ell^p(x)}^p\\
		\le C_p^p\mathbf B_{p, r}^p
		\sum_{m \in \NN_{Q_s}^d}
\big\|\calF^{-1}\big(\vrho_s F(\cdot\ ;
        m) \big)(Q_s x + m)
	\big\|_{\ell^p(x)}^p\\
=C_p^p\mathbf B_{p, r}^p
		\Big\lVert
		\sum_{a/q \in \mathscr{U}_{(s+1)^l}\setminus \mathscr{U}_{s^l}} G(a/q) \calF^{-1}\big(\vrho_s(\cdot - a/q) \hat{f}\big)
		\Big\rVert_{\ell^p}^p.
	\end{multline*}
It suffices to prove
\begin{align}
 \label{eq:152}
  \big\|\mathcal
  F^{-1}\big(\tilde{\Pi}_s^G\hat{f}\big)\big\|_{\ell^p}\lesssim \log(s+2)\|f\|_{\ell^p}
\end{align}
where
\[
  \tilde{\Pi}_s^G(\xi)=\sum_{a/q \in \mathscr{U}_{(s+1)^l}\setminus \mathscr{U}_{s^l}} G(a/q) \vrho_s(\xi - a/q). 
\]
Observe that arguing in a similar way as in the proof of Theorem
\ref{thm:7} we obtain by \eqref{eq:136} that
\begin{align}
\label{eq:154}
  |\tilde{\Pi}_s^G(\xi)-\mu_J(\xi)\tilde{\Pi}_s(\xi)|\lesssim e^{-\frac{1}{2}(s+1)^{1/2}}
\end{align}
where
\[
  \tilde{\Pi}_s(\xi)=\sum_{a/q \in \mathscr{U}_{(s+1)^l}\setminus \mathscr{U}_{s^l}} \vrho_s(\xi - a/q). 
\]
Therefore, \eqref{eq:154} combined with Plancherel's theorem yields
\begin{align}
\label{eq:158}
  \big\|\mathcal
  F^{-1}\big((\tilde{\Pi}_s^G-\mu_J\tilde{\Pi}_s)\hat{f}\big)\big\|_{\ell^2}\lesssim e^{-\frac{1}{2}(s+1)^{1/2}}\|f\|_{\ell^2}.
\end{align}
We can conclude by interpolation with \eqref{eq:158} that
\[
  \big\|\mathcal
  F^{-1}\big((\tilde{\Pi}_s^G-\mu_J\tilde{\Pi}_s)\hat{f}\big)\big\|_{\ell^p}\lesssim \|f\|_{\ell^p}.
\]
since by Theorem \ref{th:3} we have
\[ 
 \big\|\mathcal
  F^{-1}\big(\tilde{\Pi}_s\hat{f}\big)\big\|_{\ell^p}\lesssim \log(s+2)\|f\|_{\ell^p} 
\]
and the trivial bound
\[
  \big\|\mathcal
  F^{-1}\big(\tilde{\Pi}_s^G\hat{f}\big)\big\|_{\ell^p}\lesssim
  |\mathscr{U}_{(s+1)^l}|\|f\|_{\ell^p}
\lesssim e^{(d+1)(s+1)^{1/10}}\|f\|_{\ell^p}.
\]
This establishes the bound in \eqref{eq:152} and the proof of Theorem
\ref{th:1} is finished. 
\end{proof}

\begin{theorem}
\label{thm:8}
Let $p\in(1, \infty)$ and $r\in(2, \infty)$ then there is a constant $C_{p, r}>0$ such that for every
$s\in\NN_0$ and $f\in\ell^p\big(\ZZ^d\big)$ we have
\[
  \Big\|V_r\Big(\mathcal
  F^{-1}\Big(\sum_{j=s+1}^n\nu_{2^j}^s\hat{f}\Big): n\ge 2^{\kappa_s}\Big)\Big\|_{\ell^p}
  \le C_{p, r}
  \log(s+2)
  \|f\|_{\ell^p}.
\]
\end{theorem}
\begin{proof}
The proof of Theorem \ref{thm:23} ensures that the sequence
$\big(\Psi_{2^n}: n\in\NN_0\big)$  satisfies \eqref{eq:148}.
Thus in view of Theorem \ref{th:1} which will be applied with $N=2^n$ and $\Theta_{2^n}=\Psi_{2^n}$  it only suffices to
prove that for any $n\ge 2^{\kappa_s}$ we have
\begin{align}
\label{eq:161}
  \big\|\mathcal
  F^{-1}\big((\nu_{2^n}^s-\tilde{\Omega}_{2^n}^s)\hat{f}\big)\big\|_{\ell^p}\lesssim 2^{-c_p^1n}e^{c_p^2(s+1)^{1/10}}\|f\|_{\ell^p}
\end{align}
for some $c_p^1, c_p^2>0$, where
\[
\tilde{\Omega}_{2^j}^s(\xi)=\sum_{a/q \in \mathscr{U}_{(s+1)^l}\setminus \mathscr{U}_{s^l}}
	G(a/q) \big(\Psi_{2^j}(\xi - a/q)-\Psi_{2^{j-1}}(\xi - a/q)\big) \vrho_s(\xi - a/q)
\]
and $\Omega_{2^n}=\sum_{j=1}^n\tilde{\Omega}_{2^j}^s$. Obviously we have 
\begin{align}
\label{eq:163}
  \big\|\mathcal
  F^{-1}\big((\nu_{2^n}^s-\tilde{\Omega}_{2^n}^s)\hat{f}\big)\big\|_{\ell^p}\lesssim|\mathscr{U}_{(s+1)^l}|\|f\|_{\ell^p}
\lesssim e^{(d+1)(s+1)^{1/10}}\|f\|_{\ell^p}.
\end{align}
  	Next, we observe that $\vrho_s(\xi - a/q) - \eta_s(\xi - a/q) \neq 0$ implies that 
	$\abs{\xi_\gamma - a_\gamma/q} \geq (16 d)^{-1} Q_{s+1}^{-3 d\abs{\gamma}}$ for some
	$\gamma \in \Gamma$. Therefore, for $n \geq 2^{\kappa_s}$ we have
	$$
	2^{n\abs{\gamma}} \cdot \big\lvert \xi_{\gamma} - a_{\gamma}/q \big\rvert
	\gtrsim
	2^{n|\gamma|} Q_{s+1}^{-3d|\gamma|} \gtrsim 2^{n/2},
	$$
since 
\[
2^{n/2}Q_{s+1}^{-3d}\ge
2^{2^{\kappa_s-1}}e^{-3d(s+1)^{1/10}e^{(s+1)^{1/10}}}\ge e^{(s+1)^{1/10}}
\]
for sufficiently large $s\in\NN_0$. Using \eqref{eq:48}, we obtain
	$$
	\sabs{\Psi_{2^n}(\xi - a/q)-\Psi_{2^{n-1}}(\xi - a/q)} \lesssim 2^{-n/(2d)}.
	$$
	Hence, by \eqref{eq:20}
	\[
		\Big\lvert
		\sum_{a/q \in \mathscr{U}_{(s+1)^l}\setminus \mathscr{U}_{s^l}}
		G(a/q)\big(
		\Psi_{2^n}(\xi - a/q)-\Psi_{2^{n-1}}(\xi - a/q)\big) \big(\eta_s(\xi - a/q) - \vrho_s(\xi - a/q) \big)
		\Big\rvert
		\leq
		C
		(s+1)^{-\delta l}
		2^{-n/(2d)}.
	\]
	Thus, by Plancherel's theorem we obtain
        \begin{align}
          \label{eq:162}
          \big\|\mathcal
  F^{-1}\big((\nu_{2^n}^s-\Omega_{2^n}^s)\hat{f}\big)\big\|_{\ell^2}\lesssim
  2^{-n/(2d)}(s+1)^{-\delta l}\|f\|_{\ell^2}.
        \end{align}
Interpolating now \eqref{eq:162} with \eqref{eq:163} we obtain
\eqref{eq:161} and this completes the proof of Theorem \ref{thm:8}.
\end{proof}

\section{Short variation estimates for truncated singular integral operators}
\label{sec:7}
According to \eqref{eq:32} and  the estimates for long variations from
the previous section it remains to  prove that for all $p\in(1, \infty)$ there is $C_p>0$
such that for all $f\in\ell^p\big(\ZZ^d\big)$ with finite support we have
\[ 
  \bigg\|\Big(\sum_{n\ge0} V_2\big(T_{N}f: N\in[2^n,
  2^{n+1})\big)^2\Big)^{1/2}\bigg\|_{\ell^p}\le C_p\|f\|_{\ell^p}.
\]
Using multipliers $\Xi_n$ from \eqref{eq:53} observe that
\begin{multline}
  \label{eq:72}
   \bigg\|\Big(\sum_{n\ge0} V_2\big(T_{N}f: N\in[2^n,
  2^{n+1})\big)^2\Big)^{1/2}\bigg\|_{\ell^p}\\
\le 
  \bigg\|\Big(\sum_{n\ge0} V_2\big((T_{N}-T_{2^n})\mathcal
  F^{-1}\big(\Xi_{n}\hat{f}\big): N\in[2^n,
  2^{n+1})\big)^2\Big)^{1/2}\bigg\|_{\ell^p}\\
+ 
  \bigg\|\Big(\sum_{n\ge0} V_2\big((T_{N}-T_{2^n})\mathcal
  F^{-1}\big((1-\Xi_{n})\hat{f}\big): N\in[2^n,
  2^{n+1})\big)^2\Big)^{1/2}\bigg\|_{\ell^p}.
\end{multline}
\subsection{The estimate of the second norm in \eqref{eq:72}}
We may assume without of
loss of generality, that $1<
r\le \min\{2, p\}$, since $r$-variations are decreasing,  and it suffices to show that
\begin{align}
  \label{eq:76}
  \big\|V_r\big((T_{N}-T_{2^n})\mathcal
  F^{-1}\big((1-\Xi_{n})\hat{f}\big): N\in[2^n,
  2^{n+1})\big)\big\|_{\ell^p}\lesssim (n+1)^{-2}\|f\|_{\ell^p}.
\end{align}
For this purpose we shall use \eqref{eq:39}. Namely, by \eqref{eq:39}
we immediately see that
\begin{align}
\label{eq:77}
  \big\| V_r\big((T_{N}-T_{2^n})\mathcal
  F^{-1}\big((1-\Xi_{n})\hat{f}\big): N\in[2^n,
  2^{n+1})\big)\big\|_{\ell^p}
\lesssim\max\big\{\mathbf U_p, 2^{n/r}\mathbf U_p^{1-1/r}\mathbf V_p^{1/r}\big\}
\end{align}
where 
\begin{align*}
  \mathbf U_p=\sup_{2^n\le N\le 2^{n+1}}\big\|(T_{N}-T_{2^n})\mathcal
  F^{-1}\big((1-\Xi_{n})\hat{f}\big)\big\|_{\ell^p}
\end{align*}
and 
\begin{align*}
\mathbf V_p=  \sup_{2^n\le N< 2^{n+1}}\big\|(T_{N+1}-T_{N})\mathcal
  F^{-1}\big((1-\Xi_{n
})\hat{f}\big)\big\|_{\ell^p}.
\end{align*}
In view of \eqref{eq:55} we see that
\begin{align}
  \label{eq:81}
 \mathbf U_p\lesssim \log(n+2)\|f\|_{\ell^p} \quad \text{and} \quad
  \mathbf V_p\lesssim 2^{-n}\log(n+2)\|f\|_{\ell^p}.
\end{align}
In fact arguing as in Section \ref{sec:5} we can prove, using Theorem,
\ref{thm:3} that   for big enough $\alpha>0$, which will be specified
later, and for all $n\in\NN_0$ and $N\simeq 2^n$ we have
\begin{align}
\label{eq:85}
\mathbf U_2\lesssim(1+n)^{-\alpha} \log(n+2)\|f\|_{\ell^2}.
\end{align}
Interpolating \eqref{eq:85} with \eqref{eq:81} we obtain 
\begin{align}
\label{eq:87}
\mathbf U_p\lesssim(1+n)^{-c_p\alpha} \log(n+2)\|f\|_{\ell^p}.
\end{align}
for some $c_p>0$. Choosing $\alpha>0$ and $l\in\NN$ appropriately
large we see that \eqref{eq:87} combined with \eqref{eq:77} easily
imply \eqref{eq:76}.

\subsection{The estimate of the first  norm in \eqref{eq:72}}
 By Proposition \ref{prop:2} if  $\xi\in\TT^d$ is such that
\[
\bigg|\xi_{\gamma}-\frac{a_{\gamma}}{q}\bigg|\le \frac{1}{8d\cdot2^{n(|\gamma|-\chi)}}
\] 
for every $\gamma\in \Gamma$
with $1\le q\le e^{n^{1/10}}$ then we have 
\begin{align}
\label{eq:88}
  m_N(\xi)-m_{2^n}(\xi)=G(a/q)\big(\Psi_N(\xi-a/q)-\Psi_{2^n}(\xi-a/q)\big)+q^{-\delta}E_{2^n}(\xi)
\end{align}
for every $N\simeq 2^n$,
where 
\begin{align}
\label{eq:89}
  |E_{2^n}(\xi)|\lesssim 2^{-n/2}.
\end{align}
Let us introduce for every $j, n\in\NN_0$ the new multipliers
\[
  \Xi_{n}^j(\xi)=\sum_{a/q\in\mathscr U_{n^l}}\eta\big(2^{nA+jI}(\xi-a/q)\big)
\]
and note that
\begin{multline*}
    \Big\|\Big(\sum_{n\ge0}V_2\big((T_{N}-T_{2^n})\mathcal
  F^{-1}\big(\Xi_{n}\hat{f}\big): N\in[2^n,
  2^{n+1})\big)^2\Big)^{1/2}\Big\|_{\ell^p}\\
\le \Big\|\Big(\sum_{n\ge0} V_2\big((T_{N}-T_{2^n})\mathcal
  F^{-1}\Big(\sum_{-\chi n\le j< n}\big(\Xi_{n}^j-\Xi_{n}^{j+1}\big)\hat{f}\Big): N\in[2^n,
  2^{n+1})\big)^2\Big)^{1/2}\Big\|_{\ell^p}\\
+\Big\|\Big(\sum_{n\ge0} V_2\big((T_{N}-T_{2^n})\mathcal
  F^{-1}\big(\Xi_{n}^n\hat{f}\big): N\in[2^n,
  2^{n+1})\big)^2\Big)^{1/2}\Big\|_{\ell^p}=I_p^1+I_p^2.
\end{multline*}
We will estimate $I_p^1$ and $I_p^2$ separately. 
First, observe that by \eqref{eq:88} and \eqref{eq:89},
for any $N\simeq 2^n$ and any $a/q\in\mathscr U_{n^l}$ we have
\begin{align}
\label{eq:101}
	\big|
	m_{N}(\xi)-m_{2^n}(\xi)
	\big|
	&\lesssim
	q^{-\delta} \big|\Psi_N(\xi - a/q) - \Psi_{2^n}(\xi - a/q)\big|
	+ q^{-\delta} 2^{-n/2}\\
\nonumber&\lesssim q^{-\delta}\big(\min\big\{1, |2^{nA}(\xi-a/q)|_{\infty},
|2^{nA}(\xi-a/q)|_{\infty}^{-1/d}\big\}+2^{-n/2}\big)  
\end{align}
where the last bound follows from \eqref{eq:48} and \eqref{eq:49}. 
Consequently  \eqref{eq:101}  implies
\begin{align}
  \label{eq:92}
  \big|\big(m_{N}(\xi)-m_{2^n}(\xi)\big)
\big(\eta\big(2^{nA+jI}(\xi-a/q)\big)-\eta\big(2^{nA+(j+1)I}(\xi-a/q)\big)\big)\big|
\lesssim q^{-\delta}\big(2^{-|j|/d}+2^{-n/2}\big).
\end{align}

We begin with bounding $I_p^2$. Since $r$-variations are decreasing
we can assume that $1<r\le \min\{2, p\}$ and it will suffice to show, for some
$\varepsilon=\varepsilon_{p, r}>0$, that
\begin{align}
\label{eq:93}
  \big\|V_r\big((T_{N}-T_{2^n})\mathcal
  F^{-1}\big(\Xi_{n}^{n}\hat{f}\big): N\in[2^n,
  2^{n+1})\big)\big\|_{\ell^p}\lesssim 2^{-\varepsilon
    n}\|f\|_{\ell^p}.
\end{align} 
Exploiting  \eqref{eq:39} we have
\[
  \big\|V_r\big((T_{N}-T_{2^n})\mathcal
  F^{-1}\big(\Xi_{n}^{n}\hat{f}\big): N\in[2^n,
  2^{n+1})\big)\big\|_{\ell^p}
\lesssim\max\big\{\mathbf U_p, 2^{n/r}\mathbf U_p^{1-1/r}\mathbf V_p^{1/r}\big\}
\]
where 
\begin{align*}
  \mathbf U_p=\sup_{2^n\le N\le 2^{n+1}}\big\|(T_{N}-T_{2^n})\mathcal
  F^{-1}\big(\Xi_{n}^{n}\hat{f}\big)\big\|_{\ell^p}
\end{align*}
and 
\begin{align*}
\mathbf V_p=  \sup_{2^n\le N< 2^{n+1}}\big\|(T_{N+1}-T_{N})\mathcal
  F^{-1}\big(\Xi_{n}^{n}\hat{f}\big)\big\|_{\ell^p}.
\end{align*}
In view of Theorem \ref{th:3} we see that
\begin{align}
\label{eq:96}
 \mathbf U_p\lesssim \log(n+2)\|f\|_{\ell^p} \quad \text{and} \quad
  \mathbf V_p\lesssim 2^{-n}\log(n+2)\|f\|_{\ell^p}.
\end{align}
 For $p=2$ by Plancherel's theorem and  \eqref{eq:101} we obtain
 \begin{multline}
\label{eq:97}
   \big\|(T_{N}-T_{2^n})\mathcal
  F^{-1}\big(\Xi_{n}^{n}\hat{f}\big)\big\|_{\ell^2}\\
=
\bigg(\int_{\TT^d}\sum_{a/q\in\mathscr U_{n^l}}|m_{N}(\xi)-m_{2^n}(\xi)|^2\eta\big(2^{nA+nI}(\xi-a/q)\big)^2
  |\hat{f}(\xi)|^2 {\: \rm d}\xi\bigg)^{1/2}\lesssim 
2^{-n/(2d)}\|f\|_{\ell^2}.
 \end{multline}
Therefore, interpolating \eqref{eq:96} with \eqref{eq:97} we obtain
for every $p\in(1, \infty)$ that 
\begin{align*}
  \mathbf U_p\lesssim 2^{-\varepsilon n}\|f\|_{\ell^p}
\end{align*}
which in turn implies \eqref{eq:93} and $I_p^2\lesssim \|f\|_{\ell^p}$.

We shall now estimate $I_p^1$, for this purpose,  for any $0\le s< n$, we  define 
\[
  \Delta_{n, s}^j(\xi)=\sum_{a/q\in\mathscr
    U_{(s+1)^l}\setminus\mathscr
    U_{s^l}}\big(\eta\big(2^{nA+jI}(\xi-a/q)\big)-\eta\big(2^{nA+(j+1)I}(\xi-a/q)\big)\big)\eta\big(2^{s(A-\chi
  I)}(\xi-a/q)\big)
\]
hence
\begin{align*}
  \Xi_{n}^j(\xi)-\Xi_{n}^{j+1}(\xi)=\sum_{0\le s<n}\Delta_{n, s}^j(\xi).
\end{align*}
Now we see that
\begin{multline*}
  I_p^1=\Big\|\Big(\sum_{n\ge0} V_2\big((T_{N}-T_{2^n})\mathcal
  F^{-1}\Big(\sum_{-\chi n\le j< n}\sum_{0\le s<n}\Delta_{n, s}^j\hat{f}\Big): N\in[2^n,
  2^{n+1})\big)^2\Big)^{1/2}\Big\|_{\ell^p}\\
\le\sum_{s\ge 0}\sum_{j\in\ZZ}\Big\|\Big(\sum_{n\ge \max\{s, j, -j/\chi\}} V_2\big((T_{N}-T_{2^n})\mathcal
  F^{-1}\big(\Delta_{n, s}^j\hat{f}\big): N\in[2^n,
  2^{n+1})\big)^2\Big)^{1/2}\Big\|_{\ell^p}.
\end{multline*}
The task now is to show that for some $\varepsilon_p>0$ 
\begin{align}
 \label{eq:108}
  J_p=\bigg\|\Big(\sum_{n\ge \max\{s, j, -j/\chi\}} V_2\big((T_{N}-T_{2^n})\mathcal
  F^{-1}\big(\Delta_{n, s}^j\hat{f}\big): N\in[2^n,
  2^{n+1})\big)^2\Big)^{1/2}\bigg\|_{\ell^p}\lesssim
  s^{-2}2^{-\varepsilon_p|j|}\|f\|_{\ell^p}.
\end{align}
    To estimate \eqref{eq:108} we will use
\eqref{eq:1} from Lemma \ref{lem:8} with $r=2$. Namely, 
let $h_{j, s}=2^{\varepsilon |j|}(s+1)^{\tau}$ with some $\varepsilon>0$ and $\tau>2$
which we choose later. Then for some
$2^n\le t_0<t_1<\ldots<t_{h}<2^{n+1}$ such that $t_{v+1}-t_v\simeq
2^n/h$ where $h=\min\{h_{j, s}, 2^n\}$ we have
\begin{multline}
\label{eq:109}
  J_p\lesssim \bigg\|\Big(\sum_{n\ge \max\{s, j, -j/\chi\}}\sum_{v=0}^{h}\big|(T_{2^{t_v}}-T_{2^n})\mathcal
  F^{-1}\big(\Delta_{n, s}^j\hat{f}\big)\big|^2\Big)^{1/2}\bigg\|_{\ell^p}\\
+\bigg\|\bigg(\sum_{n\ge \max\{s, j, -j/\chi\}}\sum_{v=0}^{h-1}\Big(\sum_{u=t_v}^{t_{v+1}-1}\big|(T_{u+1}-T_{u})\mathcal
  F^{-1}\big(\Delta_{n, s}^j\hat{f}\big)\big|\Big)^2\bigg)^{1/2}\bigg\|_{\ell^p}=J_p^1+J_p^2.
\end{multline}
\subsubsection{The estimates for $J_p^1$}  We begin with $p=2$ and
show that
\begin{align}
 \label{eq:149}
J_2^1\lesssim 2^{-|j|(1/(2d)-\varepsilon/2)}(s+1)^{-\delta l+\tau/2}\|f\|_{\ell^2}.  
\end{align}
For the simplicity of notation define
\[
\varrho_{n, j}(\xi)=\big(\eta\big(2^{nA+jI}\xi\big)-
\eta\big(2^{nA+(j+1)I}\xi\big)\big)
\eta\big(2^{s(A-\chi I)}\xi\big).
\] 
By Plancherel's theorem and \eqref{eq:92}  we have
\begin{multline*}
  J_2^1=\Big(\sum_{v=0}^{h}\int_{\TT^d}\sum_{a/q\in\mathscr
    U_{(s+1)^l}\setminus\mathscr
    U_{s^l}}\sum_{n\ge \max\{s, j, -j/\chi\}}|m_{2^{t_v}}(\xi)-m_{2^n}(\xi)|^2
\varrho_{n,j}(\xi-a/q)^2|\hat{f}(\xi)|^2 {\: \rm d} \xi\Big)^{1/2}\\
\lesssim 
\bigg(\sum_{v=0}^{h}2^{-|j|/d}\int_{\TT^d}\sum_{a/q\in\mathscr
    U_{(s+1)^l}\setminus\mathscr
    U_{s^l}}q^{-2\delta}\eta\big(2^{s(A-\chi I)}(\xi-a/q)\big))|\hat{f}(\xi)|^2d\xi\bigg)^{1/2}\\
\lesssim h^{1/2}2^{-|j|/(2d)}\Big(\int_{\TT^d}(s+1)^{-2\delta l}\sum_{a/q\in\mathscr
    U_{(s+1)^l}\setminus\mathscr
    U_{s^l}}\eta\big(2^{s(A-\chi I)}(\xi-a/q)\big)|\hat{f}(\xi)|^2 {\: \rm d}\xi\Big)^{1/2}\\
\lesssim h^{1/2}2^{-|j|/(2d)}(s+1)^{-\delta l}\|f\|_{\ell^2}
\end{multline*}
as desired. We have used the fact that $q\gtrsim (s+1)^l$ whenever
 $a/q\in\mathscr U_{(s+1)^l}\setminus\mathscr U_{s^l}$ 
 and 
	\[
  \sum_{n\ge \max\{s, j, -j/\chi\}}\varrho_{n, j}(\xi)\lesssim
\eta\big(2^{s(A-\chi I)}(\xi-a/q)\big)
\]
and the disjointness of $\eta\big(2^{s(A-\chi I)}( \cdot-a/q)\big)$
while $a/q$ varies over $\mathscr U_{(s+1)^k}\setminus\mathscr
U_{s^k}$.

Moreover, for any  $p\in(1, \infty)$ we have
\begin{align}
\label{eq:125}
  J_p^1\lesssim 2^{\varepsilon |j|/2}(s+1)^{\tau/2}\log(s+2)\|f\|_{\ell^p}.
\end{align}
Indeed, 
 appealing to the vector-valued
inequality for the maximal function corresponding to the averaging
operators from \cite{mst1} we see that  
\begin{multline*}
  J_p^1=\bigg\|\Big(\sum_{n\ge \max\{s, j, -j/\chi\}}\sum_{v=0}^{h}\big|(T_{2^{t_v}}-T_{2^n})\mathcal
  F^{-1}\big(\Xi_{n,
    s}^j\hat{f}\big)\big|^2\Big)^{1/2}\bigg\|_{\ell^p}\\
\lesssim h^{1/2}
\bigg\|\Big(\sum_{n\ge \max\{s, j, -j/\chi\}}\sup_{N\in\NN}M_{N}\big(\big|\mathcal
  F^{-1}\big(\Xi_{n,
    s}^j\hat{f}\big)\big|\big)^2\Big)^{1/2}\bigg\|_{\ell^p}\\
\lesssim h^{1/2}
\bigg\|\Big(\sum_{n\ge \max\{s, j, -j/\chi\}}\big|\mathcal
  F^{-1}\big(\Xi_{n,
    s}^j\hat{f}\big)\big|^2\Big)^{1/2}\bigg\|_{\ell^p}\lesssim h^{1/2}\log(s+2)\|f\|_{\ell^p},
\end{multline*}
where in the last step we have used \eqref{eq:83}. Interpolating now
\eqref{eq:125} with better \eqref{eq:149} estimate we obtain for some $\varepsilon_p>0$
that
\[
  J_p^1\lesssim (s+1)^{-2}2^{-\varepsilon_p |j|}\|f\|_{\ell^p}.
\]
The proof of \eqref{eq:109} will be completed if we obtain the same
kind of bound for $J_p^2$.
\subsubsection{The estimates for $J_p^2$} 
We begin with $p=2$ and our aim will be to show 
\begin{align}
\label{eq:164}
J_2^2\lesssim 2^{-\varepsilon
  |j|/2}(s+1)^{-\tau/2}\|f\|_{\ell^2}.  
\end{align}
Since $t_{v+1}-t_v\simeq 2^n/h$ then by the Cauchy--Schwarz inequality
we obtain
\[
  J_2^2\le \bigg\|\bigg(\sum_{n\ge \max\{s, j, -j/\chi\}}2^n/h\sum_{u=2^n}^{2^{n+1}-1}\big|(T_{u+1}-T_{u})\mathcal
  F^{-1}\big(\Xi_{n,
    s}^j\hat{f}\big)\big|^2\bigg)^{1/2}\bigg\|_{\ell^2}.
\]
By \eqref{eq:92} we have for $u\simeq2^n$ that
\begin{align}
\label{eq:166}
  |m_{u+1}(\xi)-m_{u}(\xi)|\lesssim \min\big\{2^{-n}, q^{-\delta}(2^{-|j|/d}+2^{-n/2})\big\}.
\end{align}
Two cases must be distinguished. Assume now that $h=2^{\varepsilon
  |j|}(s+1)^{\tau}$, therefore, again by Plancherel's theorem,  we obtain 
\begin{multline*}
  J_2^2\le\Big(\sum_{n\ge \max\{s, j, -j/\chi\}}\int_{\TT^d}2^n/h\sum_{u=2^n}^{2^{n+1}-1}\sum_{a/q\in\mathscr
    U_{(s+1)^k}\setminus\mathscr
    U_{s^k}}|m_{u+1}(\xi)-m_{u}(\xi)|^2
\varrho_{n,j}(\xi-a/q)^2|\hat{f}(\xi)|^2 {\: \rm d}\xi\Big)^{1/2}\\
\lesssim
h^{-1/2}\Big(\int_{\TT^d}\sum_{a/q\in\mathscr
    U_{(s+1)^k}\setminus\mathscr
    U_{s^k}}\sum_{n\ge \max\{s, j, -j/\chi\}}
\varrho_{n,j}(\xi-a/q)|\hat{f}(\xi)|^2 {\: \rm d}\xi\Big)^{1/2}\lesssim h^{-1/2}\|f\|_{\ell^2}
\end{multline*}
since by the telescoping nature and the disjointness of supports  when $a/q$ varies over $\mathscr U_{(s+1)^k}\setminus\mathscr
  U_{s^k}$ we have
\[
\sum_{a/q\in\mathscr U_{(s+1)^k}\setminus\mathscr
  U_{s^k}}\sum_{n\ge \max\{s, j, -j/\chi\}} \varrho_{n,j}(\xi-a/q)\lesssim \sum_{a/q\in\mathscr U_{(s+1)^k}\setminus\mathscr
  U_{s^k}}\eta\big(2^{s(A-\chi I)}(\xi-a/q)\big)\lesssim 1.
\]
If $h=2^n$ then by \eqref{eq:166} we get
\begin{multline*}
  J_2^2\le\Big(\sum_{n\ge \max\{s, j, -j/\chi\}}\int_{\TT^d}\sum_{u=2^n}^{2^{n+1}-1}\sum_{a/q\in\mathscr
    U_{(s+1)^k}\setminus\mathscr
    U_{s^k}}|m_{u+1}(\xi)-m_{u}(\xi)|^2
\varrho_{n,j}(\xi-a/q)^2|\hat{f}(\xi)|^2 {\: \rm d}\xi\Big)^{1/2}\\
\lesssim
2^{-|j|/4}2^{-s/4}\Big(\int_{\TT^d}\sum_{a/q\in\mathscr
    U_{(s+1)^k}\setminus\mathscr
    U_{s^k}}\sum_{n\ge \max\{s, j, -j/\chi\}}
\varrho_{n,j}(\xi-a/q)|\hat{f}(\xi)|^2 {\: \rm d}\xi\Big)^{1/2} \\
\lesssim 2^{-|j|/4}2^{-s/4}\|f\|_{\ell^2}
\end{multline*}
and \eqref{eq:164} is proven.

For $p\in(1, \infty)$ we shall prove that 
\begin{align}
\label{eq:169}
  J_p^2\lesssim \log(s+2)\|f\|_{\ell^p}. 
\end{align}
Indeed, 
\begin{multline*}
  J_p^2=\bigg\|\bigg(\sum_{n\ge \max\{s, j, -j/\chi\}}\sum_{v=0}^{h-1}\Big(\sum_{u=t_v}^{t_{v+1}-1}\big|(T_{u+1}-T_{u})\mathcal
  F^{-1}\big(\Xi_{n,
    s}^j\hat{f}\big)\big|\Big)^2\bigg)^{1/2}\bigg\|_{\ell^p}\\
\le\bigg\|\bigg(\sum_{n\ge \max\{s, j, -j/\chi\}}\Big(\sum_{u=2^n}^{2^{n+1}-1}\big|(H_{u+1}-H_{u})\big|*\big(\big|\mathcal
  F^{-1}\big(\Xi_{n,
    s}^j\hat{f}\big)\big|\big)\Big)^2\bigg)^{1/2}\bigg\|_{\ell^p}\\
  \lesssim\bigg\|\bigg(\sum_{n\ge \max\{s, j, -j/\chi\}}\sup_{N\in\NN}M_{N}\big(\big|\mathcal
  F^{-1}\big(\Xi_{n,
    s}^j\hat{f}\big)\big|\big)^2\bigg)^{1/2}\bigg\|_{\ell^p}\\
\lesssim \bigg\|\bigg(\sum_{n\ge \max\{s, j, -j/\chi\}}\big|\mathcal
  F^{-1}\big(\Xi_{n,
    s}^j\hat{f}\big)\big| ^2\bigg)^{1/2}\bigg\|_{\ell^p}\lesssim \log(s+2)\|f\|_{\ell^p},
\end{multline*}
where in the penultimate line we have used vector-valued maximal
estimates corresponding to the averaging operators from \cite{mst1}
and in the last line we invoked \eqref{eq:83}.  Interpolating now
the estimate \eqref{eq:169} with the estimate from \eqref{eq:164} we
obtain for some $\varepsilon_p>0$ that
\[
  J_p^2\lesssim (s+1)^{-2}2^{\varepsilon_p j}\|f\|_{\ell^p}
\]
and the proof of \eqref{eq:109} is completed.

\section{Proof of Theorem \ref{thm:100}}
\label{sec:9}
The definition of $r$-variations for any $r\in[1, \infty)$ can be extended to more general
sets. Namely, let $A\subseteq\RR$ and let us  define for each  $\big(\seq{a_t}{t \in
  A}\big)\subseteq\CC$, the $r$-variational seminorm by setting
$$
\var{\big(\seq{a_t}{t \in A}\big)} = \sup_{\atop{t_0 < t_1 < \ldots < t_J}{t_j \in A}}
	\Big(\sum_{j = 1}^J \abs{a_{t_j} - a_{t_{j-1}}}^r \Big)^{1/r}
$$
where the supremum is taken over all finite increasing sequences 
 $t_0 < t_1 < \ldots < t_J$ and $t_j\in A$ for $0\le j\le J$.

\begin{lemma}
  \label{lem:200}
Assume that $r\in[1, \infty)$ and  $\big(\seq{a_t}{t \in
  A}\big)\subseteq\CC$. Given an increasing sequence of real numbers $\big(w_k: k\in\NN\big)$
we have
\begin{align}
\label{eq:35}
V_r\big(a_t: t>0\big)\lesssim_r  V_r\big(a_{w_k}: k\in\NN
\big)+\Big(\sum_{k\in\NN}V_r\big(a_t: t\in [w_k, w_{k+1})\big)^2\Big)^{1/2}.
\end{align}

\end{lemma}
\begin{proof}
  For any increasing sequence $t_0<t_1<\ldots<t_J$ we define
  $W_j=\{k\in\NN: t_j< w_k\le t_{j+1} \}$ for $0\le j\le J$. If
  $W_j\not=\emptyset$ then we take $u_j=\min W_j$ and $v_j=\max W_j$.
Now if  $W_j=\emptyset$ then the term
\[
|a_{t_{j+1}}-a_{t_j}|^r
\]
is part of the second term in \eqref{eq:35}. If  $W_j\not=\emptyset$
then
\[
|a_{t_{j+1}}-a_{t_j}|^r\le 3^r\big(|a_{t_{j+1}}-a_{v_j}|^r+
|a_{v_j}-a_{u_j}|^r+|a_{t_j}-a_{u_j}|^r\big)
\]
and we see that the first and the third terms are part of the square
function in \eqref{eq:35}, whereas the middle term is part of the
$r$-variations along $\big(w_k: k\in\NN\big)$.
\end{proof}

\begin{proof}[Proof of Theorem \ref{thm:100}]
  Define the set $\mathbb L= \{x>0: x^2\in \NN\}$. The methods
  presented in the previous sections allow us to establish 
that for every $p\in(1, \infty)$  and $r\in(2, \infty)$ there is $C_{p, r} > 0$ such that for all
	$f \in \ell^p\big(\ZZ^{d_0}\big)$
        \begin{align}
\label{eq:8}
\big\lVert
	V_r\big(  M_N^\calP f: N\in\mathbb L\big)
	\big\rVert_{\ell^p}+
          	\big\lVert
	V_r\big(  T_N^\calP f: N\in\mathbb L\big)
	\big\rVert_{\ell^p}\le
	C_{p, r}\|f\|_{\ell^p}.
        \end{align}
	As before, the constant $C_{p, r}\le C_p\frac{r}{r-2}$ for some
        $C_p>0$ which is independent of the coefficients of the polynomial
        mapping $\calP$.     

Then in view of Lemma \ref{lem:200} with $w_n=n^{1/2}$  we have
\begin{align}
  \label{eq:13}
\big\lVert
	V_r\big(  R_t^\calP f: t>0\big)
	\big\rVert_{\ell^p}\lesssim_r\big\lVert
	V_r\big(  R_N^\calP f: N\in\mathbb L\big)
	\big\rVert_{\ell^p}
+\Big(\sum_{n\in\NN}\big\lVert
	V_r\big(  R_t^\calP f: t\in\big[n^{1/2}, (n+1)^{1/2}\big)\big)
	\big\rVert_{\ell^p}^2
\Big)^{1/2} 
\end{align}
where $R_t^{\calP}$ is either $M_t^{\calP}$ or
$T_t^{\calP}$. According to \eqref{eq:8} we have
\[
\big\lVert
	V_r\big(  R_N^\calP f: N\in\mathbb L\big)
	\big\rVert_{\ell^p}\lesssim_{p, r}\|f\|_{\ell^p}.
\]
In order to estimate the square function in \eqref{eq:13}, note that the function $t\to
R_t^{\calP}$ is constant when $t\in\big[n^{1/2}, (n+1)^{1/2}\big)$
thus
\[
\big\lVert
	V_r\big(  R_t^\calP f: t\in\big[n^{1/2}, (n+1)^{1/2}\big)\big)\big\rVert_{\ell^p}\le\big\|R_{(n+1)^{1/2}}^\calP f
-R_{n^{1/2}}^\calP f\big\|_{\ell^p}\lesssim
n^{-1}\|M_{(n+1)^{1/2}}f\|_{\ell^p}\lesssim n^{-1}\|f\|_{\ell^p}.
\]
Therefore,
\[
\Big(\sum_{n\in\NN}\big\lVert
	V_r\big(  R_t^\calP f: t\in\big[n^{1/2}, (n+1)^{1/2}\big)\big)
	\big\rVert_{\ell^p}^2
\Big)^{1/2}\lesssim
\Big(\sum_{n\in\NN}n^{-2}\Big)^{1/2}\|f\|_{\ell^p}\lesssim \|f\|_{\ell^p}.
\]
This completes the proof of the theorem.
\end{proof}

\appendix
\section{Variational estimates for the continues analogues}
\label{sec:8}
This section is intended to provide $r$-variational estimates for
averaging and truncated singular operators of Radon type in the
continuous settings. These kinds of questions were extensively discussed in
\cite{jsw}, see also the references given there. Here we 
propose a different approach. Firstly, we will discuss long
variations estimates. We give a new proof of
L\'epingle's inequality which will be very much in spirit of good-$\lambda$
inequalities. Secondly, we present a new approach to  short variation
estimates which is based on vector-valued bounds in
\cite{mst1}. This observation, as far as we know, has
not been used in this context before.  To fix notation let
$\calP=\big(\calP_1,\ldots, \calP_{d_0}\big): \RR^k \rightarrow \RR^{d_0}$ be
a polynomial mapping whose components $\calP_j$ are real valued
polynomials on $\RR^k$ such that $\calP_j(0) = 0$. One of the main
objects of our interest will be
\[
  \mathcal M_t^{\calP}f(x)=\frac{1}{|G_t|}\int_{G_t}f\big(x-\mathcal P(y)\big){\: \rm d}y
\]
for $x\in\RR^{d_0}$ where $G$ is an open bounded convex subset of $\RR^k$, containing
the origin and	
\[
	G_t = \big\{x \in \RR^k : t^{-1} x \in G \big\}.
\]
For any $r\in[1, \infty)$ the $r$-variational seminorm $V_r$ of
complex-valued functions $\big(a_t(x):
t>0\big)$ is defined by
\[
V_r\big(a_t(x): t>0\big)=\sup_{0 < t_0<\ldots <t_J}
\bigg(\sum_{j=0}^J|a_{t_{j+1}}(x)-a_{t_j}(x)|^r\bigg)^{1/r}
\] 
where the supremum is taken over all finite increasing sequences. 
In order to avoid some problems with measureability of
$V_r\big(a_t(x): t>0\big)$ we assume that $(0, \infty)\ni t\mapsto
a_t(x)$ is always a continuous function for every $x\in\RR^{d_0}$.
The main  result of this section is the following theorem.
\begin{theorem}
\label{thm:20}
	For every $1 < p < \infty$  and $r\in(2, \infty)$ there is $C_{p, r} > 0$ such that for all
	$f \in L^p\big(\RR^{d_0}\big)$
	\[
     \big\lVert
	V_r\big( \mathcal M_t^\calP f: t>0\big)
	\big\rVert_{L^p}\le
	C_{p, r}\|f\|_{L^p}.
	\]
	Moreover, the constant $C_{p, r}\le C_p\frac{r}{r-2}$ for some
        $C_p>0$ which is independent of the coefficients of the polynomial
        mapping $\calP$.
\end{theorem}

Suppose that $K\in\mathcal C^1\big(\RR^k\setminus\{0\}\big)$ is a
Calder\'on--Zygmund kernel satisfying  the differential inequality
\begin{equation}
	\label{eq:45}
	\norm{y}^k \abs{K(y)} + \norm{y}^{k+1} \norm{\nabla K(y)} \leq 1
\end{equation}
for all $y \in \RR^k\setminus\{0\}$  and the cancellation
condition
\[
	\int_{G_t \setminus G_s}K(y){\: \rm d}y=0
\]
for every $t > s > 0$. We consider a truncated singular Radon transform defined by
\[
  \mathcal T_t^{\calP}f(x)=\int_{G_t^c}f\big(x-\mathcal P(y)\big)K(y){\: \rm d}y
\]
for $x\in\RR^{d_0}$ and $t>0$. The second main result is the following theorem.
\begin{theorem}
\label{thm:21}
	For every $p \in (1, \infty)$  and $r\in(2, \infty)$ there is $C_{p, r} > 0$ such that for all
	$f \in L^p\big(\RR^{d_0}\big)$
        \begin{align}
          \label{eq:27}
          	\big\lVert
	V_r\big( \mathcal T_t^\calP f: t>0\big)
	\big\rVert_{L^p}\le
	C_{p, r}\|f\|_{L^p}.
        \end{align}
	Moreover, the constant $C_{p, r}\le C_p\frac{r}{r-2}$ for some
        $C_p>0$ which is independent of the coefficients of the polynomial
        mapping $\calP$.
\end{theorem}
We immediately see that \eqref{eq:27}  remains true for
the operator
\[
  \tilde{\mathcal T}_t^{\calP}f(x)={\rm p.v.}\int_{G_t}f\big(x-\mathcal P(y)\big) K(y) {\: \rm d}y.
\]
We set
$$
N_0 = \max\big\{ \deg \calP_j : 1 \leq j \leq d_0\big\}.
$$
It is convenient to work with the set
$$
\Gamma =
\big\{
	\gamma \in \ZZ^k \setminus\{0\} : 0 \leq \gamma_j \leq N_0
	\text{ for each } j = 1, \ldots, k
\big\}
$$
with the lexicographic order. Then each $\calP_j$ can be expressed as
$$
\calP_j(x) = \sum_{\gamma \in \Gamma} c_j^\gamma x^\gamma
$$
for some $c_j^\gamma \in \RR$. Let us denote by $d$ the cardinality of the set $\Gamma$.
We identify $\RR^d$ with the space of all vectors whose coordinates are labeled by multi-indices
$\gamma \in \Gamma$. Let $A$ be a diagonal $d \times d$ matrix such that
$$
(A v)_\gamma = \abs{\gamma} v_\gamma.
$$
For $t > 0$ we set
$$
t^{A}=\exp(A\log t)
$$
i.e. $t^A x=(t^{|\gamma|}x_{\gamma}: \gamma\in \Gamma)$ for any $x\in\RR^d$. Next, we introduce the \emph{canonical}
polynomial mapping
$$
\calQ = \big(\seq{\calQ_\gamma}{\gamma \in \Gamma}\big) : \ZZ^k \rightarrow \ZZ^d
$$
where $\calQ_\gamma(x) = x^\gamma$ and $x^\gamma=x_1^{\gamma_1}\cdot\ldots\cdot x_k^{\gamma_k}$.
The coefficients $\big(\seq{c_j^\gamma}{\gamma \in \Gamma, j \in \{1, \ldots, d_0\}}\big)$ define
a linear transformation $L: \RR^d \rightarrow \RR^{d_0}$ such that $L\calQ = \calP$. Indeed, it is
enough to set
$$
(L v)_j = \sum_{\gamma \in \Gamma} c_j^\gamma v_\gamma
$$
for each $j \in \{1, \ldots, d_0\}$ and $v \in \RR^d$. Now, proceeding
as in Lemma \ref{lem:1} we can reduce the matters (see also \cite{deL}
or \cite[p. 515]{bigs}) to the canonical polynomial mapping. To
simplify the notation we will write $\mathcal M_t=\mathcal
M_t^{\calQ}$ and $\mathcal T_t=\mathcal T_t^{\calQ}$.

\subsection{Long variations} In this subsection we give a new proof of
L\'epingle's inequality.  Since we will appeal to the results
from \cite{jsw} we are going to follow their notation  and therefore we
will work  with a more general setup than it is necessary for our
further  purposes. 

We consider  a slightly  more general  dilation structure 
\[
t^A=\exp(A\log t)
\] 
for any $t>0$, where $A$ is a $d\times d$ matrix whose eigenvalues
have positive real parts. We say that any regular quasi-norm
$\rho:\RR^d\rightarrow [0, \infty)$ is homogeneous with respect to the
dilations $(t^A: t>0)$ if $\rho(t^Ax)=t\rho(x)$ for any $x\in\RR^d$
and $t>0$.  Here $\RR^d$ endowed with a quasi-norm $\rho$ and the
Lebesgue measure will be considered as a space of homogeneous type
with the quasi-metric induced by $\rho$.

In this setting let us recall Christ's construction of dyadic cubes
\cite{chr}. 
\begin{lemma}[\cite{chr}]
	\label{christ}
  There exists a collection of open sets
  $\{Q_{\alpha}^k: k\in\ZZ \text{ and } \alpha\in I_k\}$ and constants
  $D>1$, $\delta, \eta>0$ and $0<C_1, C_2$ such that
  \begin{itemize}
  \item[(i)] $\big|\RR^d\setminus\bigcup_{\alpha\in I_k}Q_{\alpha}^k\big|=0$
    for all $k\in\ZZ$;
\item[(ii)] if $l\le k$ then either $Q_{\beta}^l\subseteq
  Q_{\alpha}^k$ or $Q_{\beta}^l\cap
  Q_{\alpha}^k = \emptyset$;
\item[(iii)] for each $(l, \beta)$ and $l\le k$, there exists a unique
  $\alpha$ such that $Q_{\beta}^l\subseteq
  Q_{\alpha}^k$;
\item[(iv)] each $Q_{\alpha}^k$ contains some ball $B(z_{\alpha}^k,
  \delta D^k)$ and $\diam(Q_{\alpha}^k)\le C_1 D^k$;
\item[(v)] for each $(\alpha, k)$ and $t>0$ we have $|\{x\in Q_{\alpha}^k: {\rm dist}(x, \RR^d\setminus
  Q_{\alpha}^k)\le tD^k\}|\le C_2 t^{\eta}|Q_{\alpha}^k|$.
  \end{itemize}
\end{lemma}
Now two comments are in order. Firstly, each cube $Q_{\alpha}^k$
contains a ball and is contained in some ball, each with radius $\simeq
D^k$.  Secondly, the quasi-metric is
translation invariant  thus we see that for each $(\alpha, k)$ the measure of $Q_{\alpha}^k$ is 
$\simeq D^{{\rm tr}(A)k}$. In particular, there is $R > 0$ such that for all $k \in \ZZ$, $\alpha \in I_k$ and
$\beta \in I_{k+1}$
\begin{equation}
	\label{eq:173}
	|Q_\alpha^{k+1}| \leq R |Q_\beta^k|.
\end{equation}
The collection $\{Q_{\alpha}^k: k\in\ZZ \text{ and } \alpha\in I_k\}$ will be called the collection of
dyadic cubes in $\RR^d$ adapted to the dilation group $(t^A: t>0)$. In view of Lemma \ref{christ}, it
gives rise to an atomic filtration. Namely, for each $k\in\ZZ$ let
$\calF_k=\sigma(\{Q_{\alpha}^l: \alpha\in I_l \text{ and } l \geq -k\})$ be the $\sigma$-algebra
generated by the cubes at level at least $-k$. Then
\[
	\calF_k \subset \calF_{k+1}.
\]
For a localy integrable function $f$ we set
\[
\EE_k f(x) = \EE[f | \calF_k] (x)
=
\frac{1}{|Q_{\alpha}^k|} \int_{Q_{\alpha}^k}f(y){\: \rm d}y
\]
provided $Q_{\alpha}^k$ is the unique dyadic cube containing $x\in\RR^d$. Thanks to Lemma \ref{christ} (i),
it is true for almost all $x$. We define a martingale difference by
\[
	\DD_k f = \EE_k f - \EE_{k-1} f
\]
Finally, the maximal function and the sequare function are given by
\begin{align*}
  Mf=\sup_{k\in\ZZ}|\mathbb E_kf| \quad \text{ and } \quad
  Sf=\Big(\sum_{k\in\ZZ}|\mathbb D_k f|^2\Big)^{1/2},
\end{align*}
respectively.

The variational estimate for $\big(\EE_k f : k \in \ZZ\big)$ follow from estimates on
$\lambda$-jump function $J_{\lambda}$, see \cite{px, bou, jsw}. Recall, that for a sequence of
complex numbers $(a_j: j\in\ZZ)$ the function $J_{\lambda}(a_j: j\in\ZZ)$ is equal to the supremum
over all $J\in\NN$ for which there exists a sequence of integers $t_1<t_2<\ldots<t_J$ so that
\[
|a_{t_{j+1}}-a_{t_j}|>\lambda
\]
for every $j=1, 2,\ldots, N-1$. We immediately see that
$J_{\lambda}(a_j: j\in\ZZ)\le \lambda^{-r}V_r(a_j: j\in\ZZ)^r$.
\begin{theorem}[\cite{px, bou}]
  \label{thm:lep}
	For each $p \geq 1$ there exists $B_p > 0$ such that for all $f \in L^p\big(\RR^d\big)$
	and $\lambda > 0$, if $p\in(1, \infty)$
	\[
       	\big\| 
		\lambda \big(J_\lambda(\mathbb E_kf: k\in\ZZ)\big)^{1/2} 
		\big\|_{L^p}
		\leq B_p 
		\|f\|_{L^p},   
	\]
	and if $p=1$ then for any $t > 0$ we have
	\[
       	\big|\big\{x\in\RR^d: 
		\lambda \big(J_\lambda(\mathbb E_kf(x): k\in\ZZ)\big)^{1/2}  > t
		\big\}\big|
		\leq 
		B_1 t^{-1} \|f\|_{L^1}.
	\]
\end{theorem}
The next theorem is a new ingredient in proving L\'epingle's inequality and is inspired by \cite{hkt}.
\begin{theorem}
	For each $q \geq 2$ there is $C_q > 0$ such that for all $r > 2$ and
	$\lambda > 0$
        \begin{multline}
          \label{eq:177}
        	C_q \cdot \big|\big\{x\in\RR^d: V_r(\mathbb E_kf(x):
                k\in\ZZ) > \lambda \text{ and } Mf(x) < \lambda/2\big\}\big|\\
		\leq
		\big|\big\{x\in\RR^d: Sf(x) > \lambda \big\}\big|
		+ \lambda^{-q} (r-2)^{-q/2} \int_{\{S(f) \leq \lambda\}} Sf(x)^q {\: \rm d}x.  
        \end{multline}
\end{theorem}
\begin{proof}
	By homogeneity, it suffices to prove the result with $\lambda
        = 1$. Let $B = \{x\in\RR^d: Sf(x) > 1\}$,
	$B^* = \{x\in\RR^d: M \ind{B}(x) > 1/(2R) \}$ and $G =
        (B^*)^c$. By the maximal inequality, we have
	\[
		|B^*|=|\{x\in\RR^d: M \ind{B}(x) > 1/(2R)
                \}| 
		\lesssim \int_{\RR^d} \ind{B}(x)^2 {\: \rm d}x = 
		|\{x\in\RR^d: Sf(x) > 1\}|.
	\]
	Therefore, it is enough to show that
	\[
		\big|\big\{x\in G: V_r(\mathbb E_kf(x):
                k\in\ZZ) > 1 \text{ and } Mf(x) < 1/2\big\}\big|
		\lesssim
		(r-2)^{-q/2}
		\int_{B^c} Sf(x)^q {\: \rm d}x.
	\]
	We can pointwise dominate the variation (see \cite{bou})
	\[ 
		V_r(\mathbb E_kf: k\in\ZZ)^r \leq \sum_{l \in \ZZ} 2^{rl} J_{2^l}(\mathbb E_kf: k\in\ZZ).
	\]
	Since $Mf < 1/2$, the above sum runs over
        $l \leq 0$, which leads to the containment
        \begin{multline}
        	\label{eq:26}
		\big\{x\in G: V_r(\mathbb E_kf(x): k\in\ZZ) >
                1 \text{ and } Mf(x) < 1/2 \big\} \\
		\subseteq
		\Big\{x\in G: \sum_{l \leq 0} 2^{rl} J_{2^l}(\mathbb E_kf(x): k\in\ZZ) > 1 \Big\}.
	\end{multline}
	For each $n \in \ZZ$ we define $U_n =\big\{x\in\RR^d: \mathbb E_n\ind{B}(x) \leq 1/2\big\}$. Notice that,
	if $x \in G$ then $x \in U_n$ for all $n \in \ZZ$. Let
	\[
		g(x) = \sum_{n \in \ZZ} \mathbb D_nf(x) \cdot \ind{U_{n-1}}(x).
	\]
	We observe that $\mathbb E_ng(x) = \mathbb E_nf(x)$ for all $x \in G$ and $n \in \ZZ$. Indeed,
	$\mathbb D_nf \cdot \ind{U_{n-1}}$ is $\calF_n$-measurable and
	\[
        \mathbb E_m\big(\mathbb D_nf\cdot\ind{U_{n-1}}\big)
		= 0
	\]
	for every $m \leq n-1$. Thus, for $x \in G$ we have
	\[
		\mathbb E_mg(x)
		=
		\sum_{n \leq m} \mathbb D_nf(x) \cdot \ind{U_{n-1}}(x)
		=\mathbb E_m f(x).
	\]
	Hence, by \eqref{eq:26} and H\"older's inequality with
        $a=\frac{q}{2}$ and $a'=\frac{q}{q-2}$,  we obtain
\begin{align*}
		\big\{x\in G:& V_r(\mathbb E_kf(x): k\in\ZZ) >
                1 \text{ and } Mf(x) < 1/2 \big\} \\
		&\qquad \subseteq
		\big\{x\in\RR^d: \sum_{l \leq 0}2^{rl} J_{2^l}(\mathbb E_kg(x):
                k\in\ZZ) > 1 \big\}\\
&\qquad \qquad\subseteq
\Big\{x\in G:
\Big(\sum_{l\le0}2^{\frac{1}{2}l(r-2)\frac{q}{q-2}}\Big)^{\frac{q-2}{q}} 
\Big(\sum_{l\le0}2^{\frac{1}{2}l(r-2)\frac{q}{2}}\big(2^{l}J_{2^l}(\mathbb E_kg(x):
                k\in\ZZ)^{1/2}\big)^q\Big)^{\frac{2}{q}}>1\Big\}.
	\end{align*}
Define 
\[
A_{q,
  r}=\Big(\sum_{l\le0}2^{\frac{1}{2}l(r-2)\frac{q}{q-2}}\Big)^{\frac{q-2}{q}}
=\mathcal O\big((r-2)^{-\frac{q-2}{q}}\big).
\]
Now	 Theorem \ref{thm:lep} immediately
	leads to the majorization
        \begin{multline*}
          \big|\big\{x\in G: V_r(\mathbb E_kf(x): k\in\ZZ) >
               1 \text{ and } Mf(x) < 1/2 \big\}\big|\\
\le \Big|\Big\{x\in G:
\sum_{l\le0}2^{\frac{1}{2}l(r-2)\frac{q}{2}}\big(2^{l}J_{2^l}(\mathbb E_kg(x):
                k\in\ZZ)^{1/2}\big)^q>A_{q, r}^{-\frac{q}{2}}\Big\}
\Big|\\
	 	\lesssim
                A_{q, r}^{\frac{q}{2}}
		\sum_{l\le0}2^{\frac{1}{2}l(r-2)\frac{q}{2}} \int_{\RR^d} |g(x)|^q {\: \rm d}x 
		 \lesssim (r-2)^{1-q/2-1} \int_{\RR^d} \abs{g(x)}^q
                 {\: \rm d}x\lesssim (r-2)^{-q/2} \int_{\RR^d} \abs{g(x)}^q {\: \rm d}x.
        \end{multline*} 
	Next, the square function $S$ is bounded from below on $L^q\big(\RR^d\big)$, therefore
	\[
		\int_{\RR^d}
		\abs{g(x)}^q {\: \rm d}x
		\lesssim
		\int_{\RR^d}
		Sg(x)^q
		{\: \rm d}x
		=
		\int_{\RR^d}
		\Big(\sum_{k \in \ZZ}
		\abs{\mathbb D_k f(x)}^2 
		\cdot
		\ind{U_{k-1}}(x)
		\Big)^{q/2}  {\: \rm d}x.
	\]
	Since for $x \in U_{k-1}$, we have $\mathbb E_{k-1}\ind{B}(x) \leq 1/(2R)$, by the doubling
	property \eqref{eq:173} we get $\EE_k\ind{B}(x) \leq 1/2$. Hence,
	\[
		\ind{U_{k-1}}(x) \leq 2 \cdot \EE_k\ind{B^c}(x),
	\]
	and
	\[
		\int_{\RR^d} \abs{g(x)}^q {\: \rm d}x
		\lesssim
		\int_{\RR^d} 
		\Big(\sum_{k \in \ZZ}
		\abs{\mathbb D_kf(x)}^2
		\cdot
		\mathbb E_k\ind{B^c}(x)
		\Big)^{q/2}
		{\: \rm d}x.
	\]  
	We observe that for $q = 2$ we have 
	\[
		\int_{\RR^d}\sum_{k \in \ZZ}
		\abs{\mathbb D_kf(x)}^2
		\cdot
		\mathbb E_k\ind{B^c}(x)
 		{\: \rm d}x
		=
		\int_{B^c} Sf(x)^2 {\: \rm d}x.
	\]
	For $q > 2$ let $\tilde{q} = q/2 > 1$ and $\tilde{q}'$ be its dual exponent. Then for
	$h \in L^{\tilde{q}'}\big(\RR^d\big)$ 
	\begin{align*}
		\int_{\RR^d}\Big(\sum_{k \in \ZZ}
		\abs{\mathbb D_k f(x)}^2
		\cdot
		\mathbb E_k \ind{B^c}(x)\Big)h(x)
 		{\: \rm d}x
		& =
		\sum_{k \in \ZZ}\int_{\RR^d}
		\abs{\mathbb D_k f(x)}^2
		\cdot
		\mathbb E_k \ind{B^c}(x)h(x)
 		{\: \rm d}x \\
		& =
		\sum_{k \in \ZZ}
		\int_{B^c}
		\abs{\mathbb D_k f(x)}^2
		\cdot
		\mathbb E_k h(x) {\: \rm d}x.
	\end{align*}
	Therefore, by H\"older's inequality, we get
	\[
		\int_{\RR^d}\Big(\sum_{k \in \ZZ}
		\abs{\mathbb D_kf(x)}^2
		\cdot
		\mathbb E_k\ind{B^c}(x)\Big)h(x)
 		{\: \rm d}x
		\leq
		\Big( \int_{B^c} Sf(x)^q {\: \rm d} x \Big)^{2/q} 
		\lVert M h \rVert_{\tilde{q}'}.
	\]
	Taking the supremum over all $h \in L^{\tilde{q}'}\big(\RR^d\big)$ we conclude
	\[
		\int_{\RR^d} 
		\Big(\sum_{k \in \ZZ}
		\abs{\mathbb D_kf(x)}^2
		\cdot
		\mathbb E_k\ind{B^c}(x)
		\Big)^{q/2}
		{\: \rm d}x\lesssim
		\int_{B^c} Sf(x)^q {\: \rm d}x,
	\]
	which finishes the proof.
\end{proof}
Now, using Theorem \ref{thm:20} we prove L\'epingle's inequality for
the sequence $(\mathbb E_kf: k\in\ZZ)$.
\begin{theorem}[\cite{le}]
  \label{thm:22}
  For each $p \in (1, \infty)$ there exists $C_p > 0$ such that for
  all $r\in(2, \infty)$ and 
  $f \in L^p\big(\RR^d\big)$ we have
	\[
       	\big\| 
		V_r(\mathbb E_kf: k\in\ZZ)
		\big\|_{L^p}
		\leq C_p\frac{r}{r-2} 
		\|f\|_{L^p},   
	\]
Moreover, $V_r(\mathbb E_kf: k\in\ZZ)$ is also weak type $(1, 1)$. 
\end{theorem}
\begin{proof}
	Given $p > 1$ we take $q=2p>2$. By \eqref{eq:177}, for $f\in L^p\big(\RR^d\big)\cap L^q\big(\RR^d\big)$ 
	we have
	\begin{align*}
		C_q \cdot
		\big|\big\{
		x \in \RR^d :
		V_r(\EE_k f : k \in \ZZ) > \lambda
		\big\}\big|
		& \leq
		\big|\big\{
		x \in \RR^d : S f(x) > \lambda
		\big\}\big|
		+
		\big|\big\{
		x \in \RR^d : M f(x) > \lambda/2
		\big\}\big| \\
		&\quad +
		\lambda^{-q}
		(r-2)^{-q/2}
		\int_{\{Sf \le \lambda\}}
		S f(x)^q {\: \rm d} x.
	\end{align*}
	Therefore, we get 
  \begin{multline*}
    \big\|V_r(\EE_kf: k\in\ZZ)\big\|_{L^p}^p
	=
	p\int_{0}^{\infty}\lambda^{p-1}\big|\big\{x\in\RR^d: V_r(\EE_kf(x): k\in\ZZ)>\lambda\big\}\big|{\: \rm d}\lambda\\
	\lesssim 
	\vnorm{Sf}_{L^p}^p + \vnorm{Mf}_{L^p}^p 
	+ (r-2)^{-p}\int_{0}^{\infty}\lambda^{p-q-1}\int_{\{S(f) \leq \lambda\}} Sf(x)^q {\rm d}x{\: \rm d}\lambda\\
	\lesssim 
	\|f\|_{L^p}^p+(r-2)^{-p} \int Sf(x)^q\int_{Sf(x)}^{\infty}\lambda^{p-q-1} {\rm d}\lambda{\: \rm d}x\\
	\lesssim \|f\|_{L^p}^p+(r-2)^{-p}\|Sf\|_{L^p}^p\lesssim \big(1+(r-2)^{-p}\big)\|f\|_{L^p}^p.
  \end{multline*}
For $p=1$ it suffices to apply the Calder\'on--Zygmund decomposition
and the desired claim follows. 
\end{proof}

Long variational bounds for $\calM_t$ follows the same line as in
\cite{jsw}. For every $f\in L^p\big(\RR^d\big)$ with $p\in(1, \infty)$
we obtain
\begin{align}
\label{eq:34}
  \big\lVert
	V_r\big( \mathcal M_{2^n} f: n\in\ZZ\big)
	\big\rVert_{L^p}\le
  \big\lVert
	V_r\big( \mathbb E_nf: n\in\ZZ\big)
	\big\rVert_{L^p}
  +\bigg\|\Big(\sum_{n\in\NN}\big|\mathcal M_{2^n} f- \mathbb
  E_nf\big|^2\Big)^{1/2}\bigg\|_{L^p}.
\end{align}
The first term in \eqref{eq:34} is bounded by Theorem \ref{thm:22},
whereas the square function can be estimated as in \cite[Proof of
Theorem 1.1]{jsw}.
In the next theorem
we consider long variational estimates for $\mathcal T_{t}$.
\begin{theorem}
\label{thm:23}
	For every $p \in (1, \infty)$ there is $C_p > 0$ such that for all $r\in(2, \infty)$ and $f \in L^p\big(\RR^{d}\big)$
	\[
	\big\lVert
	V_r\big( \mathcal T_{2^n} f: n\in\ZZ\big)
	\big\rVert_{L^p}\le
	C_p \frac{r}{r-2} \vnorm{f}_{L^p}.
	\]
\end{theorem}
\begin{proof}
Let $\Phi\in\mathcal C^{\infty}\big(\RR^d\big)$ with compact support and integral one. 
Then we have the following decomposition (see \cite{DuoRdF})
\begin{align}
\label{eq:187}
\nonumber  \mathcal T_{2^n}f&=\Phi_{2^n}*\big(\mathcal Tf-\sum_{j<n}\mu_{2^j}*f\big)+(\delta_0-\Phi_{2^n})*\sum_{j\ge
  n}\mu_{2^j}*f\\
&=\Phi_{2^n}*\mathcal Tf-\big(\Phi_{2^n}*\sum_{j< n}\mu_{2^j}\big)*f+\sum_{j\ge
0}(\delta_0-\Phi_{2^n})*\mu_{2^{j+n}}*f,
\end{align}
where 
\[
\mu_{2^j}*f(x)=\int_{G_{2^{j+1}}\setminus
  G_{2^j}}f\big(x-\calQ(y)\big)K(y){\: \rm d}y.
\]
The variational estimates for the first term in \eqref{eq:187} follows by \cite[Theorem 1.1, Lemma 2.1]{jsw}.
For the second term we apply the Littlewood--Paley theory since for each $n \in \NN$
\[
	\Phi_{2^n} * \sum_{j < n} \mu_{2^j}
\]
is a Schwartz function with integral zero. Thus, by \eqref{eq:11},
\[
	V_r\big(\Phi_{2^n} * \sum_{j < n} \mu_{2^j}*f : n \in \ZZ\big)
	\lesssim
	\Big(
	\sum_{n \in \ZZ} 
	\big\lvert \Phi_{2^n} * \sum_{j < n} \mu_{2^j} * f
	\big\rvert ^2\Big)^{1/2}.
\]
For the last term we have
\[
	\big\|V_r\big(\sum_{j\ge0}(\delta_0-\Phi_{2^n})*\mu_{2^{j+n}}*f: n\in\ZZ\big)\big\|_{L^p}
	\le
	\sum_{j\ge0}
	\big\|\big(\sum_{n\in\ZZ}\big|(\delta_0-\Phi_{2^n})*\mu_{2^{j+n}}*f\big|^2\big)^{1/2}\big\|_{L^p}
\]
and due to the Littlewood--Paley theory one can show that there is $C_p>0$
and $\delta_p>0$ such that
\[
	\big\|\big(\sum_{n\in\ZZ}\big|(\delta_0-\Phi_{2^n})*\mu_{2^{j+n}}*f\big|^2\big)^{1/2}\big\|_{L^p}
	\le
	C_p2^{-\delta_p j}\|f\|_{L^p}.
\]
This completes the proof of Theorem \ref{thm:23}.
\end{proof}

\subsection{Short variations for averaging operators}
To deal with short variations we need a counterpart of Lemma \ref{lem:8}.
\begin{lemma}
\label{lem:10}
Let $u < v$ be real numbers and $a:[u, v]\rightarrow \CC$ be a
differentiable function. For any $h \in\NN$ and the sequence
$\big(\seq{s_j}{0 \leq j \leq h}\big)$ with $s_j=u+h^{-1}(v-u)j$ we
have for every $r\in[1, \infty)$
        \begin{align}
          \label{eq:33}
		V_r\big(a(t): t\in[u, v)\big)
		\lesssim
		 \Big(\sum_{j=0}^h|a(s_j)|^r\Big)^{1/r}
		+\Big(\sum_{j=0}^{h-1}\Big(\int_{s_j}^{s_{j+1}}|a'(t)|{\: \rm
                d}t\Big)^r\Big)^{1/r}.
        \end{align}
Moreover, if $p\ge r$ then
\begin{multline*}
\Big(\sum_{j=0}^h|a(s_j)|^r\Big)^{1/r}
          +\Big(\sum_{j=0}^{h-1}\Big(\int_{s_j}^{s_{j+1}}|a'(t)|{\: \rm
            d}t\Big)^r\Big)^{1/r}\\
          \lesssim h^{1/r-1/p}\Big(\sum_{j=0}^h|a(s_j)|^p\Big)^{1/p}
          +h^{1/r-1}(v-u)^{1-1/p}\Big(\int_{u}^{v}|a'(t)|^p{\: \rm
                d}t\Big)^{1/p}.
\end{multline*}
\end{lemma}
\begin{proof}
  Fix $h \in \NN$ and consider the sequence $\big(\seq{s_j}{0 \leq j
    \leq h}\big)$ such that $s_j=u+h^{-1}(v-u)j$. Then
	\begin{align*}
		V_r\big(a(t): t\in[u, v)\big)
		& \lesssim \Big(\sum_{j=0}^h|a(s_j)|^r\Big)^{1/r}
		+\Big(\sum_{j=0}^{h-1}V_1\big(a(t): t\in[s_j, s_{j+1})\big)^r\Big)^{1/r}\\
		& \lesssim
		 \Big(\sum_{j=0}^h|a(s_j)|^r\Big)^{1/r}
		+\Big(\sum_{j=0}^{h-1}\Big(\int_{s_j}^{s_{j+1}}|a'(t)|{\: \rm
                d}t\Big)^r\Big)^{1/r}.
	\end{align*}
		If $p \geq r$, by H\"older's inequality, we get
		\begin{align*}
			\Big(\sum_{j=0}^h|a(s_j)|^r\Big)^{1/r}
			+\Big(\sum_{j=0}^{h-1}\Big(\int_{s_j}^{s_{j+1}}|a'(t)|{\: \rm
                d}t\Big)^r\Big)^{1/r}
			& \leq
			h^{1/r - 1/p} 
			\Big(\sum_{j=0}^h|a(s_j)|^p\Big)^{1/p} \\
			& \quad +h^{1/r-1/p}\Big(\sum_{j=0}^{h-1}\Big(\int_{s_j}^{s_{j+1}}|a'(t)|{\: \rm
            d}t\Big)^p\Big)^{1/p}.
		\end{align*}
		For the second term we again use H\"{o}lder's inequality to obtain
		\begin{align*}
          h^{1/r-1/p}\Big(\sum_{j=0}^{h-1}\Big(\int_{s_j}^{s_{j+1}}|a'(t)|{\: \rm
            d}t\Big)^p\Big)^{1/p}
          & \lesssim 
		h^{1/r-1/p}\Big(\sum_{j=0}^{h-1}(s_{j+1}-s_j)^{p(1-1/p)}\Big(\int_{s_j}^{s_{j+1}}|a'(t)|^p{\: \rm
                d}t\Big)\Big)^{1/p}\\
          & \lesssim 
          h^{1/r-1}(v-u)^{1-1/p}\Big(\int_{u}^{v}|a'(t)|^p{\: \rm
                d}t\Big)^{1/p}
        \end{align*}
		where in the last estimate we have used $s_{j+1}-s_j=(v-u)/h$. 
		This completes the proof of the lemma.
\end{proof}

Now the task is to prove that for every $p\in(1, \infty)$ there are
$C_p>0$ such that for every $f\in L^p\big(\RR^d\big)$ we have
\begin{align}
\label{eq:170}
  \Big\lVert\Big(\sum_{n\in\ZZ}
	V_2\big( \big(\mathcal M_t-\mathcal M_{2^n}\big)f: t\in[2^n, 2^{n+1})\big)^2\Big)^{1/2}
	\Big\rVert_{L^p}\le C_p\|f\|_{L^p}.
\end{align}
We may assume that $f$ is a Schwartz function.  Let $S_j$ be a Littlewood--Paley projection
$\calF(S_j g)(\xi)=\phi_j(\xi)\calF g (\xi)$ associated with $\big(\phi_j: j\in\ZZ\big)$
a smooth partition of unity of $\RR^d\setminus\{0\}$ such that for each $j \in \ZZ$ we have
$0\le \phi_j\le 1$ and
\[
	\supp \phi_j
	\subseteq
	\big\{\xi\in\RR^d: 2^{-j-1} < \abs{\xi} < 2^{-j+1} \big\}
\]
and for $\xi\in \RR^d\setminus\{0\}$
\[
\sum_{j\in\ZZ}\phi_j(\xi)=1.
\]
We are going to prove that for every $p\in(1, \infty)$ there are $C_p>0$
and $\delta_p>0$ such that for every $j\in\ZZ$ we have 
\begin{align}
  \label{eq:36}
  \Big\lVert\Big(\sum_{n\in\ZZ}
	V_2\big( \big(\mathcal M_t-\mathcal M_{2^n}\big) S_{j+n}f: t\in[2^n, 2^{n+1})\big)^2\Big)^{1/2}
	\Big\rVert_{L^p}\le C_p2^{-\delta_p|j|}\|f\|_{L^p}.
\end{align}
Applying \eqref{eq:33} with $h = 2^{\varepsilon \abs{j}}$, we obtain that
\begin{multline*}
    \Big\lVert\Big(\sum_{n\in\ZZ}
	V_2\big( \big(\mathcal M_t-\mathcal M_{2^n}\big) S_{j+n}f: t\in[2^n, 2^{n+1})\big)^2\Big)^{1/2}
	\Big\rVert_{L^p} \\
\lesssim 
  \Big\lVert\Big(\sum_{n\in\ZZ}\sum_{l=0}^h
	 \big|\big(\mathcal M_{s_l}-\mathcal M_{2^n}\big) S_{j+n}f\big|^2\Big)^{1/2}
	\Big\rVert_{L^p} \\ 
+
\Big\lVert\Big(\sum_{n\in\ZZ}\sum_{l=0}^{h-1}\Big(\int_{s_l}^{s_{l+1}}
	 \big|\frac{{\rm d}}{{\: \rm d}t}\mathcal M_{t} S_{j+n}f\big|{\: \rm d}t\Big)^2\Big)^{1/2}
	\Big\rVert_{L^p}=I_p^1+I_p^2.
\end{multline*}
\subsubsection{The estimates for $I_p^1$}
First, using vector-valued estimates from \cite[Theorem A.1]{mst1} together with the Littlewood--Paley
theory we get
\begin{equation}
  \label{eq:42}
	\begin{aligned}
  I_p^1 &\lesssim 
	2^{\varepsilon|j|/2}
	\Big\lVert\Big(\sum_{n\in\ZZ}
	 \sup_{t>0}\big|\mathcal M_{t}S_{j+n}f\big|^2\Big)^{1/2}
	\Big\rVert_{L^p} \\
	& \lesssim 2^{\varepsilon|j|/2}
    \Big\|\Big(\sum_{n\in\NN}|S_{j+n}f|^2\Big)^{1/2}\Big\|_{L^p}
\lesssim 2^{\varepsilon|j|/2}\|f\|_{L^p}.
\end{aligned}
\end{equation}
Next, we are going to refine the estimate \eqref{eq:42} for $p=2$. Let $m_t$ be the multiplier
associated with the operator $\mathcal M_t$. By van der Corput's lemma
\cite{sw},  for each
$t\in[2^n, 2^{n+1})$ we have
\[
|m_t(\xi)-m_{2^n}(\xi)|\lesssim \min\big\{1, |2^{nA}\xi|_\infty, |2^{nA}\xi|_\infty^{-1/d}\big\}.
\]
Therefore, by Plancherel's theorem 
\begin{multline}
\label{eq:43}
  I_2^1
=\Big(\sum_{l=0}^h\int_{\RR^d}\sum_{n\in\ZZ}\big|(m_{s_l}(\xi)-m_{2^n}(\xi))\phi_{j+n}(\xi)\mathcal
        Ff(\xi)\big|^2{\: \rm d}\xi\Big)^{1/2}\\
\lesssim 2^{-|j|/d+\varepsilon|j|/2}\bigg(\int_{\RR^d}\sum_{n\in\ZZ}\big|\phi_{j+n}(\xi)\mathcal
        Ff(\xi)\big|^2{\rm d}\xi\bigg)^{1/2}\lesssim 2^{-\abs{j}/d + \varepsilon \abs{j} /2}\|f\|_{L^2}.
\end{multline}
Interpolating \eqref{eq:42} with \eqref{eq:43} and choosing appropriate $\varepsilon < 2/d$ we get
\begin{align*}
  I_p^1\lesssim 2^{-\delta_p|j|}\|f\|_{L^p}.
\end{align*}

\subsubsection{The estimates for $I_p^2$}
Since $G$ is an open bounded convex set containing the origin, with a help of the spherical coordinates we may write
\[
	\calM_t g(x) =
	\frac{1}{t^k \abs{G}}
	\int_{S^{k-1}}
	\int_0^{r(\omega) t}
	g\big(x - \calQ(r \omega) \big) r^{k-1}
	{\: \rm d} r
	{\: \rm d}\sigma(\omega)
\]
where $S^{k-1}$ is a unit sphere in $\RR^k$ and $\sigma$ is the surface measure on $S^{k-1}$. We observe that
if $g$ is a Schwartz function 
\begin{align}
	\label{eq:174}
	\frac{{\rm d}}{{\rm d} t} \calM_t g(x)
	& =
	-k \frac{1}{t^{k+1} \abs{G}}
	\int_{S^{k-1}}
	\int_0^{r(\omega) t}
	g\big(x - \calQ(r \omega)\big) r^{k-1} {\: \rm d}r {\: \rm d}\sigma(\omega) \\
	\nonumber
	& \quad +
	\frac{1}{t^k \abs{G}}
	\int_{S^{k-1}}
	g\big(x - \calQ(r(\omega) t \omega)\big) r(\omega)^k t^{k-1} {\: \rm d}\sigma(\omega).
\end{align}
Change of the order of integration and differentiation is permited since $g$ is bounded. Hence, if $t \in [s_l, s_{l+1})$
and $s_l, s_{l+1} \simeq 2^n$, by Tonnelli's theorem, we get
\begin{align*}
	\sum_{l=0}^{h-1}\int_{s_l}^{s_{l+1}}
	\Big|
	\frac{{\rm d}}{{\rm d} t} \calM_t g(x)
	\Big|
	{\: \rm d} t
	& \lesssim
	\calM_{2^{n+1}} \abs{g}(x)
	+
	\frac{1}{2^{nk} |G|}
	\int_{2^n}^{2^{n+1}} \int_{S^{k-1}} \big|g\big(x - \calQ(r(\omega) t \omega)\big)\big| r(\omega)^k t^{k-1} 
	{\: \rm d}\sigma(\omega)
	{\: \rm d}t \\
	& \lesssim
	\calM_{2^{n+1}} \abs{g} (x).
\end{align*}
Therefore, we obtain
\begin{equation}
  \label{eq:54}
	\begin{aligned}
  I_p^2& \lesssim
	\Big\lVert\Big(\sum_{n\in\ZZ}\big(
	 \calM_{2^{n+1}} \abs{S_{j+n}f} \big)^2\Big)^{1/2}
	\Big\rVert_{L^p} 
	& \lesssim
	\Big\lVert\Big(\sum_{n \in \ZZ}
	\sup_{t > 0} \big(\calM_t \abs{S_{j + n} f} \big)^2
	\Big)^{1/2}
	\Big\rVert_{L^p}
	\lesssim
	\vnorm{f}_{L^p},
\end{aligned}
\end{equation}
where the last inequality follows by the same line of reasoning as \eqref{eq:42}. 

Next, we refine the estimates of $I_p^2$ for $p = 2$.  Let
$\tilde{m}_{t}$ be the multiplier associated with the operator
$\frac{{\rm d}}{{\rm d}t}\mathcal M_{t}$.  We have
\begin{equation}
	\label{eq:175}
	\tilde{m}_t(\xi) = - \frac{k}{t^{k+1} \abs{G}} \int_{G_t} e^{2\pi i \sprod{\xi}{\calQ(x)}} {\: \rm d}x
	+
	\frac{1}{t^k \abs{G}} \int_{S^{k-1}} e^{2\pi i \sprod{\xi}{\calQ(r(\omega) t \omega )}} r(\omega)^k t^{k-1}
	{\: \rm d}\sigma(\omega).
\end{equation}
Indeed, by \eqref{eq:174}, we have
\begin{align*}
	\calF\Big(\frac{{\rm d}}{{\rm d} t} \calM_t g\Big)(\xi)
	& =
	-\frac{k}{t^{k+1} \abs{G}} \int_{G_t} e^{2\pi i \sprod{\xi}{\calQ(x)}} {\: \rm d} x 
	\calF g(\xi) \\
	& \quad +
	\frac{1}{t^k \abs{G}} 
	\int_{\RR^d} 
	e^{2\pi i \sprod{\xi}{x}}
	\int_{S^{k-1}} g\big(x - \calQ(r(\omega) t \omega)\big) r(\omega)^k t^{k-1} {\: \rm d}\sigma(\omega)
	{\: \rm d}x.
\end{align*}
Again, for the second term we need to justify the change of integrations. Let
\[
	R = \sup \big\{\norm{\calQ(y)} : y \in G_t \big\}.
\]
Then for all $\norm{x} \geq 2 R$ and $|y|\le R$ we have
\[
	\norm{x - y} \geq \frac{\norm{x}}{2},
\]
thus
\begin{equation}
	\label{eq:47}
	\big| g\big(x - \calQ(y)\big) \big|
	\lesssim
	\big(1 + \norm{x}\big)^{-2d}.
\end{equation}
By Fubini's theorem we get the claim. Moreover, we see that for
$t\simeq 2^n$ we obtain $|\tilde m_{t}(\xi)|\lesssim 2^{-n}$. 

Now, using \eqref{eq:175}, by the Cauchy--Schwarz inequality and Plancherel's theorem we get
\begin{multline}
  \label{eq:44}
	I_2^2
	\le
	\Big\lVert\Big(\sum_{n\in\ZZ}\frac{2^n}{h}\int_{2^n}^{2^{n+1}}
	\Big|\frac{{\rm d}}{{\rm d}t}\mathcal M_{t} S_{j+n}f\Big|^2{\: \rm d}t
	\Big)^{1/2}
	\Big\rVert_{L^2} \\
	=
	\Big(\sum_{n\in\ZZ}\frac{2^n}{h}\int_{2^n}^{2^{n+1}}\int_{\RR^d}
	\big|\tilde{m}_{t}(\xi) \phi_{j+n}(\xi)\mathcal Ff(\xi)\big|^2{\: \rm d}\xi {\: \rm d} t\Big)^{1/2}\\
	\lesssim 
	2^{-\varepsilon |j|/2}
	\Big(\sum_{n\in\ZZ}\int_{\RR^d}
	\big|\phi_{j+n}(\xi)\mathcal Ff(\xi)\big|^2{\: \rm d}\xi\Big)^{1/2}
	\lesssim 2^{-\varepsilon |j|/2}\|f\|_{L^2}
\end{multline}
for $0<\varepsilon<1/d$. Thus interpolation of \eqref{eq:54} with \eqref{eq:44} gives
\begin{align*}
  I_p^2\lesssim 2^{-\delta_p|j|}\|f\|_{L^p}
\end{align*}
and the proof of \eqref{eq:36} is completed.

\subsection{Short variations for truncated singular integral operators}
We are going to show that for any $p\in(1, \infty)$ there are $C_p>0$ and $\delta_0 > 0$
such that for every $j \in \ZZ$ and $f\in L^p\big(\RR^d\big)$ we have
\begin{align}
\label{eq:12}
  \Big\lVert
	\Big(
	\sum_{n\in\ZZ}
	V_2\big( \big(\mathcal T_t-\mathcal T_{2^n}\big) S_{j +n} f: t\in[2^n, 2^{n+1})\big)^2
	\Big)^{1/2}
	\Big\rVert_{L^p}
	\le 
	C_p
	2^{-\delta_p \abs{j}}
	\|f\|_{L^p}.
\end{align}
We may assume that $f$ is a Schwartz function. The proof of
\eqref{eq:12} follows the same line as the one for the averaging
operator. By Lemma \ref{lem:10}, for $h = 2^{\varepsilon \abs{j}}$, we
obtain
\begin{multline*}
    \Big\lVert\Big(\sum_{n\in\ZZ}
	V_2\big( \big(\mathcal T_t-\mathcal T_{2^n}\big) S_{j+n}f: t\in[2^n, 2^{n+1})\big)^2\Big)^{1/2}
	\Big\rVert_{L^p} 
	\lesssim
	\Big\lVert\Big(\sum_{n\in\ZZ}\sum_{l=0}^h
	 \big|\big(\mathcal T_{s_l}-\mathcal T_{2^n}\big) S_{j+n}f\big|^2\Big)^{1/2}
	\Big\rVert_{L^p} \\ 
+
\Big\lVert\Big(\sum_{n\in\ZZ}\sum_{l=0}^{h-1}\Big(\int_{s_l}^{s_{l+1}}
	 \Big|\frac{{\rm d}}{{\: \rm d}t}\big(\mathcal T_{t}-\mathcal T_{2^n}\big) S_{j+n}f\Big|{\: \rm d}t\Big)^2\Big)^{1/2}
	\Big\rVert_{L^p}
	=I_p^1+I_p^2.
\end{multline*}
\subsubsection{The estimates for $I_p^1$}
By the vector-valued estimates from \cite[Theorem A.1]{mst1} and the Littlewood--Paley theory we get
\begin{equation}
  \label{eq:15}
	\begin{aligned}
  I_p^1 &\lesssim 
	2^{\varepsilon|j|/2}
	\Big\lVert\Big(\sum_{n\in\ZZ}
	 \sup_{t>0}\big(\mathcal M_{t}|S_{j+n}f|\big)^2\Big)^{1/2}
	\Big\rVert_{L^p} \\
	& \lesssim 2^{\varepsilon|j|/2}
    \Big\|\Big(\sum_{n\in\NN}|S_{j+n}f|^2\Big)^{1/2}\Big\|_{L^p}
\lesssim 2^{\varepsilon|j|/2}\|f\|_{L^p}.
\end{aligned}
\end{equation}
For $p = 2$ we get better estimate. Let $m_{2^n, t}$ be the multiplier associated with the operator
$\mathcal T_t - \mathcal T_{2^n}$ for $t \in [2^n, 2^{n+1})$. By van der Corput's lemma \cite{sw} (or more precisely the method of proof of van der Corput
lemma from \cite{sw}) we have
\[
	|m_{t, 2^n}(\xi)|
	\lesssim 
	\min\big\{1, |2^{nA}\xi|_\infty, |2^{nA}\xi|_\infty^{-1/d}\big\}.
\]
Therefore, by Plancherel's theorem 
\begin{multline}
\label{eq:18}
  I_2^1
=\Big(\sum_{l=0}^h\int_{\RR^d}\sum_{n\in\ZZ}\big|(m_{s_l, 2^n}(\xi)\phi_{j+n}(\xi)\mathcal
        Ff(\xi)\big|^2{\: \rm d}\xi\Big)^{1/2}\\
\lesssim 2^{-|j|/d+\varepsilon|j|/2}\bigg(\int_{\RR^d}\sum_{n\in\ZZ}\big|\phi_{j+n}(\xi)\mathcal
        Ff(\xi)\big|^2{\rm d}\xi\bigg)^{1/2}\lesssim 2^{-\abs{j}/d + \varepsilon \abs{j} /2}\|f\|_{L^2}.
\end{multline}
Interpolating \eqref{eq:15} with \eqref{eq:18} and choosing appropriate $\varepsilon < 2/d$ we get
\begin{align*}
  I_p^1\lesssim 2^{-\delta_p|j|}\|f\|_{L^p}.
\end{align*}

\subsubsection{The estimates for $I_p^2$}
Since $G$ is an open bounded convex set containing the origin, with a help of the spherical coordinates we may write
\[
	\big(\calT_t - \calT_{2^n}\big) g(x) =
	\int_{S^{k-1}}
	\int_{r(\omega) 2^n}^{r(\omega) t}
	g\big(x - \calQ(r \omega) \big)
	K(r \omega)
	r^{k-1}
	{\: \rm d} r
	{\: \rm d}\sigma(\omega)
\]
where $S^{k-1}$ is a unit sphere in $\RR^k$ and $\sigma$ is the surface measure on $S^{k-1}$. We observe that
if $g$ is a Schwartz function 
\begin{align}
	\label{eq:22}
	\frac{{\rm d}}{{\rm d} t} \big(\calT_t - \calT_{2^n} \big) g(x)
	& =
	t^{k-1}
	\int_{S^{k-1}}
	g\big(x - \calQ(r(\omega) t \omega)\big) K(r(\omega) t \omega) r(\omega)^k
	{\: \rm d}\sigma(\omega).
\end{align}
Change of the order of integration and differentiation is permited since $g$ is bounded and the kernel $K$ satifies
\eqref{eq:45}. Hence, if $t \in [s_l, s_{l+1})$ and $s_l, s_{l+1} \simeq 2^n$, by Tonnelli's theorem, we get
\begin{align*}
	\sum_{l=0}^{h-1}\int_{s_l}^{s_{l+1}}
	\Big|
	\frac{{\rm d}}{{\rm d} t} \big(\calT_t - \calT_{2^n}\big) g(x)
	\Big|
	{\: \rm d} t
	\lesssim
	\calM_{2^{n+1}} \abs{g}(x).
\end{align*}
Therefore, we obtain
\begin{equation}
  \label{eq:23}
	\begin{aligned}
	I_p^2& \lesssim
	\Big\lVert\Big(\sum_{n\in\ZZ}\big(
	 \calM_{2^{n+1}} \abs{S_{j+n}f} \big)^2\Big)^{1/2}
	\Big\rVert_{L^p} 
	& \lesssim
	\Big\lVert\Big(\sum_{n \in \ZZ}
	\sup_{t > 0} \big(\calM_t \abs{S_{j + n} f} \big)^2
	\Big)^{1/2}
	\Big\rVert_{L^p}
	\lesssim
	\vnorm{f}_{L^p}.
\end{aligned}
\end{equation}
Now we refine the estimate of $I_p^2$ for $p = 2$.
Let $\tilde{m}_{t, 2^n}$ be the multiplier associated with the operator
$\frac{{\rm d}}{{\rm d}t} \big(\calT_t - \calT_{2^n}\big)$. We have
\begin{equation}
	\label{eq:29}
	\tilde{m}_{t, 2^n}(\xi) 
	= 
	t^{k-1}
	\int_{S^{k-1}} e^{2\pi i \sprod{\xi}{\calQ(r(\omega) t \omega )}} r(\omega)^k
	K(r(\omega) t \omega)
	{\: \rm d}\sigma(\omega).
\end{equation}
Indeed, by \eqref{eq:22}, we have
\begin{align*}
	\calF\Big(\frac{{\rm d}}{{\rm d} t} \big(\calT_t - \calT_{2^n}\big) g\Big)(\xi)
	& =
	t^{k-1}
	\int_{\RR^d} 
	e^{2\pi i \sprod{\xi}{x}}
	\int_{S^{k-1}} g\big(x - \calQ(r(\omega) t \omega)\big) r(\omega)^k K(r(\omega) t \omega) {\: \rm d}\sigma(\omega)
	{\: \rm d}x
\end{align*}
and thanks to estimates \eqref{eq:45} and \eqref{eq:47} we may change
the order of integrations. Note that if $t\simeq 2^n$ we obtain
$|\tilde m_{t, 2^n}(\xi)|\lesssim 2^{-n}$.

Now, using \eqref{eq:29}, by the Cauchy--Schwarz inequality and Plancherel's theorem we get
\begin{multline}
  \label{eq:31}
	I_2^2
	\le
	\Big\lVert\Big(\sum_{n\in\ZZ}\frac{2^n}{h}\int_{2^n}^{2^{n+1}}
	\Big|\frac{{\rm d}}{{\rm d}t}\big(\calT_t - \calT_{2^n}\big) S_{j+n}f\Big|^2{\: \rm d}t
	\Big)^{1/2}
	\Big\rVert_{L^2} \\
	=
	\Big(\sum_{n\in\ZZ}\frac{2^n}{h}\int_{2^n}^{2^{n+1}}\int_{\RR^d}
	\big|\tilde{m}_{t,2^n}(\xi) \phi_{j+n}(\xi)\mathcal Ff(\xi)\big|^2{\: \rm d}\xi {\: \rm d} t \Big)^{1/2}\\
	\lesssim 
	2^{-\varepsilon |j|/2}
	\Big(\sum_{n\in\ZZ}\int_{\RR^d}
	\big|\phi_{j+n}(\xi)\mathcal Ff(\xi)\big|^2{\: \rm d}\xi\Big)^{1/2}
	\lesssim 2^{-\varepsilon |j|/2}\|f\|_{L^2}
\end{multline}
for $0<\varepsilon<1/d$. Thus interpolation of \eqref{eq:23} with \eqref{eq:31} gives
\begin{align*}
  I_p^2\lesssim 2^{-\delta_p|j|}\|f\|_{L^p}
\end{align*}
and the proof of \eqref{eq:12} is completed.

\begin{bibliography}{discrete}
	\bibliographystyle{amsplain}
\end{bibliography}

\end{document}